\documentclass[10pt]{article}
\usepackage[dvipdfmx]{graphicx,color}
\usepackage{color}
\usepackage{amsmath}
\usepackage{amscd}
\usepackage{amssymb}
\usepackage{amsthm}       
\newtheorem{thm}{Theorem}[section]
\newtheorem{prop}[thm]{Proposition}
\newtheorem{lem}[thm]{Lemma}
  \theoremstyle{definition}
\newtheorem{df}[thm]{Definition}   \theoremstyle{definition}

\newtheorem{rem}[thm]{Remark}                \theoremstyle{plain}
 \theoremstyle{definition}
\newtheorem{ex}[thm]{Example}   \def\CC{\Bbb{C}}
\def\RR{\Bbb{R}}  
        \def\NN{\Bbb{N}} 
\def\ZZ{\Bbb{Z}}
\def\B1{{\rm\kern.32em\vrule    width.12em       height1.4ex
depth-.05ex\kern-.28em 1}}

\def\g{\gamma }

\def\CMX{\text{CM}(Y)}

\def\OCMX{\text{OCM}(Y)}

\def\Cpt{\text{Cpt}}

\def\Hol{\text{H\"{o}l}}

\def\Min{\text{Min}}
\def\emMin{\text{\em Min}}

\def\emLSf{\text{{\em LS}}({\cal U}_{f,\tau })}

\def\LSfk{\text{LS}({\cal U}_{f,\tau }(K))}
\def\emLSfk{\text{{\em LS}}({\cal U}_{f,\tau }(K))}

\def\LSfak{\text{LS}({\cal U}_{f,\tau ,\ast }(K))}
\def\emLSfak{\text{{\em LS}}({\cal U}_{f,\tau ,\ast }(K))}

\def\LSfl{\text{LS}({\cal U}_{f,\tau }(L))}
\def\emLSfl{\text{{\em LS}}({\cal U}_{f,\tau }(L))}

\def\Ufk{{\cal U}_{f,\tau }(K)}
\def\Uvk{{\cal U}_{v,\tau }(K)}
\def\Ufak{{\cal U}_{f,\tau ,\ast }(K)}
\def\Uvak{{\cal U}_{v,\tau ,\ast }(K)}

\def\Ufl{{\cal U}_{f,\tau }(L)}
\def\Uvl{{\cal U}_{v,\tau }(L)}

\def\Uval{{\cal U}_{v,\tau ,\ast }(L)}
\def\supp{\mbox{supp}}
\def\Pt{\Bbb{P}^{2}}
\def\Poi{\Bbb{P}^{1}_{\infty }}
\def\Ct{\Bbb{C}^{2}}
\def\PACt{\mbox{PA}(\Bbb{C}^{2})}
\def\ve{\varepsilon}
\begin{document}
\markboth{Random Dynamical Systems of Polynomial Automorphisms on $\Bbb{C}^{2}$}{Random Dynamical Systems of Polynomial Automorphisms on $\Bbb{C}^{2}$}
\title{Random Dynamical Systems of Polynomial Automorphisms on $\Ct$
\footnote{Date: July 14, 2026.  
The author declares no conflicts of interest associated with this manuscript. 
Data sharing is not applicable to this article as no datasets were generated or analyzed during the current study.
Keywords: Random holomorphic dynamical systems, polynomial automorphisms, mean stability, randomness-induced phenomena, cooperation principle.  
 MSC: 37F10; 37H10}}
\author{{\bf Hiroki Sumi}\\ 
Division of Mathematical and Information Sciences,\\ 
Graduate School of Human and Environmental Studies, \\ 
Kyoto University, \\ 
Yoshida Nihonmatsu-cho, Sakyo-ku, Kyoto, 606-8501, Japan\\ 
{\bf E-mail: sumi@math.h.kyoto-u.ac.jp}\\ 
http://www.math.h.kyoto-u.ac.jp/\textasciitilde sumi/index.html\\ 
\date{}
}
\maketitle
%
%
\vspace{-10mm} 
\begin{abstract}
We consider random dynamical systems of polynomial automorphisms (complex generalized H\'{e}non maps and their conjugate maps) of $\Ct.$ We show that a generic random dynamical system of polynomial automorphisms has ``mean stablity'' on $\Ct$. Further, we show that 
if a system has mean stability, then (1) for each $z\in \Ct$ and 
for almost every sequence $\gamma =(\gamma _{n})_{n=1}^{\infty }$ of maps, the maximal Lyapunov 
exponent of $\gamma $ at $z$ is negative, 
(2) there are only finitely many minimal sets of the system, (3) 
each minimal set is attracting, (4) for each 
$z\in \Ct$ and for almost every sequence 
$\gamma $ of maps, the orbit 
$\{ \gamma _{n}\cdots \gamma _{1}(z) \} _{n=1}^{\infty }$ 
tends to one of the minimal sets of the system, and 
(5) the transition operator of the system has the spectrum gap 
property on the space of \Hol der continuous functions with some exponent.   
Note that none of (1)--(5) can hold for any deterministic iteration dynamical 
system of a single complex generalized H\'{e}non map. 
We observe many new phenomena in random dynamical systems of 
polynomial automorphisms of $\Ct$ and observe the mechanisms.  
We provide new strategies and methods to study higher-dimensional 
random holomorphic dynamical systems.    
\end{abstract}
\section{Introduction and  main results} 
Nature has many random (noise) terms. 
Thus, it is reasonable and important to consider 
\textcolor{black}{random dynamical systems}. 
Moreover, holomorphic dynamical systems have also been intensively investigated
(\cite{Be, Mi, MNTT00}). The study of these helps us to 
investigate real dynamical systems. 
Combining the above two ideas, we consider 
\textcolor{black}{random holomorphic dynamical systems}.  
The first research on random holomorphic dynamical 
systems was provided by Fornaess and Sibony 
(\cite{FS}) and many researchers  
have since investigated such systems
(\cite{BBR, CD, FW, GP, GQL, GR, JS1, JSAdv}, 
\cite{S3}--\cite{W} etc.) and non-autonomous holomorphic dynamical systems 
(\cite{BV, DS, S7} etc.).  
We want to find new phenomena (so-called randomness-induced phenomena) in random dynamical systems 
that cannot hold in deterministic iteration dynamical systems of single maps (see \cite{BBR, GQL, LZ, MT}, \cite{S3}--\cite{W}). 
We have several other motivations to study random holomorphic dynamical systems.  One of them is to use ``random relaxed Newton's methods'' in which 
we can find roots of polynomials \textcolor{black}{more easily than the deterministic methods} in a sense (\cite{S21}). In addition, many researchers have been investigating   
\textcolor{black}{the actions of holomorphic automorphisms} of complex manifolds (\cite{CD}) 
and  \textcolor{black}{the actions of 
mapping class groups} of many kinds of Riemann surfaces on the 
character varieties(\cite{C}), etc. These topics are strongly related to the 
study of random holomorphic dynamical systems. 
Herein, 
we investigate random 
dynamical systems of polynomial automorphisms on $\Ct$ (generalized H\'{e}non maps and their conjugate maps) and  show that a generic 
random dynamical system of polynomial automorphisms has a kind of stability and some order 
(e.g., for every $z\in \Ct$ , for almost every sequence of maps, 
the maximal Lyapunov exponent at $z$ is negative) that cannot 
hold for any deterministic iteration dynamical system of 
a single generalized H\'{e}non map.  

We give some notations and definitions as follows. 
\begin{df}
\label{d:lc2ms}
\begin{itemize}\ 
\item[\textcolor{black}{(1)}] 
Let $f: \Ct\rightarrow \Ct$ be a \textcolor{black}{polynomial map}, i.e., 
if we write $f(x,y)=(g(x,y), h(x,y))$, then 
$g(x,y)$ and $h(x,y)$ are polynomials of $(x,y).$ 
We say that $f$ is a \textcolor{black}{polynomial automorphism  
of $\Ct$} 
if $f$ is a holomorphic automorphism of $\Ct.$ 
Let $\textcolor{black}{\PACt}$ be the space of all polynomial automorphisms 
of $\Ct$. 
Note that if $f\in \PACt$, then $f^{-1}\in \PACt$ (\cite{MNTT00}).  
   
\textcolor{black}{Remark$(\ast )$}(\cite{FM, MNTT00}).   
If $f\in \PACt$, then $f$ is conjugate by an element $g\in 
 \PACt$ to one of the following maps: 
 \begin{itemize}
 \item[\textcolor{black}{(a)}] 
 an affine map $(x,y)\mapsto (ax+by+c, a'x+b'y+c'), 
 ab'-a'b\neq 0.$
\item[\textcolor{black}{(b)}] 
an elementary map 
 $(x,y)\mapsto (ax+b, sy+p(x)), as\neq 0,$  
 where $p(x)$ is a polynomial of $x$.
 \item[\textcolor{black}{(c)}] 
 a finite composition of some generalized H\'{e}non 
 maps $(x,y)\mapsto (y, p(y)-\delta x), \delta \neq 0,$ 
 where $p(y)$ is a polynomial of $y$ with $\deg (p)\geq 2.$  
\end{itemize}   
  
\item[\textcolor{black}{(2)}] 
Let $\textcolor{black}{X}^{+} $ be the space 
of all maps $f: \Ct\rightarrow \Ct$ 
of the form   
$f(x,y)=(y+\alpha, p(y)-\delta x) $, where 
$\alpha \in \CC, 
\delta \in \CC \setminus \{ 0\}, $ and 
$p(y)$ is a polynomial of $y$  with 
$\deg (p)\geq 2. $  Note that 
\textcolor{black}{$X^{+}\subset \PACt $ (\cite{MNTT00}). } 
We endow $X^{+}$  with the topology such that a 
sequence 
$\{ f_{j}(x,y)=(y+\alpha _{j}, p_{j}(y)-\delta _{j}x)\} _{j=1}^{\infty }  $ in $X^{+}$ converges to an element $f(x,y)=(y+\alpha , p(y)-\delta x)$ in $X^{+}$ if and only if 
(i) $\alpha _{j}\rightarrow \alpha \ (j\rightarrow \infty ), $ 
(ii) $\delta _{j}\rightarrow \delta \ (j\rightarrow \infty )$, 
(iii) $\deg (p_{j})=\deg (p)$ for each large number $j$, and
(iv) the coefficients of $p_{j}$ converge to the coefficients 
of $p$ appropriately as $j\rightarrow \infty .$ 
In addition, we set $\textcolor{black}{X^{-}}:=
\{ f^{-1}\in \PACt \mid f\in X^{+}\}$ endowed with 
 topology similar to that of $X^{+}.$ Note that 
the map $\Phi : X^{+}\rightarrow X^{-}$ defined by 
$\Phi (f)=f^{-1}$ is a homeomorphism. 
In addition, $X^{+}, X^{-}$ are metrizable, separable, and locally compact and are equal to countable disjoint unions 
of finite-dimensional complex manifolds. 
\vspace{2mm}  
  
\textcolor{black}{Remark.} 
If 
$f\in X^{\pm }$, then 
$f$ is conjugate to a generalized H\'{e}non map 
by an element $g\in \PACt$ 
(if $f\in X^{+}$ and 
$f(x,y)=(y+\alpha, p(y)-\delta x)$, then setting 
$g(x,y)=(x,y+\alpha )$, we have that 
$gfg^{-1}$ is a generalized H\'{e}non map). 
Moreover, any finite composition 
of elements of $X^{+}$ (resp. $X^{-}$) 
is conjugate by an element of $\mbox{PA}(\Ct)$ 
to a finite composition of some generalized H\'{e}non map
(we can see this by Remark$(\ast) $ and by 
observing dynamical degrees).  
\vspace{2mm} 
\item[\textcolor{black}{(3)}] 
Let $({\cal Y},d)$ be a metric space. We denote by 
${\frak M}_{1}({\cal Y})$ 
 the space of all Borel probability measures on ${\cal Y}.$ 
In addition, we set ${\frak M}_{1,c}({\cal Y})
:= \{ \tau \in {\frak M}_{1}({\cal Y})\mid \mbox{ supp}\,\tau \mbox{ is a compact subset of }{\cal Y}\}.$ Here, $\supp\,\tau $ denotes the topological support of $\tau .$ 
We endow ${\frak M}_{1,c}({\cal Y})$ with 
a topology $\textcolor{black}{{\cal O}}$, which satisfies that 
$\tau _{n}\in {\frak M}_{1,c}({\cal Y})\rightarrow \tau \in {\frak M}_{1,c}({\cal Y})$ as $n\rightarrow \infty $ 
if and only if 
\begin{itemize}
\item[\textcolor{black}{(a)}] for each bounded continuous function 
$\varphi : {\cal Y} \rightarrow \CC, $ we have 
$\int \varphi \, d\tau _{n}\rightarrow \int \varphi \, d\tau $ 
as $n\rightarrow \infty$, and
\item[\textcolor{black}{(b)}]  
supp$\,\tau _{n}\rightarrow $ supp$\,\tau $ as 
$n\rightarrow \infty $ with respect to the 
Hausdorff metric in the space of all nonempty 
compact subsets of ${\cal Y}.$ 
\end{itemize}

The above topology ${\cal O}$ is called the wH-topology.
  
We can show that any element $f\in X^{+} $ (resp. $X^{-}$) 
can be extended to a holomorphic map on 
$\Pt \setminus \{ [1:0:0]\} $ (resp. 
$\Pt \setminus \{ [0:1:0]\}$) 
(Lemma~\ref{l:ghsm}).    
For each $\tau \in {\frak M}_{1,c}(X^{+}),$ we consider 
\textcolor{black}{i.i.d. random dynamical system on 
$\Pt \setminus \{ [1:0:0]\} $} (resp. $\Pt \setminus \{ [0:1:0]\}$) 
 such that 
at every step we choose a map $f\in X^{+}$ 
(resp. $f\in X^{-}$) according to 
$\tau .$
 This defines a Markov process  whose state space 
is $\Pt\setminus \{ [1:0:0]\} $ 
(resp. $\Pt \setminus \{ [0:1:0]\}$) 
 and whose transition probability $p(z, A)$ 
from a point $z\in \Pt\setminus \{ [1:0:0]\} $ 
(resp. $\Pt \setminus \{ [0:1:0]\}$) 
 to a Borel subset $A$ of $\Pt\setminus \{ [1:0:0]\} $ 
 (resp. $\Pt \setminus \{ [0:1:0]\}$) 
satisfies $p(z, A)=\tau (\{ f\in X^{+}\mid f(z)\in A\} )$ (resp. 
$p(z,A)=\tau (\{ f\in X^{-} \mid f(z)\in A\} $).  
\item[\textcolor{black}{(4)}]
For  each $\tau \in {\frak M}_{1,c}(X^{\pm})$, 
let $\textcolor{black}{G_{\tau }}:= 
\{ \gamma _{n}\circ \cdots \circ \gamma _{1}\mid 
n\in \NN, \gamma _{j}\in \mbox{supp}\,\tau (\forall j)\} $. 
This is a \textcolor{black}{semigroup} with the semigroup operation being the 
 composition of maps.  
(Remark: It is important to study the dynamics of $G_{\tau }$ 
to investigate the random dynamical system generated by 
$\tau.$ For the studies on the dynamics of semigroups of 
holomorphic maps in one dimension, see \cite{HM, JS1, JS2, St1}, \cite{S3}--\cite{SU2}.)
\end{itemize}
\end{df} 
 
We introduce the notion of mean stability 
(on complex-two-dimensional spaces).  
\begin{df}
Let $\Lambda $ be a nonempty subset of 
$\Pt.$ 
We say that an element $\tau \in {\frak M}_{1,c}(X^{\pm})$ is \textcolor{black}{ mean stable} on $\Lambda $ 
if each $f\in \supp\,\tau $ is defined on $\Lambda $ and 
$f(\Lambda)\subset \Lambda $ and  
there exist an $\textcolor{black}{n}\in \NN $, an 
$\textcolor{black}{m}\in \NN $, 
nonempty open subsets 
\textcolor{black}{$U_{1},\ldots ,U_{m}$} of $\Lambda $ 
with respect to the relative topology from $\Pt$,  
a nonempty compact subset 
\textcolor{black}{$Q$} of $\cup _{j=1}^{m}U_{j}$, 
and a constant \textcolor{black}{$c$} with $0<c<1$ such that the following \textcolor{black}{(a)} and \textcolor{black}{(b)}  hold.
\begin{itemize}
\item[\textcolor{black}{(a)}] For each $(\gamma _{1},\ldots, \gamma _{n})\in 
(\mbox{supp}\,\tau)^{n},$ we have 
$\gamma _{n}\circ \cdots \circ \gamma _{1}(\cup _{j=1}^{m}
U_{j})\subset Q.$ Moreover, 
for each $j=1,\ldots, m$, for all $x,y\in U_{j}$ and for each $(\gamma _{1},\ldots, \gamma _{n})
\in (\mbox{supp}\,\tau)^{n}$, we have 
$$d(\gamma _{n}\circ \cdots \circ \gamma _{1}(x), 
\gamma _{n}\circ \cdots \circ \gamma _{1}(y))\leq 
cd(x,y),$$ where 
\textcolor{black}{$d$} denotes the distance induced by the 
Fubini--Study metric 
on $\Pt.$    
\item[\textcolor{black}{(b)}] 
For each $z\in \Lambda$, there exists an element 
$\textcolor{black}{f_{z}}\in G_{\tau }$ such that 
$f_{z}(z)\in \cup _{j=1}^{m}U_{j}.$  
\vspace{2mm} 
  
\end{itemize}
\end{df}
%
We introduce a nice class of elements $\tau \in {\frak M}_{1,c}(X^{\pm}).$ 
\begin{df}
Let $\textcolor{black}{{\cal MS}}$ 
be the set of all $\tau \in {\frak M}_{1,c}(X^{+})$ 
satisfying that 
\begin{itemize}
\item[\textcolor{black}{(i)}] 
$\tau $ is mean stable on 
$\Pt \setminus \{ [1:0:0]\}$ and
\item[\textcolor{black}{(ii)}]   
$\tau ^{-1} $ is mean stable on 
$\Pt \setminus \{ [0:1:0]\}$,
 where 
\textcolor{black}{$\tau ^{-1} $} is the element of 
${\frak M}_{1,c}(X^{-})$ \\ such that 
$\tau ^{-1}(A)= \tau (\{ f\in X^{+}\mid f^{-1}\in A\}) 
$ for each Borel subset $A$ of $X^{-}.$
\end{itemize}  
\end{df} 
\begin{rem}
It is easy to see that ${\cal MS}$ is \textcolor{black}{open} in 
${\frak M}_{1,c}(X^{+})$ with respect to the wH-topology ${\cal O}.$  
\end{rem}
 The following definition is needed to consider the continuity of 
nonautonomous Julia sets.  
\begin{df}
\label{d:contiseqsets}
Let $(\Lambda _{1}, d_{1}), (\Lambda _{2}, d_{2})  $ be  metric spaces and let 
$(J_{\lambda})_{\lambda \in \Lambda _{1}}$ be a family of closed subsets of 
$(\Lambda _{2}, d_{2}).$ Let $\lambda _{0}\in \Lambda _{1}.$ 
We say that the set-valued map 
$\lambda \in \Lambda _{1} \mapsto J_{\lambda }$ is continuous at $\lambda =\lambda _{0}$ 
if all of the following hold. 
\begin{itemize}
\item (upper semicontinuity) If $\lambda _{n}\rightarrow \lambda _{0}$ as $n\rightarrow \infty $ in 
$\Lambda _{1}$, $z_{n}\in J_{\lambda _{n}} (n\in \NN)$ and $z_{n}\rightarrow 
z_{0}$ in $\Lambda _{2}$ as $n\rightarrow \infty $, then $z_{0}\in J_{\lambda _{0}}.$ 
\item  (lower semicontinuity) If $\lambda _{n}\rightarrow \lambda _{0}$ as $n\rightarrow \infty $ in 
$\Lambda _{1}$ and $z_{0}\in J_{\lambda _{0}}$, then there exists a sequence 
$\{ z_{n}\} _{n=1}^{\infty }$ in $\Lambda _{2}$ with $z_{n}\in J_{\lambda _{n}} (n\in \NN)$ 
such that $z_{n}\rightarrow z_{0}$ in $(\Lambda _{2}, d_{2})$ as $n\rightarrow \infty .$  
  \end{itemize}

\end{df}
\begin{df}
For each $\tau \in {\frak M}_{1}(X^{\pm})$, we set 
$\tau ^{\ZZ}:= \otimes _{n=-\infty }^{\infty }\tau \in {\frak M}_{1}((X^{\pm})^{\ZZ}).$
\end{df}
We now present the first main result of this paper.

\begin{thm}
\label{t:rpmms1} 
{\em (}see Theorems~\ref{t:yniceaod}, \ref{t:mtauspec}, 
\ref{t:kjemfhf}, \ref{t:kjemfsp}.{\em )}  
 ${\cal MS}$ is \textcolor{black}{open and dense} in 
 ${\frak M}_{1,c}(X^{+})$ with respect to the wH-topology ${\cal O}.$  
Moreover, for each $\tau \in {\cal MS}$, 
we have all of the
 following \textcolor{black}{{\em (1)--(6)}}. 
 \vspace{-1mm} 
\begin{itemize}
\item[\textcolor{black}{{\em (1)}}] 
There exists a constant $c_{\tau }$ with \textcolor{black}{$c_{\tau }<0$} 
such that the following holds. 

 \begin{itemize}
 \item \textcolor{black}{For each $z\in \Pt\setminus \{ [1:0:0]\}$}, there exists a Borel subset $\textcolor{black}{B^{+}_{\tau, z}}$ 
 of $(X^{+})^{\ZZ }$ with $\tau ^{\ZZ}(B^{+}_{\tau ,z})=1$ 
 such that for each $\gamma =(\gamma _{j}) _{j\in \ZZ}\in 
 B^{+}_{\tau ,z}$, we have
 \vspace{-2mm} 
 $$\limsup _{n\rightarrow \infty }\frac{1}{n}\log \| D(\gamma _{n-1}\circ \cdots \circ  
 \gamma _{0})_{z}\| \leq c_{\tau }\textcolor{black}{<0}.$$ 
 In addition, \textcolor{black}{for each $z\in \Pt\setminus \{ [0:1:0]\}$}, there exists a Borel subset $\textcolor{black}{B^{-}_{\tau, z}}$ 
 of $(X^{+})^{\ZZ }$ with $\tau ^{\ZZ}(B^{-}_{\tau ,z})=1$ 
 such that  for each $\gamma =(\gamma _{j}) _{j\in \ZZ}\in 
 B^{-}_{\tau ,z}$, we have
 \vspace{-2mm}  
 $$\limsup _{n\rightarrow \infty }\frac{1}{n}\log \| D(\gamma^{-1} _{-n}\circ \cdots \circ  
 \gamma ^{-1}_{-1})_{z}\| \leq c_{\tau }\textcolor{black}{<0}.$$ 
 
 Here, for each rational map $f$ on $\Pt$  and for each $z\in \Pt$ where $f$ is defined, 
 we denote by \textcolor{black}{$\| Df_{z}\| $} the norm of the 
 differential of $f$ at $z$ with respect to the 
 Fubiny--Study metric in $\Pt.$   
  \end{itemize}
    
\item[\textcolor{black}{{\em (2)}}]    
\textcolor{black}{For each $z\in \Pt\setminus \{ [1:0:0]\}$}, there exists a 
Borel subset $\textcolor{black}{C^{+}_{\tau, z}} $ of 
$(X^{+})^{\ZZ }$ 
with $\tau^{\ZZ}(C^{+}_{\tau , z})=1$ 
such that for each $\gamma =(\gamma _{j})_{j\in \ZZ}
\in C^{+}_{\tau ,z}$, there exists a number 
$r^{+}=r^{+}(\tau, z, \gamma )>0$ satisfying that 
$$\mbox{diam} (\gamma _{n-1}\circ \cdots \circ \gamma _{0}
(B(z, r^{+})))\rightarrow 0 \mbox{ as }n\rightarrow \infty $$ 
exponentially fast, and 
\textcolor{black}{for each $z\in \Pt\setminus \{ [0:1:0]\}$}, there exists a 
Borel subset $\textcolor{black}{C^{-}_{\tau, z}} $ of 
$(X^{+})^{\ZZ }$ 
with $\tau^{\ZZ}(C^{-}_{\tau , z})=1$ 
such that for each $\gamma =(\gamma _{j})_{j\in \ZZ}
\in C^{-}_{\tau ,z}$, there exists a number 
$r^{-}=r^{-}(\tau, z, \gamma )>0$ satisfying that 
$$\mbox{diam} (\gamma ^{-1}_{-n}\circ \cdots \circ 
\gamma^{-1} _{-1}
(B(z, r^{-})))\rightarrow 0 \mbox{ as }n\rightarrow \infty $$ 
exponentially fast,
where $\textcolor{black}{B(z, r)}$ denotes the ball 
with center $z$ and radius $r$ with respect to the 
distance $d$ induced by the 
Fubini--Study metric on $\Pt$, and for each subset $A$ of 
$\Pt$, we set $\textcolor{black}{\mbox{diam}A}:=\sup _{x,y\in A}d(x,y).$  \\ 
  
\item[\textcolor{black}{{\em (3)}}]
For each $\gamma =(\gamma _{j})_{j\in \ZZ}\in 
(X^{+})^{\ZZ}$, let 
$\textcolor{black}{K_{\gamma }^{+}}:=
\{ z\in \CC^{2}\mid \{ \gamma _{n-1}\circ \cdots \circ \gamma _{0}(z)\} _{n=1}^{\infty } \mbox{ is bounded in }\CC^{2}\} ,$
$\textcolor{black}{K_{\gamma }^{-}}:=
\{ z\in \CC^{2}\mid \{ \gamma ^{-1}_{-n}\circ \cdots \circ \gamma ^{-1}_{-1}(z)\} _{n=1}^{\infty } \mbox{ is bounded in }\CC^{2}\} ,$ 
$\textcolor{black}{J_{\gamma }^{+}}:=\partial K_{\gamma}^{+}, $
$\textcolor{black}{J_{\gamma }^{-}}:=\partial K_{\gamma}^{-} $ 
with respect to the topology in $\Ct.$ 
Then for each $\gamma \in (\supp\tau)^{\ZZ}$, 
$J_{\gamma }^{+}$ is a closed subset of $\Pt \setminus \{ [0:1:0]\}$ 
and $J_{\gamma }^{-}$ is a closed subset of 
$\Pt \setminus \{ [1:0:0]\}. $ Moreover,  
there exists a Borel subset 
\textcolor{black}{$D_{\tau }$} of 
$(X^{+})^{\ZZ}$ with $\tau ^{\ZZ}
(D_{\tau })=1$ such that the following \textcolor{black}
{{\em (a)(b)}} hold.
\vspace{2mm} 
\begin{itemize}
\item[\textcolor{black}{{\em (a)}}]
For each $\gamma \in D_{\tau }$, we have 
\textcolor{black}{{\em Leb}$_{4}(J_{\gamma }^{\pm })=0$},  where \textcolor{black}{{\em Leb}$_{4}$} denotes the 
$4$-dimensional Lebesgue measure on $\CC^{2}.$ 
\vspace{2mm} 
\item[\textcolor{black}{{\em (b)}}]
For each $\gamma \in D_{\tau }$,  
the map 
$\beta =(\beta _{j})_{j\in \ZZ} \in (\supp\,\tau)^{\ZZ} 
\mapsto J_{\beta }^{\pm }$
is 
\textcolor{black}{continuous at $\beta =\gamma $}.  
(Here, we regard $J_{\beta }^{+}$ as a closed subset of $\Pt \setminus \{ [0:1:0]\}$ 
and we regard $J_{\beta }^{-}$ as a closed subset of 
$\Pt \setminus \{ [1:0:0]\}. $)

\end{itemize}
  
\vspace{2mm} 
\item[\textcolor{black}{{\em (4)}}] 
Let $\textcolor{black}{Y}:=\Pt\setminus \{ [1:0:0]\}$ and 
let $\textcolor{black}{C(Y)}$ be the space of all complex-valued bounded continuous 
 functions on $Y$. Let 
$\textcolor{black}{M_{\tau }}: C(Y)\rightarrow C(Y)$ be the linear operator 
defined by 
\vspace{-1mm} 
$$M_{\tau }(\varphi )(z)=\int _{X^{+}}\varphi (f(z))\ d\tau (f), 
\mbox{ for } \varphi \in C(Y), z\in Y.$$
(This $M_{\tau }$ is the transition operator of the Markov process generated by $\tau$ (\cite{D}).) 
Then there exists a finite dimensional subspace 
$\textcolor{black}{W_{\tau }}\neq \{ 0\} $ of $C(Y)$ with $M_{\tau }(W_{\tau })=W_{\tau }$ 
such that the following holds. \\ 
\vspace{2mm} 
\textcolor{black}{For each neighborhood $B$ of $[1:0:0]$}, 
there exists a compact subset  \textcolor{black}{$K$} of 
$\Pt \setminus \{ [1:0:0]\}$ 
with 
$\Pt \setminus K\subset B$ 
satisfying all of the following \textcolor{black}{{\em (a)} and {\em (b)}}.  
\begin{itemize}
\item[\textcolor{black}{{\em (a)}}]
$f(K)\subset K$ for each $f\in \supp\, \tau $.   
\item[\textcolor{black}{{\em (b)}}]
\textcolor{black}{For each $\varphi \in C(K)$}, 
where \textcolor{black}{$C(K)$} denotes the 
Banach space of all complex-valued 
continuous functions on $K$ endowed with the supremum norm,    
\textcolor{black}{$\{ M_{\tau }^{n}(\varphi )\} _{n=1}^{\infty }$ tends to 
$W_{\tau }|_{K}:=\{ \psi |_{K}\in C(K)\mid \psi \in W_{\tau }\} $} 
in $C(K)$.  
More precisely, there exists a continuous map 
$\pi _{\tau , K}: C(K)\rightarrow W_{\tau }|_{K}$ such that 
for each $\varphi \in C(K)$,  
$\| M_{\tau }^{n}(\varphi -\pi _{\tau ,K}(\varphi ))\| _{\infty }
\rightarrow 0$ as $n\rightarrow \infty $, where 
$\| \cdot \| _{\infty }$ denotes the supremum norm in $C(K)$. 
\end{itemize} 
%
  
\item[\textcolor{black}{{\em (5)}}] 
\textcolor{black}{For each neighborhood $B$ of $[1:0:0]$}, 
there exists a compact subset  \textcolor{black}{$K$} of 
$\Pt \setminus \{ [1:0:0]\}$ 
with 
$\Pt \setminus K\subset B$
such that $f(K)\subset K$ for each $f\in \supp\, \tau $,   
and
such that 
there exists a number
 $\textcolor{black}{\alpha} =\alpha (\tau, K)$ with $0<\alpha <1$  
satisfying 
the following \textcolor{black}{{\em (a)}} and 
\textcolor{black}{{\em (b)}}.   
\begin{itemize}
\item[\textcolor{black}{{\em (a)}}] 
The space $W_{\tau }|_{K}$, where $W_{\tau }$ is the 
space in \textcolor{black}{{\em (4)}},  is included in 
the Banach space $\textcolor{black}{C^{\alpha }(K)}$ of all $\alpha $-\Hol der continuous 
functions on $K$ endowed with $\alpha$-\Hol der norm. 
\item[\textcolor{black}{{\em (b)}}] 
For each $\varphi \in C^{\alpha }(K)$, 
the sequence $\{ M_{\tau }^{n}(\varphi )\} _{n=1}^{\infty }$ 
tends to $W_{\tau }|_{K}$ \textcolor{black}{exponentially fast}. 
More precisely, there exist a constant $\lambda =\lambda (K)\in (0,1)$ and 
a constant $C=C(K)>0$ such that 
for each $\varphi \in C^{\alpha }(K)$ and for each $n\in \NN$, 
we have $\| M_{\tau }^{n}(\varphi -\pi _{\tau ,K}(\varphi ))\| _{\alpha }
 \leq C\lambda ^{n}\| \varphi \| _{\alpha }$. 
{\em (}Thus, $M_{\tau }: C^{\alpha }(K)\rightarrow C^{\alpha }(K)$ 
has the ``\textcolor{black}{spectral gap property}'' 
(see Theorem~\ref{t:kjemfsp}{\em ).)}  
\end{itemize}
%
  
\item[\textcolor{black}{{\em (6)}}] We have all of the following. 
\begin{itemize}
\item[\textcolor{black}{{\em (a)}}] 
Let $\textcolor{black}{\emMin (\tau )}$ be the set of all minimal sets of $\tau $ in 
$\Pt \setminus \{ [1:0:0]\}.$ Then, 
$1\leq \sharp \emMin (\tau )\textcolor{black}{<}\infty . $ Here, 
a nonempty compact subset $L$ of $\Pt \setminus \{ [1:0:0]\} $ 
is said to be a \textcolor{black}{minimal set of $\tau $} if 
$L=\overline{\cup _{h\in G_{\tau }}\{ h(z)\}}$ for each 
$z\in L.$ 
\item[\textcolor{black}{{\em (b)}}] 
\textcolor{black}{For each $z\in \Pt \setminus \{ [1:0:0]\}$}, 
there exists a Borel subset ${\cal C}_{z}$ of 
$(\supp\,\tau )^{\ZZ}$ with 
$\tau ^{\ZZ}({\cal C}_{z})=1$ such that 
for each $\gamma =(\gamma _{j})_{j\in \ZZ}\in 
{\cal C}_{z}$, 
we have 
$d(\gamma _{n-1}\circ \cdots \circ \gamma _{0}(z), \cup _{L\in \emMin (\tau )}L)\rightarrow 0\ \mbox{ as } n\rightarrow \infty . $ 
\item[\textcolor{black}{{\em (c)}}] 
For each $L\in \emMin (\tau )$ and for each 
$z\in \Pt \setminus \{ [1:0:0]\}$, 
let \textcolor{black}{$T_{L,\tau }(z)$} be the probability 
of tending to $L$ starting with $z$, i.e., 
$$\textcolor{black}{T_{L, \tau }(z)}=\tau ^{\ZZ}
(\{ (\gamma _{j})_{j\in \ZZ}\in (\supp\,\tau)^{\ZZ}\mid d(\gamma _{n-1}\circ 
\cdots \circ \gamma _{0}(z), L)\rightarrow 0 \mbox{ as }
n\rightarrow \infty \}).$$ 
Then, 
\textcolor{black}{$T_{L, \tau } \in W_{\tau }$} and 
$T_{L, \tau }$ is \textcolor{black}{locally 
\Hol der continuous} on $Y=\Pt \setminus \{ [1:0:0]\}.$ 
Also, $T_{L,\tau }$ is constant on any connected component 
of the Fatou set $F(G_{\tau})$ of $G_{\tau }$, 
where $F(G_{\tau })$ denotes the set of points 
$z\in Y$ for which there exists a neighborhood $U$ of $z$ in $Y$ 
such that $G_{\tau }$ is equicontinuous on $U$ as a family of maps 
from $U$ to $\Pt $ with respect to the distance induced by the 
Fubini-Study metric on $\Pt.$ 
\end{itemize} 
\end{itemize}
\end{thm}

\begin{rem}
\label{r:furres}Moreover, we show that if ${\cal Y}$ is a subset of $X^{+}$ 
that satisfies some 
 mild conditions (e.g.``nice condition (I)'' (Definition~\ref{d; nicesetI})),  
then the set of elements $\tau \in {\frak M}_{1,c}({\cal Y})$ which are mean stable on 
$\Pt \setminus \{ [1:0:0]\}$ is open and dense in $({\frak M}_{1,c}({\cal Y}), {\cal O})$  
(Theorems~\ref{t:rpmms2}, \ref{t:yniceaod}, Example~\ref{e:addnoise}.)  
For further results on elements $\tau \in {\cal MS}$, 
see Theorem~\ref{t:mtauspec} and Remark~\ref{r:sametauthm}. 

We remark that 
none of the statements \textcolor{black}{(1)(2)(4)(5)(6)} \textcolor{black}{in 
Theorem~\ref{t:rpmms1}  can hold for 
any deterministic dynamical system of a single $f\in X^{\pm }$} 
(\cite{BLS, MNTT00}) 
(thus, for each $f\in X^{+}$, $\delta _{f}\not\in {\cal MS}$).   
We also remark that for any $n\in \NN $ there exists an element 
$\tau \in {\cal MS}$ with $\sharp \Min (\tau )\geq n+1\geq 2.$ 
In fact, if $f\in X^{+}$ has $n$ attracting periodic cycles $E_{1},\ldots, E_{n}$ 
in $\Ct$ (e.g. let $p(y)$ be a polynomial  having $n$ attracting periodic cycles in $\CC $ and let $f(x,y)=(y, p(y)-\delta x)$ with small $|\delta |$), then by Theorem~\ref{t:rpmms1} for each small neighborhood 
$U$ of $\delta _{f}\in {\frak M}_{1,c}(X^{+})$, there exists an element 
$\tau \in U\cap {\cal MS}.$ If $U$ is small enough, then 
$\sharp \Min (\tau )\geq n+1$ 
($\{ [0:1:0]\} \in \Min (\tau )$, and for each $j=1,\ldots ,n$, there exists an element 
$L_{E_{j}}\in \Min (\tau )$ with $E_{j}\subset L_{E_{j}}\subset \Ct$). 
For such an element $\tau$, for each $L\in \Min (\tau )$, 
the function $T_{L,\tau}$ on $Y=\Pt \setminus \{ [1:0:0]\}$ 
is continuous on $Y$, nonconstant on $Y$, and constant on any connected 
component of  $F(G_{\tau }).$ 
\end{rem}
  We now consider random dynamical systems of elements in 
  $X^{+}$ 
  that preserve the volume of $\Ct.$ 
\begin{df}
Let $\textcolor{black}{X_{1}^{+}}$ be the space of all 
elements $f\in X^{+}$ of the form 
$f(x,y)=(y+\alpha , p(y)-x)$, where 
$\alpha \in \CC$ and $p(y)$ is a polynomial of 
$y$ with $\deg (p)\geq 2.$ 
We endow $X_{1}^{+}$ with the relative 
topology from $X^{+}$. 
\end{df}
We now present the second main result of this paper. 
  
\begin{thm}
\label{t:rpmms2}
{\em (}see Theorems~\ref{t:yniceaod} and \ref{t:mtauspec}.{\em )} 
${\cal MS}\cap {\frak M}_{1,c}(X_{1}^{+})$ is \textcolor{black}
{open and dense} in ${\frak M}_{1,c}(X_{1}^{+})$ with respect 
to the wH-topology ${\cal O}$ in ${\frak M}_{1,c}(X_{1}^{+}).$ 
Moreover, for each 
$\tau \in {\cal MS}\cap {\frak M}_{1,c}(X_{1}^{+})$, 
we have the following {\em (1)} and {\em (2)}. 
\begin{itemize}
\item[\textcolor{black}{{\em (1)}}]  
\textcolor{black}{For each $z\in \Pt\setminus \{ [1:0:0]\}$}, 
there exists a Borel subset $\textcolor{black}{E_{\tau, z}^{+}}$ of 
$(X_{1}^{+})^{\ZZ}$ with $\tau ^{\ZZ}(E_{\tau, z}^{+})=1$ 
such that for each 
$\gamma =(\gamma_{j})_{j\in \ZZ}\in E_{\tau ,z}^{+}$, we have 
$$\gamma _{n-1}\circ \cdots \circ \gamma _{0}(z)
\rightarrow [0:1:0] \mbox{ as }n\rightarrow \infty ,$$
and 
\textcolor{black}{for each $z\in \Pt\setminus \{ [0:1:0]\}$}, 
there exists a Borel subset $\textcolor{black}{E_{\tau, z}^{-}}$ of 
$(X_{1}^{+})^{\ZZ}$ with $\tau ^{\ZZ} (E_{\tau, z}^{-})=1$ 
such that for each 
$\gamma =(\gamma_{j})_{j\in \ZZ}\in E_{\tau ,z}^{-}$, we have 
$$\gamma _{-n}^{-1}\circ \cdots \circ \gamma _{-1}^{-1}(z)
\rightarrow [1:0:0] \mbox{ as }n\rightarrow \infty .$$
\item[\textcolor{black}{{\em (2)}}]
For $\tau ^{\ZZ}$-a.e. 
$\gamma \in (X_{1}^{+})^{\ZZ}$, 
we have 
\textcolor{black}{$J_{\gamma }^{\pm}=K_{\gamma }^{\pm }$} 
and 
\textcolor{black}{{\em Leb}$_{4}(K_{\gamma }^{\pm })=0.$}
\end{itemize}  
\end{thm}

\begin{rem}
\label{r:msfinsupp}
It is easy to see that 
 the set $\{ \tau \in {\cal MS}\mid 
\sharp \supp\,\tau <\infty \}$ is dense 
in ${\cal MS}$ with respect to the wH-topology, 
and 
the set $\{ \tau \in {\cal MS}\cap {\frak M}_{1,c}(X_{1}^{+})\mid 
\sharp \supp\,\tau <\infty \}$ is dense 
in ${\cal MS}\cap {\frak M}_{1,c}(X_{1}^{+})$ 
with respect to the wH-topology. 
Hence, by Theorems~\ref{t:rpmms1} and \ref{t:rpmms2}, we have 
${\frak M}_{1,c}(X^{+})=\overline{\{ \tau \in {\cal MS}\mid 
\sharp \supp\,\tau <\infty \}}$ in 
$({\frak M}_{1,c}(X^{+}), {\cal O})$  
and 
${\frak M}_{1,c}(X_{1}^{+})=
\overline{\{ \tau \in {\cal MS}\cap {\frak M}_{1,c}(X_{1}^{+})\mid 
\sharp \supp\,\tau <\infty \}}$ in 
$({\frak M}_{1,c}(X_{1}^{+}), {\cal O})$.  
\end{rem}
The following theorem, which deals with the stability and bifurcation of 
a family $\{ \tau _{t}\} _{t}$ in ${\frak M}_{1,c}(X^{+})$,  is the key for  
proving the density of ${\cal MS}$ in ${\frak M}_{1,c}(X^{+})$ with respect to 
${\cal O}$ in 
Theorem~\ref{t:rpmms1} and 
the density of ${\cal MS}\cap {\frak M}_{1,c}(X_{1}^{+})$ in 
${\frak M}_{1,c}(X_{1}^{+})$ with respect to ${\cal O}$ 
in Theorem~\ref{t:rpmms2}. 
\begin{thm}
\label{t:inyItaut} (For more detailed version, see Theorem~\ref{t:nyItaut}.) 
Let ${\cal Y}=X^{+}$ or let ${\cal Y}=X^{+}_{1}.$ 
Let $\{ \tau _{t}\} _{t\in [0.1]}$ be a family of 
elements of ${\frak M}_{1,c}({\cal Y})$ such that 
all of the following hold. 
\begin{itemize}
\item[{\em (i)}] 
$t\in [0,1]\mapsto \tau _{t}\in {\frak M}_{1,c}({\cal Y})$ is continuous 
with respect to the wH-topology ${\cal O}.$ 

\item[{\em (ii)}]
If $t_{1}, t_{2}\in [0,1]$ and $t_{1}<t_{2}$, then 
$\supp\,\tau _{t_{1}}\subset \mbox{int}(\supp\,\tau _{t_{2}}).$ 
Here, int denotes the set of interior points with respect to the topology 
in ${\cal Y}.$  
\item[{\em (iii)}] 
$\mbox{int} (\supp\,\tau _{0})\neq \emptyset .$ 
\end{itemize} 
Let $A:=\{ t\in [0,1]\mid \tau _{t} \mbox{ is not mean stable on }
\Pt \setminus \{ [1:0:0]\} \}.$
Then $\sharp A\leq \emMin(\tau _{0})-1<\infty $ and 
$B:=\{ t\in [0,1]\mid s\mapsto \sharp \emMin(\tau _{s}) \mbox{ is 
constant in a neighborhood of }t\} $ satisfies 
$B= [0,1]\setminus A.$

\end{thm}
The set $A$ in Theorem~\ref{t:inyItaut} (Theorem~\ref{t:nyItaut}) is called the set of bifurcation parameters 
for the family $\{ \tau _{t}\} _{t}.$ We can construct many families $\{ \tau _{t}\} _{t}$ 
for which conditions (i)(ii)(iii) in Theorem~\ref{t:inyItaut} are satisfied and the 
set $A$ of bifurcation parameters is not empty (see Example~\ref{e:addnoise}). 
 
Note that none of the statements \textcolor{black}{(1)}(2)(4)(5)\textcolor{black}{(6)} 
in Theorem~\ref{t:rpmms1}   
\textcolor{black}{can hold for deterministic dynamics of a single 
polynomial automorphism $f$ on $\CC^{2}$ 
which is 
conjugate to a generalized H\'{e}non map by a polynomial automorphism 
(\cite{BLS, MNTT00}).} 
Hence, herein, we see many \textcolor{black}
{randomness-induced phenomena} (phenomena in random dynamical systems that cannot 
hold for iteration dynamics of single maps). 
In particular, we see 
\textcolor{black}{randomness-induced order}. 
Many kinds of maps in one random dynamical system 
automatically cooperate together to 
make the chaoticity weaker. 
We call this phenomena as the 
``\textcolor{black}{{\bf Cooperation Principle}.''} 
Even if a random dynamical system has  randomness-induced order, the system can still have multiple attracting minimal sets 
(see Remark~\ref{r:furres}), can have some diversity,  
 and 
can have some complexity
(\cite{Splms10}, \cite{Sadv}).  
We have to investigate the 
\textcolor{black}{gradation between chaos and order}
in random dynamical systems.  
The exponent 
\textcolor{black}{$\alpha $} 
in Theorem~\ref{t:rpmms1} (5) may represent the 
quantity of the gradation between chaos and order.   
  
A rough idea of the proof of  
Theorem~\ref{t:rpmms1} is as follows. 
  To show the density of ${\cal MS}$ in ${\frak M}_{1,c}(X^{+}) $, 
 let $\zeta \in {\frak M}_{1,c}(X^{+})$ and let $U$ be an open 
 neighborhood of $\zeta $ in ${\frak M}_{1,c}(X^{+}).$ 
 Then, there exists an element $\zeta _{0}\in U$ with 
 $\sharp \supp\,\zeta _{0}<\infty .$ By enlarging the support of 
 $\zeta _{0}$, we can construct a family 
 $\{ \tau _{t}\}_{t\in [0,1]}$ of elements in $U$  
 such that  $\{ \tau _{t}\}_{t\in [0,1]}$ satisfies 
 the conditions (i)(ii)(iii) in Theorem~\ref{t:inyItaut}. 
Then, by Theorem~\ref{t:inyItaut}, 
there exists a $t>0$ such that $\tau _{t}\in U\cap {\cal MS}$.  
This argument shows the density of ${\cal MS}$ in ${\frak M}_{1,c}(X^{+}) .$ Similarly, we can show the density of 
${\cal MS}\cap {\frak M}_{1,c}(X_{1}^{+})$ in ${\frak M}_{1,c}(X_{1}^{+}) .$ 
For the proofs of the results on elements $\tau \in {\cal MS}$, 
we develop the ideas in \cite{Splms10, Sadv} 
(where we mainly deal with complex-one-dimensional random holomorphic 
dynamical systems) to the settings of the higher-dimesional random holomorphic 
dynamical systems. To show the continuity of nonautonomous Julia sets  in Theorem~\ref{t:rpmms1} (\ref{t:mtauspec}), 
we construct nonautonomous Green's functions 
(see section~\ref{s:nadj}), show 
the lower semicontinuity of 
the nonautonomous Julia sets (Proposition~\ref{p:Ggcontiph}), 
and then combine the above arguments with the results on 
elements $\tau \in {\frak M}_{1,c}(X^{+})$ with mean stability.  

A rough idea of the proof of Theorem~\ref{t:inyItaut} 
(\ref{t:nyItaut}),   
 which is the key for  
proving the density of ${\cal MS}$ in ${\frak M}_{1,c}(X^{+})$ 
 in Theorem~\ref{t:rpmms1} and 
the density of 
${\cal MS}\cap {\frak M}_{1,c}(X_{1}^{+})$ in 
${\frak M}_{1,c}(X_{1}^{+})$,  
is as follows. 
Under the assumptions of Theorem~\ref{t:inyItaut}, by Zorn's lemma, we can show that  
 if $t_{1},t_{2}\in [0,1], t_{1}<t_{2}$ then 
 $\sharp \Min(\tau _{t_{2}})\leq \sharp \Min(\tau _{t_{1}}).$ Moreover, 
 we can easily show that $\sharp \Min(\tau _{0})<\infty .$ 
 It follows that 
 we have $\sharp ([0,1]\setminus B)<\infty .$ 
 Furthermore, by using the Carath\'{e}odory distances on bounded domains, we can show that if $t\in B$, then any $L\in \Min(\tau _{t})$ 
 with $L\subset \Ct$ is attracting for $\tau _{t}$, which implies that 
 $\tau _{t}$ is mean stable on $\Pt \setminus \{ [0:1:0]\}.$ 
 From these arguments, we can prove Theorem~\ref{t:inyItaut} 
(Theorem~\ref{t:nyItaut}).   

We remark that we have several results on complex-one-dimensional 
random holomorphic dynamical systems in \cite{Sadv} which 
are similar to Theorem~\ref{t:inyItaut} and the result on the density of 
${\cal MS}$ in ${\frak M}_{1,c}(X^{+})$ in Theorem~\ref{t:rpmms1}, 
but 
{\bf the methods and arguments in the proofs of them in \cite{Sadv, S21} 
  were valid only for complex-one-dimensional random holomorphic dynamical systems}. 
However, in this paper,  
we introduce {\bf some new and powerful methods 
that are valid for  random holomorphic dynamical systems 
of any dimension} in the proofs of many results, especially in the proof of Theorem~\ref{t:inyItaut} 
(Theorem~\ref{t:nyItaut})(the key for  
proving the density of ${\cal MS}$ in ${\frak M}_{1,c}(X^{+})$ 
 in Theorem~\ref{t:rpmms1} and 
the density of 
${\cal MS}\cap {\frak M}_{1,c}(X_{1}^{+})$ in 
${\frak M}_{1,c}(X_{1}^{+})$). We emphasize that {\bf the methods in 
the proof of Theorem~\ref{t:inyItaut} 
(Theorem~\ref{t:nyItaut}) are completely different from those in the proof of 
\cite[Theorem 1.8]{Sadv}}. 
Note also that as $\Pt \setminus \{ [1:0:0]\} $ and 
$\Pt \setminus \{ [0:1:0]\}$ are not compact, 
some extra efforts are required to prove the results on the global behavior of 
the random dynamical systems and the iterations of transition operators. 
In particular, we need 
some observations on the actions of the systems 
on small neighborhoods of the line at infinity.  
\begin{rem}
\label{r:compare} 
This paper is the first one in which we see that for a 
generic random holomorphic dynamical systems 
(with some natural and mild conditions) in 
a complex manifold of dimension greator than $1$, 
{\bf (I)} for all initial value $z$ and for almost 
every seqence $\gamma $ of maps 
the maximal Lyapunov exponent at $z$ with respect to $\gamma $ 
is negative (Theorem~\ref{t:rpmms1}(1)), {\bf (II)} the transition operator of the system 
has the spectral gap property on the Banach space of 
\Hol der continuous functions with some \Hol der exponent 
(Theorem~\ref{t:rpmms1}(5)), 
{\bf (III)} the orbit of any Borel probability measure on the phase space 
under the iterations of the dual of the transition operator of the system 
tends to a periodic cycle of probability measures 
(Theorem~\ref{t:mtauspec}-\ref{t:mtauspec8},\ref{t:mtauspecdual}),   
and 
{\bf (IV)} the map $\beta \mapsto J_{\beta }^{\pm }$ is continuous 
at almost every $\gamma $ (see Theorem~\ref{t:rpmms1}(3), note that the lower semicontinuity follows 
directly from the standard arguments on Green's functions but the 
upper semicontinuity does not related to Green's functions and the upper semicontinuity follows from 
the mean stability).  Note that the topic on the 
transition operators is completely different from the topic on 
the Green's functions, the Green current or decay of correlations. 
We emphasize that the main results and the methods in the proofs in this paper 
are very different from those of the other papers (\cite{BV, CD, DS, FW, GP}) on higher-dimensional 
random holomorphic dynamical systems or non-autonomous dynamical systems. 
We explain the difference between this paper and the other papers as follows. 

(1) The paper \cite{FW} deals with the randon dynamical 
 systems of holomorphic maps on $\Bbb{P}^{k}.$ More precisely 
 the authors of \cite{FW} investigated holomorphic families 
 $\{ f_{\lambda }\} _{\lambda \in B(0,\delta)}$ on 
 the $k$-dimensional complex ball $B(0,\delta )$ for which the map 
 $\lambda \mapsto f_{\lambda }(z)$ is finite (hence open) 
 for each $z\in \Bbb{P}^{k}$,  
 and denoting  by $\Phi : B(0,\delta ) \rightarrow 
 \mbox{Holo}(\Bbb{P}^{k})$ the map $\lambda \mapsto f_{\lambda }$ 
  (where Holo($\Bbb{P}^{k}$) denotes the space of all holomorphic maps on 
  $\Bbb{P}^{k}$) and denoting by $\nu $ the normalized 
  Lebesgue measure on $B(0,\delta )$, 
  they investigated the i.i.d. random dynamical systems on $\Bbb{P}^{k}$ 
  generated 
  by $\tau = \Phi _{\ast }(\nu ).$ The results and the methods in 
  \cite{FW} depend heavily on the assumption that $\nu $ is the normalized Lebesgue measure (or that $\nu $ is equivalent to Lebesgue measure), whereas in this paper 
  the probability measure $\tau $ is very general  and $\nu $ may have a  
  singular part with respect to the Lebesgue measure 
  (e.g. the support of $\tau $ can be 
  finite even if $\tau \in {\cal MS}$).  
 Note that the statement ``for each $\mu , \ T^{n}(\mu )\rightarrow \mu _{A}$'' (where $T=M_{\tau }^{\ast }$ (the dual map of the transition operator of the system), 
 $\mu $ is an arbitrary Borel probability measure on $\Bbb{P}^{k}$, and 
 $\mu _{A}$ denotes a Borel probability measure on $\Bbb{P}^{k}$ with 
 $M_{\tau }^{\ast }(\mu _{A})=\mu _{A}$) in 
 \cite[Theorem 8]{FW} is not correct.  
In fact, if $k=1, f_{0}(z)=z^{2}-1,$
 $f_{\lambda }(z)=\frac{1}{\frac{1}{z^{2}-1}+\lambda }+\lambda\ 
 (\lambda \in B(0,\delta ))$,   
   $\delta $ 
 is small enough, and $\mu $ is the Dirac measure concentrated at $0$, 
 then $(M_{\tau }^{\ast })^{n}(\mu )$ does not converge weakly 
 (although $(M_{\tau }^{\ast })^{2n}(\mu ) $ converges). (For, since $\{0,-1\}$ is an attracting periodic cycle of 
 $f_{0}$ of period $2$, for any small enough $\delta >0$, there exist a neighborhood $V_{0}$ of $0$ and 
 a neighborhoold $V_{-1}$ of $-1$ with $V_{0}\cap V_{-1}=\emptyset $ 
 such that 
 $\supp\, (M_{\tau }^{\ast })^{2n}(\mu )\subset V_{0}$ and $\supp\,(M_{\tau }^{\ast })^{2n+1}(\mu )
 \subset V_{-1}.$) Similarly, for any $k\in \NN$, 
 if $f_{0}\in \mbox{Holo}(\Bbb{P}^{k})$ has an attracting periodic cycle 
 $\{a_{1},\ldots ,a_{r}\}$ 
 of period $r$ with $r>1$, 
 $\{f_{\lambda }\} _{\lambda \in B(0,\delta )}$ is a holomorphic family of 
 small perturbations of $f_{0}$,  
 and $\mu $ is the Dirac measure 
 concentrated at $a_{1}$, then $(M_{\tau }^{\ast })^{n}(\mu )$ does not converge weakly  
 (although $(M_{\tau }^{\ast })^{rn}(\mu )$ converges weakly). Like this example, 
 we have to consider ``period'' $r_{L}$ that may be greater than $1$ for any minimal set 
 $L$ of $\tau $ (see Theorem~\ref{t:mtauspec}, Example~\ref{e:periodn} and \cite[Theorem 3.15]{Splms10}). 
 We also remark that the Lyapunov exponents of the systems were not investigated 
 in \cite{FW}.   
 
 (2) The paper \cite{GP} deals with the random local dynamical systems at a point 
 $p\in \Bbb{C}^{n}$ 
 of holomorphic maps fixing the point $p$, and the global dynamics and the 
 transition operators were not discussed.  
 The paper \cite{CD} deals with the classification of stationary measures of 
 random holomorphic dynamical systems on complex manifolds, and  the details 
 of the behavior of the transition operators and the Lyapunov exponents were not 
 discussed.    
 
 (3) The papers \cite{BV, DS} deal with non-autonomous dynamical systems of 
 sequences of 
 holomorphic maps on $\Bbb{C}^{n}$ and the non-autonomous Green's functions were discussed. 
 In this paper we have to use non-autonomous Green's functions to prove the main results, and we use some arguments similar to those in \cite {BV,DS}. 
 However, the results on Green's functions are not the main topics in this paper.  
 Also, the setting of this paper is different from \cite{BV} (in this paper 
 we deal with the maps $f(x,y)=(y+a, p(y)-\delta x)$ where $a$ might not be zero), 
 and the continuity of Green's functions ${\cal G}_{\gamma }^{\pm }(x,y)$ with respect to the sequence $\gamma $ of 
 maps, which we need to prove one of the main results 
 (Theorem~\ref{t:rpmms1}(3)) in this paper, were not discussed 
 in \cite{BV, DS}. 
Note that the papers \cite{BV, DS} do not deal with random dynamical systems 
 (Markov process) induced by probability measures on the space of holomorphic maps. 
 Thus the main framework of \cite{BV, DS} are different from that of this paper.  
\end{rem}

The rest of this paper is organized as follows. In section~\ref{s:nadj}, we provide 
some fundamental observations on nonautonomous 
dynamics of elements of $X^{\pm}$, 
Green's functions, and the Julia sets. 
In section~\ref{s:fratp}, we provide further detailed results 
and we  prove them and the main results in this paper.

\section{Nonautonomous dynamics and the Julia sets}
\label{s:nadj}
In this section we provide  
some fundamental observations on nonautonomous 
dynamics of elements of $X^{\pm}$, 
Green's functions, and the Julia sets (although the topics on Green's functions are not the main ones of this paper) since we need them to prove the main results of this paper. 
Some of the contents and discussions below follow from the arguments in \cite{FM, FW, MNTT00, BV, DS}. However, 
since we deal with the elements $f\in X^{+}$ of the form 
$f(x,y)=(y+a, p(y)-\delta x)$ (where $a$ might not be zero and might depend on 
$f$) and the non-autonomous iterations of sequences $\{ f_{n}\}$ 
with 
$f_{n}(x,y)=(y+a_{n}, p_{n}(y)-\delta _{n}x)$ (note that 
\cite{BV} deals with non-autonomous iterations of sequences 
$\{ f_{n}\}$ with $f_{n}(x,y)=(y, p_{n}(y)-\delta _{n}x)$ and the higher-dimensional analogues), and since 
we need to show some uniform estimates on the elements $f$ in any 
compact subset $\Gamma $ in $X^{+}$ (Lemma~\ref{l:wrconv}), 
the uniform convergence of the sequences of the functions ${\cal G}_{\gamma , n}^{+}(x,y)$ on $\gamma \in \Gamma ^{\ZZ}$ 
(Lemmas~\ref{l:greenvr}, \ref{l:Gganpuc}) and the continuity of 
the functions ${\cal G}_{\gamma }^{+}(x,y)=\lim _{n\rightarrow \infty }{\cal G}_{n,\gamma }(x,y)$ with respect to  $\gamma \in \Gamma ^{\ZZ}$ 
that implies the lower-semicontinuity of $\gamma \mapsto J_{\gamma }^{+}$(Proposition~\ref{p:Ggcontiph}) 
(we remark that such observations were not included in \cite{FM, MNTT00, BV,DS}, and 
\cite{FW} deals with non-autonomous iterations of sequences of 
holomorphic maps on $\Bbb{P}^{n}$ and the setting of \cite{FW} is different from that of this paper), we include the proofs of lemmas and proposition in this section for the readers. Note that the lower semicontinuity of the map 
$\gamma \mapsto J_{\gamma }^{+}$ is needed to prove one of the 
main results of this paper (Theorem~\ref{t:rpmms1} (3)).  
   
\noindent {\bf Notation.} 
For each $R>1$ we set 
$V_{R}^{+}=\{ (x,y)\in \Ct \mid \max\{R, |x|\}<|y|\}$ and 
$V_{R}^{-}=\{ (x,y)\in \Ct \mid \max\{R, |y|\} <|x|\}$.  

\begin{df}
\label{d:condA}
Let $g\in \PACt.$ Let $R>0$ and $\rho >1.$ 
We say that $g$ satisfies condition (A) for $(R,\rho)$ if we have the following. 
\begin{itemize}
\item[(1)]
If $(x,y)\in \overline{V_{R}^{+}}$, then $g(x,y)\in V_{R}^{+}$ and 
$|\pi _{y}(g(x,y))|>\rho |y|.$ 

\item[(2)]
If $(x,y)\in \overline{V_{R}^{-}}$, then $g^{-1}(x,y)\in V_{R}^{-}$ and 
$|\pi _{x}g^{-1}(x,y)|>\rho |x|.$ 

\end{itemize}
Here, we set $\pi _{x}(x,y)=x$ and $\pi _{y}(x,y)=y.$ 

\end{df}

\begin{lem}
\label{l:Rrho}
Let $g\in X^{+}.$ Then,  
there exists a number $R_{0}>1$ such that for each $R\geq R_{0}$, 
there exists a $\rho >1$ such that $g$ satisfies condition (A) for 
$(R,\rho ).$ 
\end{lem}
\begin{proof}
Let $g\in X^{+}$ be of the form 
$g(x,y)=(y+a, p(y)-\delta x)$, where $a\in \CC, \delta \in \CC, \delta \neq 0$, 
and $p(y)$ is a polynomial of $y$ with $\deg (p)\geq 2.$  
Then, it is easy to see that 
\begin{equation}
\label{eq:ginverse}
g^{-1}(x,y)=(\frac{1}{\delta }(p(x-a)-y), x-a).
\end{equation}
Let $\rho _{0}>\max\{ |\delta |+8, 16|\delta |+2\} $ be a number.  
Then, there exists a number $R_{0}=R_{0}(\rho _{0})>\max\{ 1, 2|a|\} $  such that 
for each $R>R_{0}$ and for each $\zeta \in \CC$ with $|\zeta |\geq R$, we have 
$\frac{|p(\zeta )|}{|\zeta |}\geq \rho _{0}.$ 

 Let $R>R_{0}.$ Let $z=(x,y)\in \overline{V_{R}^{+}}.$ 
Then $|y|\geq R, |y|\geq |x|.$ Hence,  
\begin{eqnarray*}
\frac{|p(y)-\delta x|}{|y+a|} 
& \geq  & \frac{\rho _{0}|y|-|\delta ||x|}{|y|+|a|}
 \geq   \frac{(\rho _{0}-|\delta |)|y|}{|y|+|a|} 
\geq \frac{\rho _{0}-|\delta |}{1+\frac{|a|}{|y|}} 
 \geq  \frac{\rho _{0}-|\delta |}{1+\frac{|a|}{R}}>4. 
\end{eqnarray*}
Thus $|p(y)-\delta x|> 4|y+a|\geq 4(|y|-|a|)\geq 4|y|(1-\frac{|a|}{|y|})
\geq 4|y|(1-\frac{|a|}{R})\geq 2|y|>R.$ 

We now let $(x,y)\in \overline{V_{R}^{-}}.$ 
Then, $|x|\geq R, |x|\geq |y|.$ Hence,   
\begin{eqnarray*}
\frac{|\frac{1}{\delta}(p(x-a)-y)|}{|x-a|}
& \geq & \frac{\frac{1}{|\delta |}(\rho _{0}|x-a|-|y|)}{|x|+|a|}
 \geq  \frac{1}{|\delta |} \frac{\rho _{0}(|x|-|a|)-|x|}{|x|+|a|}\\ 
& = & \frac{1}{|\delta |} \frac{\rho _{0}(1-\frac{|a|}{|x|})-1}
        {1+\frac{|a|}{|x|}}
 \geq  \frac{1}{|\delta |} 
\frac{\rho _{0}(1-\frac{|a|}{R})-1}{1+\frac{|a|}{R}}
 \geq  \frac{1}{|\delta |} \frac{\frac{\rho _{0}}{2}-1}{2}
 >  4.
\end{eqnarray*}
Therefore,  
$|\frac{1}{\delta }(p(x-a)-y)|>4|x-a|\geq 4(|x|-|a|)\geq 
4|x|(1-\frac{|a|}{|x|})\geq 4|x|(1-\frac{|a|}{R})\geq 2|x| 
>R.$ From these arguments,  
it follows that 
$g$ satisfies condition (A) for $(R, 2).$ 

Hence, we have proved our lemma. 
\end{proof}
\begin{df}
We set $\Bbb{P}_{\infty }^{1}
:=\{ [x:y:0]\in \Pt \mid (x,y)\in \Ct \setminus \{ (0,0)\}\}.$ 
This is called the line at infinity. 
\end{df}
\begin{lem}
\label{l:ghsm}
Let $g\in X^{+}.$ Then, we have the following. 
\begin{itemize}
\item[{\em (1)}] 
$g$ can be extended to a holomorphic self map 
$\hat{g}$ 
on 
$\Pt \setminus \{ [1:0:0]\} $ and 
$[0:1:0]$ is an attracting fixed point of $\hat{g}.$ 
Also we have $\hat{g}(\Poi \setminus \{ [1:0:0]\} )=\{ [0:1:0]\}.$ 
\item[{\em (2)}]
$g^{-1}$ can be extended to a holomorphic self map 
$\hat{g}^{-1}$ on $\Pt \setminus \{ [0:1:0]\}$ and 
$[1:0:0]$ is an attracting fixed point of $\hat{g}^{-1}.$ 
Also we have $\hat{g}^{-1}(\Poi \setminus \{[0:1:0]\} =\{ [1:0:0]\} .$
\end{itemize} 
\end{lem}
\begin{proof}
Let $g\in X^{+}.$ We write $g(x,y)=(y+a, p(y)-\delta x)$ where 
$a\in \CC, \delta \in \CC, \delta \neq 0$ and $p(y)=c_{0}y^{n}+c_{1}y^{n-1}+
\cdots +c_{n}$ is a polynomial of $y$ with 
$c_{0}\neq 0.$  

Let $\hat{g}$ be the birational map on $\Pt$ such that 
$\hat{g}|_{\Ct}=g.$  Then,  
$$\hat{g}([x:y:z])=[yz^{n-1}+az^{n}: c_{0}y^{n}+c_{1}y^{n-1}z+\cdots +c_{n}z^{n}-\delta xz^{n-1}: z^{n}].$$ 
It is easy to see that the set of indeterminancy points of $\hat{g}$ is 
$\{ [1:0:0]\} $ and 
$\hat{g}$ is a holomorphic self map on 
$\Pt \setminus \{ [1:0:0]\}.$ 
Also, it is easy to see that $[0:1:0]$ is an attracting fixed point of $\hat{g}$ and 
$\hat{g}(\Poi \setminus \{ [1:0:0]\} )=\{ [0:1:0]\}.$ 
Similarly, we can show statement (2).
\end{proof}
\begin{rem}
For each $g\in X^{+}$, we regard $g$ as a holomorphic self map on 
$\Pt \setminus \{ [1:0:0]\}$ and we regard $g^{-1}$ as a holomorphic 
self map on $\Pt \setminus \{ [0:1:0]\}.$ 

\end{rem}

\begin{df}
\label{d:drwr}
For each $R>0$, let 
$D_{R}=\{ (x,y)\in \Ct \mid \max\{ |x|, |y|\} <R\}$, 
$W_{R}^{+}=\Ct \setminus \overline{V_{R}^{-}}$, and 
$W_{R}^{-}=\Ct \setminus \overline{V_{R}^{+}}$. 
Note that $W_{R}^{+}=\mbox{int}(D_{R}\cup \overline{V_{R}^{+}})$ and 
$W_{R}^{-}=\mbox{int}(D_{R}\cup \overline{V_{R}^{-}}).$ 
\end{df}
\begin{lem}
\label{l:wrconv}
Let $\Gamma $ be a compact subset of $X^{+}.$ Then, there exists a 
number $R_{0}>0$ such that for each $R\geq R_{0}$, we have the following. 
\begin{itemize}
\item[{\em (1)}]
There exists a number $\rho >1$ such that 
each $g\in \Gamma $ satisfies condition (A) for $(R,\rho ).$  

\item[{\em (2)}] 
For each $(x,y)\in \Ct$ and for each $\gamma =(\gamma _{j})_{j\in \ZZ}
\in \Gamma ^{\ZZ }$ there exists a number $n_{0}=n_{0}((x,y), \gamma)\in \NN $ such that 
for each $n\in \NN $ with $n\geq n_{0}$, we have 
$\gamma _{n-1}\circ \cdots \circ \gamma _{0}(x,y)\in W_{R}^{+}$ and 
$\gamma _{-n}^{-1}\circ \cdots \circ \gamma _{-1}^{-1}(x,y)\in 
W_{R}^{-}.$ 
\item[{\em (3)}]
For each $\gamma =(\gamma _{j})_{j\in \ZZ }\in \Gamma ^{\ZZ}$, 
$\gamma _{n-1}\circ \cdots \circ \gamma _{0}\rightarrow 
[0:1:0] \mbox{ as } n\rightarrow \infty $ uniformly on the closure of 
$V_{R}^{+}$ in $\Pt$, and 
$\gamma _{n}^{-1}\circ \cdots \circ \gamma _{1}^{-1}\rightarrow 
[1:0:0] \mbox{ as } n\rightarrow \infty $ uniformly on the closure of 
$V_{R}^{-}$ in $\Pt.$  
\end{itemize}

\end{lem}
\begin{proof}
By using the argument in the proof of Lemma~\ref{l:Rrho},  
we can take a number $R_{0}>0$ such that 
for each $R\geq R_{0}$, statement (1) holds. 
Taking 
a large enough  value for $R_{0}$, 
Lemma~\ref{l:ghsm} implies that 
for each $R\geq R_{0}$, statement (2) holds. 
(Alternatively, for any large $R$, we can prove statement (2) by using statement (1) and 
the arguments similar to \cite[p236--239]{MNTT00}.)   
For this $R_{0}$, for each $R\geq R_{0}$, statement (3) follows from statement (1) and Lemma~\ref{l:ghsm}.
\end{proof}
\begin{df}
Let $\gamma =(\gamma _{j})_{j\in \ZZ} \in \mbox{PA}(\Ct )^{\ZZ}.$ 
We set\\  
$I_{\gamma }^{+}:=\{ z\in \Ct \mid \| \gamma _{n-1}\circ \cdots \circ 
\gamma _{0}(z)\| \rightarrow \infty \mbox{ as }n\rightarrow \infty \},$\\ 
 $I_{\gamma }^{-}:=\{ z\in \Ct \mid \| \gamma _{n}^{-1}\circ \cdots \circ 
\gamma _{-1}^{-1}(z)\| \rightarrow \infty \mbox{ as }n\rightarrow \infty \},$\\
$K_{\gamma }^{+}:=\{ \gamma _{n-1}\circ \cdots \gamma _{0}(z)\} _{n\in \NN} \mbox{ is bounded in } \Ct\}$, 
$K_{\gamma }^{-}:=\{ \gamma _{-n}^{-1}\circ \cdots \gamma _{-1}^{-1}(z)\} _{n\in \NN} \mbox{ is bounded in } \Ct\}$,\\ 
and 
$J_{\gamma }^{\pm }:= \partial K_{\gamma }^{\pm }$ 
(with respect to the topology in $\Ct$). (Hence,  
$J_{\gamma }^{\pm}\subset \Ct$.)  
$J_{\gamma }^{\pm }$ are called the Julia sets of $\gamma .$ 
\end{df}
\begin{lem}
\label{l:igkg}
Let $\Gamma $ be a compact subset of $X^{+}$ and 
let $\gamma =(\gamma _{j})_{j\in \ZZ}\in \Gamma ^{\ZZ}.$ 
Then, $I_{\gamma }^{+}\cup K_{\gamma }^{+}=\Ct, I_{\gamma }^{+}\cap 
K_{\gamma }^{+}=\emptyset$, 
$I_{\gamma }^{-}\cup K_{\gamma }^{-}=\Ct, $ and $I_{\gamma }^{-}\cap 
K_{\gamma }^{-}=\emptyset.$ 
Moreover, $J_{\gamma }^{+}$ is equal to the set of all 
$z\in \Ct $ for which there exists no neighborhood $U$ of $z$ in $\Ct$ such that 
$\{ \gamma _{n-1}\circ \cdots \circ \gamma _{0}\}_{n\in \NN}$ 
is equicontinuous as a family of maps from $U$ to $\Pt.$   
Also, $J_{\gamma }^{-}$ is equal to the set of all 
$z\in \Ct $ for which there exists no neighborhood $U$ of $z$ in $\Ct$ such that 
$\{ \gamma _{n}^{-1}\circ \cdots \circ \gamma _{-1}^{-1}\}_{n\in \NN}$ 
is equicontinuous as a family of maps from $U$ to $\Pt.$ 
\end{lem}
\begin{proof}
The statement of our lemma follows from statement (1) in 
Lemma~\ref{l:wrconv} and the arguments similar to \cite[p236--239]{MNTT00}. 
\end{proof}
\begin{df}
For each $\tau \in {\frak M}_{1}(X^{\pm })$ we denote by $G_{\tau }$ 
the subsemigroup of $\PACt$ generated by $\supp\,\tau$, i.e. 
$G_{\tau }=\{ g_{n}\circ \cdots \circ g_{1}\mid n\in \NN, 
g_{j}\in \supp\,\tau (\forall j)\}.$  
\end{df}
\begin{df}
\begin{itemize}
\item[(1)]
Let $\tau \in {\frak M}_{1,c}(X^{+})$ (resp. $\tau \in {\frak M}_{1,c}(X^{-})$). 
We say that a nonempty compact subset $L$ of 
$\Pt \setminus \{ [1:0:0]\} $ (resp. $\Pt \setminus \{ [0:1:0]\}$) 
is a minimal set of $\tau $ if $L=\overline{\cup _{g\in G_{\tau }}\{ g(x)\} }$ 
for each $x\in L.$ Also, we denote by $\Min (\tau )$ the set of 
all minimal sets of $\tau .$  
\item[(2)]
Let $\tau \in {\frak M}_{1,c}(X^{+})$ (resp. $\tau \in {\frak M}_{1,c}(X^{-})$).  
Let $\tau $ be a minimal set of $\tau .$ 
We say that $L$ is attracting for $\tau $ if 
there exist finitely many open subsets $U_{1}, \ldots, U_{m}$ of 
$\Pt \setminus \{ [1:0:0]\}$ (resp. $\Pt \setminus \{ [0:1:0]\}$), 
a compact subset $K$ of $\cup _{j=1}^{m}U_{j}$, a number 
$n\in \NN$, and a constant $c\in (0,1)$ 
such that $L\subset K$, 
such that 
for each $g\in \supp\,\tau$, we have 
$g(\cup _{j=1}^{m}U_{j})\subset K$, 
and such that for each $j=1,\ldots, m$, for each $x,y\in U_{j}$ and 
for each $(g_{1},\ldots, g_{n})\in (\supp\,\tau)^{n}$, we have 
$d(g_{n}\circ \cdots \circ g_{1}(x), g_{n}\circ \cdots \circ g_{1}(y))
\leq cd(x,y)$ where $d$ denotes the distance on $\Pt$ induced from the 
Fubini-Study metric on $\Pt.$  A minimal set $L$ of $\tau $ that  
is attracting for $\tau $ is called an attracting minimal set of $\tau .$   
\end{itemize}
\end{df}
\begin{df}
For each  $(\gamma _{j})_{j\in \ZZ}\in (\mbox{PA}(\Ct ))^{\ZZ}$ 
and for each $m,n\in \ZZ$ with $m\geq n$, 
we set $\gamma _{m,n}=\gamma _{m}\circ \cdots \circ \gamma _{n}.$ 
Similarly, for each  $(\gamma _{j})_{j\in \NN}\in (\mbox{PA}(\Ct ))^{\NN}$ 
and for each $m,n\in \NN$ with $m\geq n$, 
 we set $\gamma _{m,n}=\gamma _{m}\circ \cdots \circ \gamma _{n}.$ 
 Also, we denote by $\sigma : (\mbox{PA}(\Ct ))^{\ZZ} 
 \rightarrow (\mbox{PA}(\Ct ))^{\ZZ}$ the shift map, i.e., 
 $\sigma ((\gamma _{j})_{j\in \ZZ})=(\sigma _{j+1})_{j\in \ZZ}.$ 
 Similarly, we denote by $\sigma : (\mbox{PA}(\Ct ))^{\NN} 
 \rightarrow (\mbox{PA}(\Ct ))^{\NN}$ the shift map, i.e., 
 $\sigma ((\gamma _{j})_{j\in \NN})=(\sigma _{j+1})_{j\in \NN}.$ 
\end{df} 
We provide some observations on $K_{\gamma }^{\pm }.$ 
\begin{lem}
\label{l:gjcaek}
Let $\gamma =(\gamma _{j})_{j\in \ZZ} \in (X^{+})^{\ZZ}. $ 
Suppose that there exists a number 
$R_{0}>0$ such that for each $R>R_{0}$, there exists a number 
$\rho >1$ such that for each $j\in \ZZ$, $\gamma _{j}$ 
satisfies condition (A) for $(R, \rho ).$ 
Then, we have the following. 
\begin{itemize}
\item[{\em (1)}] 
Let $R>R_{0}$ and let $E_{\gamma ,R}:=D_{R}\cap K_{\gamma }^{+}.$ 
Then, $\gamma _{1}(E_{\gamma, R})\subset E_{\sigma (\gamma), R}.$ 

\item[{\em (2)}]
\begin{itemize}
\item[{\em (i)}] 
For each $R>R_{0}$ and each $n\in \NN$, we have 
$E_{\gamma ,R}\subset \gamma _{1}^{-1}(E_{\sigma (\gamma ), R})
\subset \gamma _{1}^{-1}(\gamma _{2}^{-1}(E_{\sigma ^{2}(\gamma ), R}))
\cdots \subset \gamma _{1}^{-1}\cdots \gamma _{n}^{-1}
(E_{\sigma ^{n}(\gamma ), R}).$
\item[{\em (ii)}] 
$K_{\gamma }^{+}=
\cup _{n=1}^{\infty }\gamma _{1}^{-1}\cdots \gamma _{n}^{-1}
(E_{\sigma ^{n}(\gamma ),R}).$
\end{itemize}
\item[{\em (3)}] 
Suppose that $\limsup _{n\rightarrow \infty }\frac{1}{n} 
\Sigma _{j=1}^{n}\log |\delta (\gamma _{j})|>0,$ where 
$\delta (g)$ denotes the Jacobian of $g$ for each $g\in X^{\pm}$ 
(this is a constant).  
Then, {\em Leb}$_{4}(K_{\gamma }^{+})=0.$ 
\item[{\em (4)}] 
Suppose that 
$\limsup _{n\rightarrow \infty} \frac{1}{n}\sum _{j=1}^{\infty }
\log |\delta (\gamma _{j})|<0$ and 
$\limsup _{n\rightarrow \infty }\mbox{{\em Leb}}_{4}(K_{\sigma ^{n}(\gamma )}^{+})>0.$ 
Then, {\em Leb}$_{4}(K_{\gamma }^{+})=\infty .$ 
\item[{\em (5)}] 
Suppose that 
$\limsup _{n\rightarrow \infty} \frac{1}{n}\sum _{j=1}^{\infty }
\log |\delta (\gamma _{j})|<0$ and 
suppose also that there exists a nonempty open bounded 
set $U$ in $\Ct $ such that 
for each $j\in \NN$, $\gamma _{j}(U)\subset U.$ Then,  
{\em Leb}$_{4}(K_{\gamma }^{+})=\infty .$  
\item[{\em (6)}] 
Let $\tau \in {\frak M}_{1,c}(X^{+}). $ 
Suppose that there exists an attracting minimal set 
$L$ 
of $\tau $
 with $L\subset \Ct.$ 
 Then 
$\max \{ |\delta (h)|\mid h\in \supp\,\tau \}<1$ and  
 for each $\gamma =(\gamma _{j})_{j\in \ZZ} \in (\supp \,\tau)^{\ZZ}$, 
 {\em Leb}$_{4}(K_{\gamma }^{+})=\infty .$ 
  
\end{itemize}
\end{lem}
\begin{proof}
We first show  statement (1). 
We have $\gamma _{1}(E_{\gamma , R})=
\gamma _{1}(D_{R})\cap \gamma _{1}(K_{\gamma }^{+})
=\gamma _{1}(D_{R})\cap K_{\sigma (\gamma )}^{+}.$ 
Since $K_{\sigma (\gamma )}^{+}\cap V_{R}^{+}=\emptyset$ 
(Lemma~\ref{l:wrconv}),  
we have $\gamma _{1}(E_{\gamma, R})\subset 
\Ct \setminus V_{R}^{+}.$ Moreover,  
since there exists a $\rho >1$ such that  
$\gamma _{1}$ satisfies Condition (A) for $(R,\rho )$, we have 
$\gamma _{1}(E_{\gamma ,R})\subset \gamma _{1}(D_{R})
\subset D_{R}\cup V_{R}^{+}.$ It follows that 
$\gamma _{1}(E_{\gamma ,R})\subset D_{R}.$ From these arguments, 
we obtain that $\gamma _{1}(E_{\gamma ,R})\subset D_{R}\cap 
K_{\sigma (\gamma )}^{+}=E_{\sigma (\gamma ),R}.$ 
Thus, we have proved  statement (1). 

 By statement (1), we can easily show  statement (2)(i). 
 
 We now prove the statement (2)(ii).   
 For each $n\in \NN, $ we have 
$$\gamma _{1}^{-1}\cdots \gamma _{n}^{-1}(E_{\sigma ^{n}(\gamma ),R})
\subset \gamma _{1}^{-1}\cdots \gamma _{n}^{-1}
(K_{\sigma ^{n}(\gamma )}^{+}) \subset K_{\gamma }^{+}.$$ 
Thus,  
$\cup _{n=1}^{\infty }\gamma _{1}^{-1}\cdots \gamma _{n}^{-1}(E_{\sigma ^{n}(\gamma ),R})\subset K_{\gamma }^{+}.$ 
To prove $K_{\gamma }^{+} \subset 
\cup _{n=1}^{\infty }\gamma _{1}^{-1}\cdots \gamma _{n}^{-1}(E_{\sigma ^{n}(\gamma ),R})$, let 
$z_{0}\in K_{\gamma }^{+}.$ 
Suppose that 
$z_{0}\not\in \cup _{n=1}^{\infty }\gamma _{1}^{-1}\cdots \gamma _{n}^{-1}(E_{\sigma ^{n}(\gamma ),R}).$ Then,  
for each $n\in \NN,$ we have 
$\gamma _{n,1}(z_{0})\in 
\Ct \setminus E_{\sigma ^{n}(\gamma ), R}=
(\Ct \setminus K_{\sigma ^{n}(\gamma )}^{+})\cup 
(\Ct \setminus D_{R}).$ 
Moreover, siince $z_{0}\in K_{\gamma }^{+}$, we have 
$\gamma _{n,1}(z_{0})\in 
K_{\sigma ^{n}(\gamma )}^{+}\subset 
\Ct \setminus \overline{V_{R}^{+}}$ for each $n\in \NN.$  
It follows that 
$\gamma _{n,1}(z_{0})\in 
(\Ct \setminus D_{R})\cap 
(\Ct \setminus \overline{V_{R}^{+}})\subset \overline{V_{R}^{-}}$ 
for each $n\in \NN.$ 
Combining this with $z_{0}\in K_{\gamma }^{+}$,  
we obtain that there exist a strictly increasing sequence 
$\{ n_{j}\} _{j=1}^{\infty }$ in $\NN $ and a point 
$z_{1}\in \overline{V_{R}^{-}}$ such that 
\begin{equation}
\label{eq:z0z1conv}
\gamma _{n_{j}}\cdots \gamma _{1}(z_{0})\rightarrow 
z_{1}\ \mbox{ as }j\rightarrow \infty .
\end{equation}
Let $R_{1}\in \RR $ with $R_{0}<R_{1}<R.$ 
By the assumptions of our lemma, there exists a number 
$\rho _{1}>1$ such that 
each $\gamma _{i}$ satisfies condition (A) for $(R_{1}, \rho _{1}).$ 
Hence, $\| (\gamma _{n_{j}}\cdots \gamma _{1})^{-1}(z)\| 
\rightarrow \infty \mbox{ as } j\rightarrow \infty $ uniformly 
on a neighborhood of $z_{1}.$ However, this contradicts to 
(\ref{eq:z0z1conv}). Thus, $z_{0}\in \cup _{n=1}^{\infty }\gamma _{1}^{-1}\cdots \gamma _{n}^{-1}(E_{\sigma ^{n}(\gamma ),R}).$
Hence, we have proved statement (2)(ii). 

 We now prove statement (3).  Let $\{n_{i}\}_{i=1}^{\infty }$ be a strictly increasing sequence 
 in $\NN $ such that 
 $\prod _{k=1}^{n_{i}}|\delta (\gamma _{k})|^{-1}\rightarrow 
 0$ as $i\rightarrow \infty.$ By statement (2), 
 we have that 
 \begin{eqnarray*}
 \mbox{Leb}_{4}(K_{\gamma }^{+})
 & = & \lim _{i\rightarrow \infty }
 \mbox{Leb}_{4}(\gamma _{1}^{-1}\cdots 
 \gamma _{n_{i}}^{-1}
 (E_{\sigma ^{n_{i}}(\gamma ),R}))
 =  \lim _{i\rightarrow \infty }\prod _{k=1}^{n_{i}}| \delta (\gamma _{k})
 |^{-2}\mbox{Leb}_{4}(E_{\sigma ^{n_{i}}(\gamma ),R})\\ 
 & \leq & \lim _{i\rightarrow \infty }
 \prod _{k=1}^{n_{i}}|\delta (\gamma _{k})|^{-2}\mbox{Leb}_{4}(D_{R})=0
  \end{eqnarray*} 
Therefore, $\mbox{Leb}_{4}(K_{\gamma }^{+})=0.$ Hence, we have 
proved statement (3). 

We now prove statement (4). By the assumptions of statement (4), 
$\prod _{j=1}^{n}|\delta (\gamma _{j})|^{-2}\rightarrow \infty 
\mbox{ as }n\rightarrow \infty $ 
and there exist a strictly increasing sequence 
$\{ n_{i}\} _{i=1}^{\infty }$ in $\NN $ and a positive constant $c>0$ such that 
$\mbox{Leb}_{4}(K_{\sigma ^{n_{i}}(\gamma )}^{+})\geq c \mbox{ for each }i\in \NN.$
Combining these with the fact $\gamma _{n_{i}}\cdots \gamma _{1}(K_{\gamma }^{+})
=K_{\sigma ^{n_{i}}(\gamma )}^{+}$ for each $i\in \NN$, 
we obtain that 
$\mbox{Leb}_{4}(K_{\gamma }^{+})=
\prod _{j=1}^{n_{i}}|\delta (\gamma _{j})|^{-2}
(\mbox{Leb}_{4}(K_{\sigma ^{n_{i}}(\gamma )}^{+}))\rightarrow 
\infty \mbox{ as } i\rightarrow \infty .
$
Hence, $\mbox{Leb}_{4}(K_{\gamma }^{+})=\infty .$ Therefore, 
we have proved statement (4) of our lemma. 

Statement (5) follows from statement (4). 

 We now prove statement (6). By the assumptions of statement (6), 
 there exists a constant $a\in (0,1)$ such that 
 for each $h\in \supp\, \tau, |\delta (h)|<a.$ Combining this with 
 statement (5), we can see that statement (6) follows. 
 
  Hence, we have proved our lemma.  
\end{proof}
We now consider nonautonomous Green's functions. 
\begin{df}
For each $\gamma =(\gamma _{j})_{j\in \ZZ}
\in (X^{+})^{\ZZ}$, 
for each $n\in \NN$ and for each $(x,y)\in \Ct$, 
let ${\cal G}_{\gamma ,n}^{+}(x,y)=
\frac{1}{\deg(\gamma _{n-1,0})} \log ^{+}\| \gamma _{n-1,0}(x,y)\| $ 
and let ${\cal G}_{\gamma ,n}^{-}(x,y)
=\frac{1}{\deg (\gamma _{-1, -n})}
\log ^{+}\| \gamma _{-n}^{-1}\circ \cdots \circ \gamma _{-1}^{-1}
(x,y)\| .$  
Here, for each $f\in \mbox{PA}(\Ct)$ of the form 
$f(x,y)=(g(x,y), h(x,y))$, we set $\deg (f):=
\max\{\deg (g), \deg (h)\}.$ Also, 
we set $\log ^{+}(x):=\max\{\log x, 0\}$ for each 
$x\in [0,\infty )$ ($\log 0:= -\infty $).

\end{df}
Regarding the following lemma we remark that 
if $\Gamma $ is a nonempty compact subset of $X^{+}$, then 
there exist a number $R>1$ and a number $\rho >1$ such that 
each $h\in \Gamma $ satisfies condition (A) for $(R, \rho ).$ 
\begin{lem}
\label{l:greenvr}
Let $\Gamma $ be a nonempty compact subset of 
$X^{+}.$ Let $R>1, \rho >1$ be two numbers such that 
each $h\in \Gamma $ satisfies condition (A) for 
$(R, \rho ).$ 
Then, we have the following. 
\begin{itemize}
\item[{\em (1)}] 
\begin{itemize}
\item[{\em (i)}] 
For each $\gamma =(\gamma _{j})_{j\in \ZZ}\in \Gamma ^{\ZZ}$, 
there exists a function 
${\cal G}_{\gamma }^{+}(x,y)$ on $\Ct$ such that 
${\cal G}_{\gamma ,n}^{+}(x,y)\rightarrow {\cal G}_{\gamma }^{+}(x,y)$ 
as $n\rightarrow \infty $ on $\Ct.$ 
\item[{\em (ii)}] 
We have 
${\cal G}_{\gamma ,n}^{+}(x,y)\rightarrow 
{\cal G}_{\gamma }^{+}(x,y) \mbox{ as }n\rightarrow \infty $ 
uniformly on $\overline{V_{R}^{+}}\times \Gamma ^{\ZZ}.$ 
Here, $\overline{V_{R}^{+}}$ denotes the closure 
of $V_{R}^{+}$ in $\Ct.$ 
\item[{\em (iii)}] 
The function ${\cal G}_{\gamma }^{+}$ is pluriharmonic on 
$I_{\gamma }^{+}$ and 
$(x,y,\gamma )\in \overline{V_{R}^{+}}\times \Gamma ^{\ZZ}\mapsto 
{\cal G}_{\gamma }^{+}(x,y)$ is continuous 
on $\overline{V_{R}^{+}} \times \Gamma ^{\ZZ}.$  
\item[{\em (iv)}] 
For each $\gamma =(\gamma _{j})_{j\in \ZZ}\in \Gamma ^{\ZZ}$ and for each 
$(x,y)\in \Ct$, we have 
$\frac{1}{\deg (\gamma _{0})}{\cal G}_{\sigma (\gamma )}^{+}
(\gamma _{0}(x,y))={\cal G}_{\gamma }^{+}(x,y).$ 
For each $\gamma =(\gamma _{j})_{j\in \ZZ}\in \Gamma ^{\ZZ}$, 
there exist a 
constant $r_{\gamma }^{+}\in \RR $ and a 
bounded pluriharmonic function $u_{\gamma }^{+}$ on 
$V_{R}^{+}$ 
such that ${\cal G}_{\gamma }^{+}(x,y)=\log |y|+r_{\gamma }^{+}
+u_{\gamma }^{+}(x,y)$ on $V_{R}^{+}$ and such that 
$u_{\gamma }^{+}(x,y)\rightarrow 0$ as $|y|\rightarrow \infty $. 
Moreover, 
$\gamma \in \Gamma ^{\ZZ}\mapsto r_{\gamma }^{+}$ is 
continuous on $\Gamma ^{\ZZ}.$ Furthermore, 
$(x,y,\gamma )\in \overline{V_{R}^{+}} \times \Gamma ^{\ZZ}
\mapsto u_{\gamma }^{+}(x,y)$ is continuous on 
$\overline{V_{R}^{+}} \times \Gamma ^{\ZZ}.$    
 \end{itemize}
 \item[{\em (2)}] 
\begin{itemize}
\item[{\em (i)}] 
For each $\gamma =(\gamma _{j})_{j\in \ZZ}\in \Gamma ^{\ZZ}$
there exists a function 
${\cal G}_{\gamma }^{-}(x,y)$ on $\Ct$ such that 
${\cal G}_{\gamma ,n}^{-}(x,y)\rightarrow {\cal G}_{\gamma }^{-}(x,y)$ 
as $n\rightarrow \infty $ on $\Ct.$ 
\item[{\em (ii)}] 
We have 
${\cal G}_{\gamma ,n}^{-}(x,y)\rightarrow 
{\cal G}_{\gamma }^{-}(x,y) \mbox{ as }n\rightarrow \infty $ 
uniformly on $\overline{V_{R}^{-}}\times \Gamma ^{\ZZ}.$ 
Here, $\overline{V_{R}^{-}}$ denotes the 
closure of $V_{R}^{-}$ in $\Ct.$ 
\item[{\em (iii)}] 
The function ${\cal G}_{\gamma }^{-}$ is pluriharmonic on 
$I_{\gamma }^{-}$ and 
$(x,y,\gamma )\in \overline{V_{R}^{-}}\times \Gamma ^{\ZZ}\mapsto 
{\cal G}_{\gamma }^{-}(x,y)$ is continuous 
on $\overline{V_{R}^{-}} \times \Gamma ^{\ZZ}.$  
\item[{\em (iv)}] 
For each $\gamma =(\gamma _{j})_{j\in \ZZ}\in \Gamma ^{\ZZ}$ and for each 
$(x,y)\in \Ct$, we have 
$\frac{1}{\deg (\gamma _{-1})}{\cal G}_{\sigma ^{-1}(\gamma )}^{-}
(\gamma _{-1}(x,y))={\cal G}_{\gamma }^{-}(x,y).$ 
For each $\gamma =(\gamma _{j})_{j\in \ZZ}\in 
\Gamma ^{\ZZ}$, there exist a 
constant $r_{\gamma }^{-}\in \RR $ and a 
bounded pluriharmonic function $u_{\gamma }^{-}$ on 
$V_{R}^{-}$ 
such that ${\cal G}_{\gamma }^{-}(x,y)=\log |x|+r_{\gamma }^{-}
+u_{\gamma }^{-}(x,y)$ on $V_{R}^{-}$ and such that 
$u_{\gamma }^{-}(x,y)\rightarrow 0$ as $|x|\rightarrow \infty $. 
Moreover, 
$\gamma \in \Gamma ^{\ZZ}\mapsto r_{\gamma }^{-}$ is 
continuous on $\Gamma ^{\ZZ}.$ Furthermore, 
$(x,y,\gamma )\in \overline{V_{R}^{-}} \times \Gamma ^{\ZZ}
\mapsto u_{\gamma }^{-}(x,y)$ is continuous on 
$\overline{V_{R}^{-}} \times \Gamma ^{\ZZ}.$    
 \end{itemize}
 \end{itemize} 
\end{lem}
\begin{proof}
For each $h\in \Gamma $ with 
$h(x,y)=(y+\alpha, p(y)-\delta x)$, 
where $\alpha \in \CC, \delta \in \CC $ with 
$\delta \neq 0$, $p(y)$ is a polynomial of $y$ with 
$\deg (p)\geq 2$, 
 let $c(h)\in \CC $ be the 
coefficient of the highest degree term of $p(y).$ 
Since each $h\in \Gamma $ satisfies condition (A) for 
$(R,\rho )$, 
there exist  
constants $a_{1}>0, a_{2}>0$ such that 
for each $h\in \Gamma $ and for each 
$(x,y)\in \overline{V_{R}^{+}}$, 
\begin{equation}
\label{eq:a1ydh}
a_{1}|y|^{d(h)}\leq |\pi _{2}(h(x,y))|\leq a_{2}|y|^{d(h)}
\end{equation}
where $d(h)=\deg (h).$ 
By using (\ref{eq:a1ydh}), it is easy to see that for each 
$(\gamma _{j})_{j\in \ZZ}\in \Gamma ^{\ZZ}$, for each 
$(x,y)\in \overline{V_{R}^{+}}$ and 
for each $n\in \NN$, 
\begin{eqnarray}
\label{eqn:d0dn-11}
 &  & (\frac{1}{d_{0}\cdots d_{n-1}}+\frac{1}{d_{0}\cdots d_{n-2}}+
\cdots +\frac{1}{d_{0}})\log a_{1}+\log |y|\\ 
 & \leq & 
 \frac{1}{d_{0}\cdots d_{n-1}}\log ^{+}|\pi _{2}(\gamma _{n-1,0}(x,y))| 
 =\frac{1}{d_{0}\cdots d_{n-1}}\log | \pi _{2}(\gamma _{n-1,0}(x,y))|\\ 
 \label{eqn:d0dn-13}
 & \leq & 
  (\frac{1}{d_{0}\cdots d_{n-1}}+\frac{1}{d_{0}\cdots d_{n-2}}+
\cdots +\frac{1}{d_{0}})\log a_{2}+\log |y|
\end{eqnarray} 
where $d_{j}:=\deg (\gamma _{j})\geq 2$ for each $j=0,\ldots, n-1.$  
Moreover,  
for each $(\gamma _{j})_{j\in \ZZ}\in \Gamma ^{\ZZ}, $ 
for each $(x,y)\in \overline{V_{R}^{+}}$ and for each $n\in \NN$, since $\gamma _{n-1,0}(x,y)\in \overline{V_{R}^{+}}$, 
we have 
\begin{equation}
\label{eq:gapi2ga}
1\leq \frac{\| \gamma _{n-1,0}(x,y)\| }
{|\pi _{2}(\gamma _{n-1,0}(x,y))|}\leq \sqrt{2} 
\mbox{ for each }(x,y,\gamma,n)\in \overline{V_{R}^{+}}\times 
\Gamma ^{\ZZ}\times \NN.
\end{equation}  
From these arguments, it follows that 
for each $\gamma =(\gamma _{j})_{j\in \ZZ}\in \Gamma ^{\ZZ}$
there exists a pluriharmonic function 
${\cal G}_{\gamma }^{+}(x,y)$ on $V_{R}^{+}  $ 
such that 
\begin{equation}
\label{eq:cGgamman}
{\cal G}_{\gamma ,n}^{+}(x,y)\rightarrow 
{\cal G}_{\gamma }^{+}(x,y) \mbox{ as }n\rightarrow \infty 
\mbox{ uniformly on } \overline{V_{R}^{+}}\times \Gamma ^{\ZZ}.
\end{equation} 
Therefore, as a sequence of functions of $(x,y)$, 
$\{ {\cal G}_{\gamma ,n}^{+}(x,y)\}_{n=1}^{\infty }$ converges 
 as $n\rightarrow \infty $
uniformly on each compact subset of $I_{\gamma }^{+}.$ 
We set ${\cal G}_{\gamma }^{+}(x,y)=\lim _{n\rightarrow \infty }
{\cal G}_{\gamma ,n}^{+}(x,y)$ for each $(x,y)\in I_{\gamma }^{+}.$  
Then, ${\cal G}_{\gamma }^{+}$ is pluriharmonic on $I_{\gamma }^{+}$ 
for each $\gamma \in \Gamma ^{\ZZ}.$ 
Moreover, setting ${\cal G}_{\gamma }^{+}(x,y)=0$ on 
$K_{\gamma }^{+}$ for each $\gamma \in \Gamma ^{\ZZ}$, 
we have ${\cal G}_{\gamma ,n}^{+}\rightarrow {\cal G}_{\gamma }^{+}$ 
as $n\rightarrow \infty $ on $K_{\gamma }^{+}.$   
Then, ${\cal G}_{\gamma ,n}^{+}(x,y)
\rightarrow {\cal G}_{\gamma }^{+}(x,y)$ as $n\rightarrow 
\infty $ on $\Ct$ for each $\gamma \in \Gamma ^{\ZZ}.$ 
Hence, $\frac{1}{\deg (\gamma _{0})}{\cal G}_{\sigma (\gamma )}^{+}
(\gamma _{0}(x,y))={\cal G}_{\gamma }^{+}(x,y)$ for each 
$(x,y,\gamma )\in \Ct \times \Gamma ^{\ZZ}.$  
By (\ref{eqn:d0dn-11})--(\ref{eqn:d0dn-13}), (\ref{eq:gapi2ga}),  (\ref{eq:cGgamman}), and \cite[an argument in p283]{MNTT00}, 
 statements (1)(iii)(iv) follow.   
Hence, statement (1) follows. 
We can prove statement (2) similarly. 

Thus, we have proved our lemma. 
\end{proof}
We now prove that if $\Gamma $ is a nonempty 
compact subset of $X^{+}$, then 
${\cal G}_{\gamma ,n}(x,y)$, which is regarded as a 
function of $((x,y), \gamma )\in \Ct \times \Gamma ^{\NN}$, 
tends to ${\cal G}_{\gamma }^{+}(x,y)$ as $n\rightarrow \infty $ 
uniformly on any compact subset of $\Ct \times \Gamma ^{\NN}.$ 
\begin{lem}
\label{l:Gganpuc}
Let $\Gamma $ be a nonempty compact subset of 
$X^{+}.$ Let $R>1, \rho >1$ be two numbers such that 
each $h\in \Gamma $ satisfies condition (A) for 
$(R, \rho ).$ 
Then 
for each $A>R$, 
${\cal G}_{\gamma ,n }^{+}(x,y)\rightarrow {\cal G}_{\gamma }^{+}(x,y)
\mbox{ as } n\rightarrow \infty $ uniformly on 
$C^{A}\times \Gamma ^{\ZZ}$, where 
$C^{A}:=\{ (x,y)\in \Ct \mid |x|\leq A\} \cap \overline{D_{R}\cup V_{R}^{-}}.$ In particular, 
${\cal G}_{\gamma, n}^{+}(x,y)\rightarrow {\cal G}_{\gamma }^{+}(x,y) 
\mbox{ as }n\rightarrow \infty $ uniformly on any compact subset of 
$\Ct \times \Gamma ^{\ZZ}.$ 

\end{lem}
\begin{proof}
Let $A>R.$ 
For each $\gamma =(\gamma _{j})_{j\in \ZZ}
\in \Gamma ^{\ZZ}$ and for each 
$n\geq 2$, let 
$C_{n, \gamma }^{A}=C^{A}\cap 
(\gamma _{n-1,0}^{-1}(V_{R}^{+})\setminus \gamma _{n-2,0}^{-1}(V_{R}^{+})).$ Also,  
let $C_{1,\gamma }^{A}=C^{A}\cap \gamma _{0}^{-1}(V_{R}^{+}).$ 
Note that 
\begin{equation}
\label{eq:Igammapn0}
I_{\gamma }^{+}=
\cup _{n=0}^{\infty }\gamma _{n,0}^{-1}(V_{R}^{+}) \mbox{ and } 
\gamma _{n-2,0}^{-1}(V_{R}^{+})\subset \gamma _{n-1,0}^{-1}
(V_{R}^{+}) \mbox{ for each } n\in \NN \mbox{ with }n\geq 2.
\end{equation}
 We prove the following claim. 

Claim 1. $\lim _{n\rightarrow \infty }\sup _{\gamma \in \Gamma ^{\ZZ}} \sup _{(x,y)\in C_{n,\gamma }^{A}}\sup _{m\geq n}
{\cal G}_{\gamma ,m}^{+}(x,y)=0.$ 

To prove this claim, let $\gamma \in \Gamma ^{\ZZ}, n\in \NN$ 
with $n\geq 2$,  
and $z\in C_{n,\gamma }^{A}.$ 
Then, 
\begin{equation}
\label{eq:gaj0zca}
\gamma _{j,0}(z)\in C^{A} \mbox{ for each } j=0,\ldots, n-2. 
\end{equation}
For, since $z\in C_{n,\gamma }^{A}$, we have 
$\gamma _{n-1,0}(z)\in V_{R}^{+} $ and 
$\gamma _{n-2,0}(z)\in \overline{D_{R}\cup V_{R}^{-}}.$
Suppose that $\gamma _{0,0}(z)=\gamma _{0}(z)\in \Ct \setminus C^{A}.$ 
Then, $\gamma _{0}(z)\in V_{R}^{+}\cup (V_{R}^{-}\setminus C^{A}).$ 
We consider the following two cases.\\ 
 Case 1. $\gamma _{0}(z)\in V_{R}^{+}$.\\ 
 Case 2. $\gamma _{0}(z)\in V_{R}^{-}\setminus C^{A}.$\\ 
 Suppose we have Case 1. Then, 
 since each $h\in \Gamma $ satisfies condition (A) for $(R,\rho )$, 
we have  $\gamma _{n-2,0}(z)\in 
 V_{R}^{+}.$ However, this contradicts that $\gamma _{n-2,0}(z)\in 
 \overline{D_{R}\cup V_{R}^{-}}.$ \\ 
 Suppose we have Case 2. Then, $\gamma _{0}(z)\in 
 V_{R}^{-}\setminus C^{A}$, 
 $\gamma _{0}^{-1}(\gamma _{0}(z))=z$,  and 
 $\gamma _{0}^{-1}(V_{R}^{-1}\setminus C^{A})\subset 
 V_{R}^{-}\setminus C^{A}$ since each $h\in \Gamma $ satisfies 
 condition (A) for $(R,\rho ).$ It follows that 
 $z\in V_{R}^{-}\setminus C^{A}.$ However, this contradicts 
 $z\in C_{n,\gamma }^{A}\subset C^{A}.$ 
Hence, we should have $\gamma _{0,0}(z)=\gamma _{0}(z)\in C^{A}.$ 
Repeating the above arguments, we can show that 
(\ref{eq:gaj0zca}) holds. 

By (\ref{eq:gaj0zca}), there exists a compact set $L$ in 
$\overline{V_{R}^{+}}$  such that 
 for each $n\in \NN, \gamma \in \Gamma ^{\ZZ}$ and 
 $z\in C_{n,\gamma }^{A},$  we have 
 $\gamma _{n-1,0}(z)\in L.$ 
 Combining this with Lemma~\ref{l:greenvr}, 
 it follows that there exists a constant 
 $S$ (depending only on $\Gamma $ and $A$) such that 
 for each $\gamma \in \Gamma ^{\ZZ}, $ for each 
 $m,n\in \NN $ with $m\geq n$, and for each 
 $z\in C_{n,\gamma }^{A}$, we have 
$$
 {\cal G}_{\gamma, m}^{+}(z)
  =  \frac{1}{\deg (\gamma_{n-1,0})}
 {\cal G}_{\sigma ^{n}(\gamma ), m-n}(\gamma _{n-1,0}(z))
 \leq  \frac{1}{2^{n}}S $$ 
 where we set ${\cal G}_{\rho , 0}(x,y):=\log ^{+}\| (x,y)\|$ 
 for each $\rho \in \Gamma ^{\ZZ}$ and $(x,y)\in \Ct.$ 
 Therefore, Claim 1 holds. 

Let $\ve >0$. By Claim 1, there exists a number 
$n_{0}\in \NN $ such that for each 
$n\in \NN $ with $n\geq n_{0}$, we have 
\begin{equation}
\label{eq:sssGge}
\sup _{\gamma \in \Gamma ^{\ZZ} }
\sup _{(x,y)\in C_{n,\gamma }^{A}}
\sup _{m\geq n}{\cal G}_{\gamma, m}^{+}(x,y)\leq \ve.
\end{equation}
Hence, for each $n\geq n_{0}$, 
\begin{equation}
\label{eq:ssGge}
\sup_{\gamma \in \Gamma ^{\ZZ}} \sup _{(x,y)\in C_{n,\gamma }^{A}}
{\cal G}_{\gamma }^{+}(x,y)\leq \ve.
\end{equation}
Also, we may assume 
\begin{equation}
\label{eq:12n0log}
\frac{1}{2^{n_{0}}}\log ^{+}\sup _{z\in C^{A},h\in \Gamma }\| h(z)\| <\ve .
\end{equation}
Moreover, we have 
\begin{equation}
\label{eq:n0j1ugj}
\cup _{j=1}^{n_{0}}\{ \gamma _{j-1, 0}(x,y) 
\mid \gamma \in \Gamma ^{\ZZ}, (x,y)\in C_{j,\gamma }^{A}\} 
\subset \overline{V_{R}^{+}}.
\end{equation}
Furthermore, by Lemma~\ref{l:greenvr}, 
there exists a number $n_{1}\in \NN $ such that 
for each $n\geq n_{1}$, for each $\gamma \in \Gamma ^{\ZZ}$ 
and for each $(x,y)\in \overline{V_{R}^{+}}$, 
\begin{equation}
\label{eq:Ggnpxymi}
|{\cal G}_{\gamma, n}^{+}(x,y)-{\cal G}_{\gamma }^{+}(x,y)|\leq \ve.
\end{equation}
It follows that for each $\gamma \in \Gamma ^{\ZZ}$, 
for each $n\geq n_{0}+n_{1}$, for each $j=1,\ldots, n_{0}$ and for each 
$(x,y)\in C_{j,\gamma }^{A}$, 
\begin{eqnarray}
 & & |{\cal G}_{\gamma, n}^{+}(x,y)-{\cal G}_{\gamma }^{+}(x,y)| \nonumber \\ 
& = & |\frac{1}{\deg (\gamma _{n_{0}-1,0})}
{\cal G}^{+}_{\sigma ^{n-n_{0}}(\gamma ), n-n_{0}}
(\gamma _{n_{0}-1,0}(x,y))-
\frac{1}{\deg (\gamma _{n_{0}-1,0})}
{\cal G}^{+}_{\sigma ^{n-n_{0}}(\gamma )}
(\gamma _{n_{0}-1,0}(x,y))| \nonumber \\ 
& \leq & 
\frac{1}{\deg (\gamma _{n_{0}-1,0})}\ve <\ve . \label{eq:1deggn0ve}
\end{eqnarray} 
We now let $\gamma \in \Gamma ^{\ZZ}, n\geq n_{0}+n_{1}, 
j\geq n_{0}, $ and $(x,y)\in C_{j, \gamma }^{A}.$ 
We consider the following two cases. 

Case 1. $n\geq j.$ Case 2. $n<j.$ 

Suppose we have Case 1. Then, 
by (\ref{eq:sssGge})(\ref{eq:ssGge}), we have 
$|{\cal G}_{\gamma, n}^{+}(x,y)-{\cal G}_{\gamma }^{+}(x,y)|\leq \ve.$ 

Suppose we have Case 2. Then, by 
(\ref{eq:gaj0zca}), we have $\gamma _{n-2,0}(x,y)\in C^{A}.$ 
Therefore, by (\ref{eq:12n0log}), we obtain that 
$$
{\cal G}_{\gamma, n}^{+}(x,y)=\frac{1}{\deg (\gamma _{n-1,0})}
\log ^{+}\| \gamma _{n-1,0}(x,y)\| 
\leq \frac{1}{2^{n_{0}+n_{1}}}\log ^{+}\sup _{z\in C^{A}, h\in \Gamma }\| h(z)\| <\ve.$$
Combining this with  
(\ref{eq:ssGge}), 
we see that 
$$|{\cal G}_{\gamma, n}^{+}(x,y)-{\cal G}_{\gamma }^{+}(x,y)|\leq \ve.$$
From these arguments, it follows that 
for each $\gamma \in \Gamma ^{\ZZ}$, for each 
$n\geq n_{0}+n_{1}$, for each $j\geq n_{0}$ and for each 
$(x,y)\in C_{j,\gamma }^{A}$, 
\begin{equation}
\label{eq:GgnpGgpe}
|{\cal G}_{\gamma, n}^{+}(x,y)-{\cal G}_{\gamma }^{+}(x,y)|\leq \ve.
\end{equation}
Combining Lemma~\ref{l:igkg}, (\ref{eq:Igammapn0}), 
(\ref{eq:1deggn0ve}), and 
(\ref{eq:GgnpGgpe}), we obtain that 
${\cal G}_{\gamma, n}^{+}(x,y)\rightarrow {\cal G}_{\gamma }^{+}(x,y)$ 
as $n\rightarrow \infty$ uniformly on $C^{A}\times \Gamma ^{\ZZ}.$ 
Thus, we have proved our lemma. 
\end{proof}
Similarly, we can prove the following. 
\begin{lem}
\label{l:Gganmuc}
Let $\Gamma $ be a nonempty compact subset of 
$X^{+}.$ Let $R>1, \rho >1$ be two numbers such that 
each $h\in \Gamma $ satisfies condition (A) for 
$(R, \rho ).$ 
Then,  
for each $A>R$, 
${\cal G}_{\gamma ,n }^{-}(x,y)\rightarrow {\cal G}_{\gamma }^{-}(x,y)
\mbox{ as } n\rightarrow \infty $ uniformly on 
$C^{A,-}\times \Gamma ^{\ZZ}$, where 
$C^{A,-}:=\{ (x,y)\in \Ct \mid |y|\leq A\} \cap \overline{D_{R}\cup V_{R}^{+}}.$ In particular, 
${\cal G}_{\gamma, n}^{-}(x,y)\rightarrow {\cal G}_{\gamma }^{-}(x,y) 
\mbox{ as }n\rightarrow \infty $ uniformly on any compact subset of 
$\Ct \times \Gamma ^{\ZZ}.$ 
\end{lem}
We now prove the pluriharmonicity of ${\cal G}_{\gamma }^{\pm }$ 
in $\Ct$, that $I_{\gamma }^{+}$ (resp. $I_{\gamma }^{-}$) 
is equal to the set of 
points where ${\cal G}_{\gamma }^{+}$ (resp. ${\cal G}_{\gamma }^{-}$) is not pluriharmonic, 
that $J_{\gamma }^{\pm }$ is uncountable, and that 
the map $\gamma \in \Gamma ^{\ZZ}\mapsto 
J_{\gamma }^{\pm}$ is lower semicontinuous if $\Gamma $ is a 
compact subset of $X^{\pm }.$  
\begin{prop}
\label{p:Ggcontiph}
Let $\Gamma $ be a nonempty compact subset of 
$X^{+}.$ 
We endow $\Gamma ^{\ZZ}$ with the product topology. 
Then, we have the following. 
\begin{itemize}
\item[{\em (i)}] 
$(x,y,\gamma )\in \Ct \times \Gamma ^{\ZZ}\rightarrow 
{\cal G}_{\gamma }^{\pm }(x,y)$ is continuous on 
$\Ct\times \Gamma ^{\ZZ}.$ 
\item[{\em (ii)}] 
If $\rho \in \Gamma ^{\ZZ}$ tends to $\gamma \in \Gamma ^{\ZZ}$ 
in $\Gamma ^{\ZZ}$, then 
${\cal G}_{\rho }^{\pm }(x,y)\rightarrow {\cal G}_{\gamma }^{\pm}(x,y)$ 
locally uniformly on $\Ct.$ 
\item[{\em (iii)}] 
For each $\gamma \in \Gamma ^{\ZZ}$, 
${\cal G}_{\gamma }^{\pm }(x,y)$ is plurisubharmonic on $\Ct.$ 
\item[{\em (iv)}] 
For each $\gamma \in \Gamma ^{\ZZ}$, 
${\cal G}_{\gamma }^{\pm }(x,y)$ is positive on $I^{\pm}_{\gamma }$ 
and is equal to zero on $K_{\gamma }^{\pm}.$ 
\item[{\em (v)}] 
Let $\gamma \in \Gamma ^{\ZZ}$. Then, ${\cal G}_{\gamma }^{\pm}$ 
is pluriharmonic on $I_{\gamma }^{\pm }$ and 
is not pluriharmonic on any neighborhood of any point 
of $J_{\gamma }^{\pm }.$ 
A point $z_{0}\in \Ct$ belongs to $J_{\gamma }^{\pm }$ if and only 
if ${\cal G}_{\gamma }^{\pm }$ is not pluriharmonic in any neighborhood of $z_{0}.$ 
\item[{\em (vi)}]
There exists a number $R_{0}>0$ such that 
for each $\gamma \in \Gamma ^{\ZZ}$, 
$K_{\gamma }:=K_{\gamma }^{+}\cap K_{\gamma }^{-}$ is a 
nonempty compact subset of $\overline{D_{R_{0}}}.$ 
\item[{\em (vii)}] 
For each $\gamma \in \Gamma ^{\ZZ}$, 
$J_{\gamma }^{\pm }$ has no isolated point.  
In particular, $J_{\gamma }^{\pm }$ is uncountable. 
\item[{\em (viii)}] 
 The map 
$\gamma \in \Gamma ^{\ZZ}\mapsto 
J_{\gamma }^{\pm }$ is lower semicontinuous, i.e., 
for any $\gamma \in \Gamma ^{\ZZ}$, 
for any $z_{0}\in J_{\gamma }^{\pm }$, and 
for any sequence $(\zeta ^{n})_{n=1}^{\infty }$ in 
$\Gamma ^{\ZZ}$ with $\zeta ^{n}\rightarrow \gamma $ as 
$n\rightarrow \infty $ in $\Gamma ^{\ZZ}$, there exists a sequence 
$(z_{n})_{n=1}^{\infty }$ in $\Ct $ with 
$z_{n}\in J_{\gamma }^{\pm }$ for each $n\in \NN$ such that 
$z_{n}\rightarrow z_{0} $ as $n\rightarrow \infty .$   
\end{itemize}
\end{prop}
\begin{proof}
Statements (i)(ii)(iii) follow from Lemmas~\ref{l:wrconv}, \ref{l:Gganpuc}, and \ref{l:Gganmuc}. 
Statement (iv) follows from Lemmas~\ref{l:wrconv}, \ref{l:greenvr}(1)(iv) (2)(iv). 

We now prove statement (v). Let $\gamma \in \Gamma ^{\ZZ}.$ 
By Lemmas~\ref{l:wrconv} and \ref{l:greenvr}(1)(iii)(2)(iii), we have that 
${\cal G}_{\gamma }^{\pm}$ is pluriharmonic on 
$I_{\gamma }^{\pm}.$  
Let $z_{0}\in J_{\gamma }^{\pm }.$ 
Then, statement (iv) implies that ${\cal G}_{\gamma }^{\pm }$ 
takes its minimal value at $z_{0}$.  
Hence, if there exists a neighborhood 
of $U$ of $z_{0}$ such that ${\cal G}_{\gamma }^{\pm }$ is 
pluriharmonic, then it contradicts the minimum principle of 
pluriharmonic functions. Thus, for each neighborhood $U$ 
of $z_{0}$, ${\cal G}_{\gamma }^{\pm }$ cannot be 
pluriharmonic on $U.$ Moreover, statement (iv) implies that  
${\cal G}_{\gamma }^{\pm }$ is pluriharmonic 
on $\mbox{int}(K_{\gamma }^{\pm})$. 
It follows that for any $z_{0}\in \Ct$, 
$z_{0}\in J_{\gamma }^{\pm }$ if and only if 
${\cal G}_{\gamma }^{\pm }$ is not pluriharmonic on 
any neighborhood of $z_{0}.$   

We now prove statement (vi). 
By Lemma~\ref{l:wrconv}, there exist a number $R_{0}>0$ and 
a number $\rho >1$ such that each 
$h\in \Gamma $ satisfies condition (A) for $(R_{0}, \rho ).$ 
Then, for each cyclic $\zeta \in \Gamma ^{\ZZ}$, 
$K_{\zeta }$ is a nonempty compact subset of 
$\overline{D_{R_{0}}}.$ Since for each $\gamma \in \Gamma ^{\ZZ}$, 
there exists a sequence of cyclic elements of $\Gamma ^{\ZZ}$ 
that converges to $\gamma $, statement (i) implies 
$K_{\gamma }\cap \overline{D_{R_{0}}}\neq \emptyset .$ 
Moreover, since each $h\in \Gamma $ satisfies condition 
(A) for $(R_{0}, \rho )$, we have that 
$K_{\gamma }\subset \overline{D_{R_{0}}}.$ Hence, statement (vi) holds. 

Statement (vii) follows from statements (v)(vi). 
Statement (viii) follows from statements (ii)(v). 

Thus, we have proved our proposition. 
\end{proof}
\section{Further results and the proofs of the main results}
\label{s:fratp}
In this section, we give further detailed results 
and we prove them and the main results in this paper.
We first give some definitions. 
\begin{df}
\begin{itemize}
\item[(1)] 
Let $(\Omega, d_{\Omega })$ be a metric space. 
Let $G$ be a semigroup  of continuous maps on $\Omega$, where 
the semigroup operation is the composition of maps. 
We denote by $F(G)$ the set of points $z\in \Omega $ 
for which there exists an open neighborhood $U$ of $z$ in $\Omega $ 
such that $G$ is equicontinuous on $U$ as a family of maps 
from $U$ to $\Omega$ with respect to the distance $d_{\Omega }.$ 
The set $F(G)$ is called the Fatou set of $G.$ 
We set $J(G)=\Omega \setminus  F(G).$ This is called the Julia set of $G.$ 
Also, we set $J_{\ker }(G)=\cap _{g\in G}g^{-1}(J(G)).$ 
This is called the kernel Julia set of $G$. 
Moreover, for each subset $A$ of $\Omega, $ we set 
$G(A)=\cup _{g\in G}g(A)$ and for each $z\in \Omega, $ 
we set $G(z)=\cup _{g\in G}\{ g(z)\}.$  
\item[(2)] 
Let $G$ be a semigroup generated by a subset of $X^{+}$ 
(resp. $X^{-}$)  
where the semigroup operation is the composition of maps on 
$Y:=\Pt \setminus \{ [1:0:0]\}$ (resp. 
$Y^{-}=\Pt \setminus \{ [0:1:0]\}$).   
We regard $G$ as a semigroup of self-maps on the metric space $(Y, d)$ (resp. $(Y^{-}, d)$),  
where $d$ denotes the distance induced by the Fubini--Study metric on $\Pt$, and 
$F(G), J(G), J_{\ker }(G)$ are considered for $G$ as in (1). 

\end{itemize}
\end{df}

\begin{rem}
\label{r:ifgsgx}
If $G$ is a semigroup generated by a subset of $X^{+}$
 (resp. $X^{-}$),  
then it is easy to see that $G(F(G))\subset F(G)$ 
since each $g\in X^{+}$ (resp. $X^{-}$) is an open map on 
$\Pt \setminus \{ [1:0:0]\}$ (resp. $\Pt \setminus \{ [0:1:0]\}). $ 
\end{rem}
\begin{df}
For any topological space $Y$, 
we denote by $\mbox{CM}(Y)$ the space of all 
continuous maps on $Y$ endowed with 
the compact-open topology. 
Note that $\mbox{CM}(Y)$ is a semigroup 
where the semigroup 
operation is the composition of maps.  
\end{df}
\begin{df}
For a topological space $Y$, we denote by Cpt$(Y)$ the space of all nonempty compact subsets of $Y$. 
If $Y$ is a metric space, we endow Cpt$(Y)$  
with the Hausdorff metric.
\end{df}
\begin{df}
\label{d:minimal}
Let $Y$ be a metric space and let $G$ be a subsemigroup of $\CMX .$ Let $K\in \Cpt(Y).$ 
We say that $K$ is a minimal set for $(G,Y)$ if 
$K$ is minimal among the space 
$\{ L\in \Cpt(Y)\mid G(L)\subset L\} $ with respect to inclusion. 
Moreover, we set $\Min(G,Y):= \{ K\in \Cpt(Y)\mid K \mbox{ is a minimal set for } (G,Y)\} .$  
\end{df}
\begin{rem}
\label{r:minimal}
Let $Y$ be a metric space and let $G$ be a subsemigroup of $\CMX .$
By Zorn's lemma, it is easy to see that 
if $K_{1}\in \Cpt(Y)$ and $G(K_{1})\subset K_{1}$, then 
there exists a $K\in \Min(G,Y)$ with $K\subset K_{1}.$  
Moreover, it is easy to see that 
for each $K\in \Cpt(Y)$, 
$K\in \Min(G,Y)$ if and only if 
$\overline{G(z)}=K$ for each $z\in K.$ 
In particular, if $K_{1},K_{2}\in \Min(G,Y)$ with $K_{1}\neq K_{2}$, then 
$K_{1}\cap K_{2}=\emptyset .$ 
Furthermore, it is easy to see that 
if $\Gamma \in \Cpt(\CMX), 
G=\{ h_{1}\circ \cdots \circ h_{n}\mid n\in \NN, 
h_{j}\in \Gamma (\forall j)\}$ and 
$K\in \Min(G,Y)$, then $K=\bigcup _{h\in \Gamma }h(K).$   
\end{rem}

\begin{df}
Let $\tau \in {\frak M}_{1}(X^{\pm}).$ We set 
$\tau ^{\NN }:= \otimes _{j=1}^{\infty }\tau \in {\frak M}_{1}((X^{\pm })^{\NN})$ and 
$\tau ^{\ZZ}:=\otimes _{j=-\infty }^{\infty }\tau \in {\frak M}_{1}((X^{\pm })^{\ZZ}).$ 

\end{df}
We provide a sufficient condition for $\tau \in {\frak M}_{1,c}(X^{+})$ 
to be mean stable on $\Pt \setminus \{ [1:0:0]\}.$ 
\begin{prop}
\label{p:intsupdlarge}
Let $\tau \in {\frak M}_{1,c}(X^{+}).$ Suppose that the following conditions {\em (i)} and {\em (ii)} are satisfied.
\begin{itemize}
\item[{\em (i)}]
For each $z\in \Ct$, {\em Leb}$_{4}(\{ h(z)\mid h\in \supp\,\tau \})>0$.
\item[{\em (ii)}] There exists an element $g\in \supp\,\tau $ 
such that $\delta (g)>1.$ 
\end{itemize}
Then, we have the following. 
\begin{itemize}
\item[{\em (a)}]For each $z\in \Ct$, there exists a Borel subset 
$A_{z}$ of $(\supp\,\tau )^{\NN}$ with $
\tau ^{\NN}(A_{z})=1$ such that 
for each $\gamma =(\gamma _{j})_{j=1}^{\infty }\in A_{z}$, 
$\gamma _{n,1}(z)\rightarrow [0:1:0] \mbox{ as }n\rightarrow \infty .$

\item[{\em (b)}]
$\tau $ is mean stable on 
$\Pt \setminus \{ [1:0:0]\}.$ 

\end{itemize} 

\end{prop}
\begin{proof}
We first prove the following claim. 

Claim 1. There exists no minimal set $L$ of $\tau 
$ with $L\neq \{ [0:1:0]\}.$ 

To prove this claim, 
suppose that there exists a minimal set $L$ of $\tau $ with 
$L\neq \{ [0:1:0]\} .$ 
Then, by Lemmas~\ref{l:ghsm} and \ref{l:wrconv}, we have $L\subset \Ct.$ 
Since $L$ is compact in $\Pt \setminus \{ [1:0:0]\}$, it follows that 
$L$ is a compact subset of $\Ct.$   
Combining this with  (i) in the assumptions of our proposition, 
we obtain $0<\mbox{Leb}_{4}(L)<\infty .$ 
Let $g\in \supp\,\tau $ be an element with 
$\delta (g)>1.$ 
Since $g(L)\subset L$, it follows that   
$\delta (g)^{2}\mbox{Leb}_{4}(L)=\mbox{Leb}_{4}(g(L))\leq \mbox{Leb}_{4}(L).$ However, this contradicts $0<\mbox{Leb}_{4}(L)<\infty $ since 
$\delta (g)>1.$ 
Therefore, we have proved Claim 1. 

By Lemma~\ref{l:ghsm}(1) and the assumption that 
$\supp\,\tau $ is compact, 
there exists an open neighborhood $W$ of $
[0:1:0]$ in $\Pt $ and a constant $c\in (0,1)$ 
such that for each $h\in \supp\,\tau$ and for each 
$x,y\in W$, we have  
$h(W)\subset W$ and $d(h(x), h(y))\leq cd(x,y)$ where 
$d$ denotes the distance induced from the Fubini--Study metric 
on $\Pt.$ In particular, 
\begin{equation}
\label{eq:dgn1z0}
d(\gamma _{n,1}(z), [0:1:0])\rightarrow 0 \mbox{ as } 
n\rightarrow \infty 
\end{equation}
uniformly on $W \times (\supp\,\tau)^{\NN}.$  

We now let $U$ be any neighborhood of $[1:0:0]$ in $\Pt$ with 
$[0:1:0]\in \Pt \setminus U.$ 
Then, by Lemma~\ref{l:ghsm}, there exists an open neighborhood $B$ of $[1:0:0]$ 
in $\Pt$ with $B\subset U$ such that 
$h(\Pt\setminus B)\subset \Pt \setminus \overline{B}$ for each 
$h\in \supp\,\tau$ and such that for each $z\in \Pt \setminus 
\{ [0:1:0]\} $ and for each $h\in \supp\,\tau $, there exists a number 
$n\in \NN $ with $h^{n}(z)\in \Pt \setminus B.$ 

By Claim 1, the set of minimal sets of $\tau $ in 
$\Pt \setminus B$ is equal to 
$\{ \{ [0:1:0]\} \}. $ 
Hence, for each $z\in \Pt \setminus B$, we have 
\begin{equation}
\label{eq:010ino}
[0:1:0]\in \overline{G_{\tau }(z)}.
\end{equation}  
For, otherwise, $\overline{G_{\tau }(z)}$ is a nonempty compact 
subset of $\Pt \setminus (B\cup \{ [0:1:0]\})$ which is forward invariant under any element 
of $G_{\tau }.$ Combining this with Zorn's lemma, we must have a 
minimal set $L_{0}$ of $\tau $ with $L_{0}\subset \overline{G_{\tau }(z)}
\subset \Pt \setminus (B\cup \{ [0:1:0]\})$, which contradicts Claim 1.
Thus, we have (\ref{eq:010ino}). 

 By (\ref{eq:010ino}), we have that 
 for each $z\in \Pt \setminus B$, there exists a map $h_{z}\in G_{\tau }$ 
 such that $h_{z}(z)\in W.$ 
 Combining this with that $G_{\tau }(W)\subset W$ and 
 \cite[Lemma 4.6]{Splms10}, 
we obtain that for each $z\in \Pt \setminus B$, 
 there exists a Borel subset 
$A_{z}$ of $(\supp\,\tau )^{\NN}$ with $
\tau ^{\NN}(A_{z})=1$ such that 
for each $\gamma =(\gamma _{j})_{j=1}^{\infty }\in A_{z}$, 
there exists a number $m=m(z, \gamma)\in \NN $ such that 
$\gamma _{m,1}(z)\in W.$ 
Combining this with  (\ref{eq:dgn1z0}), 
it follows that for each $\gamma \in {\cal A}_{z}$, 
$\gamma _{n,1}(z)\rightarrow [0:1:0] \mbox{ as }n\rightarrow \infty .$ 

Since $U$ can be an arbitrarily small neighborhood of $[1:0:0]$, 
the above argument 
 implies that $\tau $ is mean stable on $\Pt \setminus 
 \{ [0:1:0]\}.$
Thus, we have proved our proposition. 
\end{proof}
\begin{ex}
\label{ex:lebsupppos}
Let $f\in X^{+}$ with $|\delta (f)|>1$, where 
$\delta (f)$ denotes the Jacobian of $f.$ Let $B$ be a bounded 
Borel measurable subset of $\Ct $ with 
Leb$_{4}(B)>0.$ Let $\mu $ be the Borel probability measure 
on $\Ct $ defined by $\mu (C)=\frac{\mbox{Leb}_{4}(B\cap C)}
{\mbox{Leb}_{4}(B)}.$ Let $\Phi : \Ct \rightarrow X^{+}$ 
be the continuous map defined by $\Phi (a,b)=f+(a,b).$ 
Let $\tau =\Phi _{\ast }(\mu )\in {\frak M}_{1,c}(X^{+}).$ 
Then, $\tau $ satisfies the assumptions of 
Proposition~\ref{p:intsupdlarge}.  
Hence by Proposition~\ref{p:intsupdlarge}, we have 
that for each $z\in \Ct$, there exists a Borel subset 
$A_{z}$ of $(\supp\,\tau )^{\NN}$ with $
\tau ^{\NN}(A_{z})=1$ such that 
for each $\gamma =(\gamma _{j})_{j=1}^{\infty }\in A_{z}$, 
$\gamma _{n,1}(z)\rightarrow [0:1:0] \mbox{ as }n\rightarrow \infty $,  
and $\tau $ is mean stable on 
$\Pt \setminus \{ [1:0:0]\}.$ 
\end{ex}

We now introduce the definition of nice subsets of $X^{\pm}.$ 
\begin{df}
\label{d; nicesetI}
Let ${\cal Y}$ be a nonempty intersection of 
an open subset of $X^{+}$ (resp. $X^{-}$) and a closed subset of 
$X^{+}$ (resp. $X^{-}$).  We endow ${\cal Y}$ with the relative topology from 
$X^{\pm }.$ 
We say that $ {\cal Y}$ satisfies nice condition (I) if 
for each $z_{0}\in \Ct$, for each $f_{0}\in {\cal Y}$, 
for each neighborhood $U_{0}$ of $f_{0}$ in ${\cal Y}$, 
there exist a neighborhood $V_{0}$ of $z_{0} $ and a number $\delta >0$ 
such that for each $z\in V_{0}$, 
$\{ f(z)\mid f\in U_{0}\} \supset \{ w\in \Ct \mid \| w-f_{0}(z)\| <\delta \}.$  

\end{df}
\begin{ex}
$X^{+}$ and $X^{-}$  satisfy nice condition (I). 
Let $X_{1}^{-}:=\{ f^{-1}\mid f\in X_{1}^{+}\}.$ Then,  
$X_{1}^{+}$ and $X_{1}^{-}$ satisfy nice condition (I).  
Let $f\in X^{+}$ and let $D$ be a nonempty open subset 
of $\Ct .$ Let  
${\cal Y}:= \{ f+(a,b) \in X^{+}\mid (a,b)\in D\} .$ Then, ${\cal Y}$ satisfies 
nice condition (I).   
\end{ex}
We now provide a result on a sufficient condition 
for generic Julia sets to be of measure zero. 
\begin{prop}
\label{p:intl4jg0}
Let ${\cal Y}$ be a nonempty intersection of 
an open subset of $X^{+}$ and a closed subset of 
$X^{+}$. 
We endow ${\cal Y}$ with the relative topology from 
$X^{+}.$ 
Suppose that 
${\cal Y}$ satisfies nice condition (I). 
Let $\tau \in {\frak M}_{1,c}({\cal Y})$ 
with int$(\supp\,\tau)\neq \emptyset$, 
where int denotes the set of interior points  
with respect to the topological space ${\cal Y}.$   
Then, for $\tau ^{\ZZ}$-a.e. $\gamma =(\gamma _{j})_{j\in \ZZ}
\in (X^{+})^{\ZZ}$, 
we have Leb$_{4}(J_{\gamma }^{+})=0$. 

\end{prop}
\begin{proof}
Let $U$ be any neighborhood of $[1:0:0]$ in $\Pt.$ 
Then, by Lemma~\ref{l:ghsm} (2) and Lemma~\ref{l:wrconv}(2), there exists an open neighborhood 
$B$ of $[1:0:0]$ with $B\subset U$ such that 
for each $g\in \supp\,\tau$, 
$g(\Pt \setminus B)\subset \Pt \setminus B$,  
and such that for each $z\in \Pt \setminus \{ [1:0:0]\}$ 
and for each $g\in \supp\,\tau $, there exists a number 
$n_{0}\in \NN $ with $g^{n_{0}}(z)\in \Pt \setminus B.$ 

From the above and Lemma~\ref{l:ghsm} (1)(2), it follows that there exists a number 
$R>0$ such that any minimal set $L$ of $\tau $ satisfies that 
either $L=\{ [0:1:0]\}$ or $L\subset \{ (x,y)\in \Ct \mid \|(x,y)\| <R\}.$ 
  Since ${\cal Y}$ satisfies nice condition (I) and 
  $\mbox{int}(\supp\,\tau )\neq \emptyset, $ 
  there exists a number $r>0 $ such that 
  any minimal set $L$ of $\tau $ with 
  $L\subset \{ (x,y)\in \Ct \mid \|(x,y)\| <R\}$ 
  contains a ball of radius $r.$  
  Therefore the number of minimal sets $L$ of $\tau $ 
  with $L\subset \{ (x,y)\in \Ct \mid \|(x,y)\| <R\}$ is finite. 
  In particular, the number of minimal sets of $\tau $ is finite.  
 We now prove the following claim. \\ 
Claim 1. Let $z\in \Pt \setminus B.$ Then, there exists a 
minimal set $L$ of $\tau $ with $\overline{G_{\tau }(z)}\cap L\neq 
\emptyset$, where $\overline{G_{\tau }(z)}$ denotes the closure
 of $G_{\tau }(z)$ in $\Pt.$  

To prove this claim, let  $z\in \Pt \setminus B$ and 
suppose that there exists  no minimal set $L$ of $\tau $ with $\overline{G_{\tau }(z)}\cap L\neq 
\emptyset.$ Then, since the number of minimal sets of $\tau $ is finite, 
we obtain that $\overline{G_{\tau }(z)}\cap \overline{\cup _{L\in \Min(\tau )}L}=\emptyset.$ Since $\overline{G_{\tau }(z)}$ is a compact subset 
of $\Pt \setminus B$ and since 
$G_{\tau }(\overline{G_{\tau }(z)})\subset \overline{G_{\tau }(z)}$,  
Zorn's lemma implies that 
there exists a minimal set $L_{0}$ of $\tau $ with 
$L_{0}\subset \overline{G_{\tau }(z)}.$ However, this contradicts 
$\overline{G_{\tau }(z)}\cap \overline{\cup _{L\in \Min(\tau )}L}=\emptyset.$ Hence, we have proved Claim 1.   
 
 Let $L$ be a minimal set of $\tau $ with $L\subset \Ct.$ 
Since ${\cal Y}$ satisfies nice condition (I) and int$(\supp\,\tau )\neq \emptyset$, we have 
int$(L)\neq \emptyset.$ 
Hence, if $z\in \Pt \setminus B$ and  $\overline{G_{\tau }(z)}\cap L\neq \emptyset$, then 
$G_{\tau }(z)\cap \mbox{int}(L)\neq \emptyset.$ 
Since  int$(L)\subset F(G_{\tau })$, we have that 
\begin{equation}
\label{eq:ifzcfg}
\mbox{if \ } z\in \Pt \setminus B \mbox{ and } \overline{G_{\tau }(z)}\cap L\neq \emptyset , 
\mbox{\ then\ } G_{\tau }(z)\cap F(G_{\tau })\neq \emptyset.
\end{equation}
Moreover, by Lemma~\ref{l:ghsm}(1), the point $[0:1:0]$ belongs to $F(G_{\tau }).$ 
Hence,  
\begin{equation}
\label{eq:ifzcfg2}
\mbox{if \ } z\in \Pt \setminus B \mbox{ and } \overline{G_{\tau }(z)}\cap \{ [0:1:0]\} \neq \emptyset , 
\mbox{\ then\ } G_{\tau }(z)\cap F(G_{\tau })\neq \emptyset.
\end{equation}
By Claim 1, (\ref{eq:ifzcfg}), (\ref{eq:ifzcfg2}) and Lemma~\ref{l:ghsm}(1),  
it follows that 
\begin{equation}
\label{eq:ifzcfg3}
\mbox{for each } z\in \Pt \setminus B, \ 
G_{\tau }(z)\cap F(G_{\tau })\neq \emptyset .
\end{equation}
By (\ref{eq:ifzcfg3}), we have $J_{\ker }(G_{\tau }| _{\Pt\setminus B})=\emptyset, $ 
where $G_{\tau }|_{\Pt \setminus B}=\{g|_{\Pt \setminus B}: \Pt \setminus B\rightarrow \Pt \setminus B \mid g\in G_{\tau }\}.$ 
Moreover, for each $\gamma \in (\supp\,\tau )^{\ZZ}, $ 
$J_{\gamma }^{+}\cap (\Pt \setminus \overline{B})$ is equal to the set of all 
points $z\in \Pt \setminus \overline{B}$ for which there exists no open neighborhood 
$U$ of $z$ in $\Pt \setminus \overline{B}$ such that 
$\{ \gamma _{n-1}\circ \cdots \circ \gamma _{1}\}_{n=1}^{\infty }$ is equicontinuous 
on $U$. 

Combining the above facts with \cite[Proposition 4.8]{Splms10}, 
it follows that 
Leb$_{4}(J_{\gamma }^{+}\cap (\Pt\setminus \overline{B}))=0$ for 
$\tau ^{\ZZ}$-a,e. $\gamma \in (\supp\,\tau)^{\ZZ}.$  
Since $U$ is an arbitrary neighborhood of $[1:0:0]$, it follows that 
Leb$_{4}(J_{\gamma }^{+})=0$ for 
$\tau ^{\ZZ}$-a,e. $\gamma \in (\supp\,\tau)^{\ZZ}.$   
\end{proof}

We need the following lemma to prove Theorem~\ref{t:nyItaut}. 
\begin{lem}
\label{l:taumsema}
Let $\tau \in {\frak M}_{1,c}(X^{+})$ 
(resp. ${\frak M}_{1,c}(X^{-})$). 
Then, $\tau $ is mean stable 
on $\Pt \setminus \{ [1:0:0]\}$ (resp. $\Pt \setminus \{ [0:1:0]\}$) 
if and only if 
each minimal set $L$ of $\tau $ with $L\subset \Ct$ 
is attracting for $\tau .$  
\end{lem}
\begin{proof}
Let $\tau \in {\frak M}_{1,c}(X^{+})$ (resp. 
${\frak M}_{1,c}(X^{-})$).  
By Lemma~\ref{l:ghsm},  
any minimal set $L$ of $\tau $ 
with $L\neq \{ [0:1:0]\}$ (resp. $L\neq \{ [1:0:0]\}$) 
satisfies $L\subset \Ct .$ 
Also, since $\supp\,\tau $ is compact, 
 Lemma~\ref{l:ghsm} implies that  
$\{ [0:1:0]\} $ (resp. $\{ [1:0:0]\} $) is an attracting minimal 
set for $\tau $. If $\tau $ is mean stable on $\Pt \setminus \{ [1:0:0]\}$ 
(resp. $\Pt \setminus \{ [0:1:0]\}$), then by the definition of mean stability it 
is easy to see that each minimal set of $\tau $ is attracting for $\tau.$ 

We now suppose that each minimal set of $\tau $ in $\Ct$ is attracting for
 $\tau .$ Then we have $\sharp \Min (\tau )<\infty .$ For, if $\sharp \Min (\tau )=\infty $ then by Lemma~\ref{l:ghsm} there exist a sequence 
 $\{ L_{n}\} $ of mutually distinct elements of $\Min (\tau )$ in $\Ct$, a sequence 
 $\{ z_{n}\} $ in $\Ct $ with $z_{n}\in L_{n}$, and a point $z_{\infty }\in \Ct$ with 
 $z_{n}\rightarrow z_{\infty }$ as $n\rightarrow \infty .$ Then 
 $\overline{G_{\tau }(z_{\infty })} \subset \Ct $ and Zorn's lemma implies 
 that there exists an element $L\in \Min (\tau )$ with $L\subset \overline{G_{\tau }(z)}$, which is attracting for $\tau.$  However, this implies that $L_{n}=L$ for each large 
 $n$ and we have a contradiction. Hence we have  $\sharp \Min (\tau )<\infty .$ 
 Combining this with that each minimal set of $\tau $ is attracting, we obtain that 
 $\tau $ is mean stable on $\Pt \setminus \{ [1:0:0]\}$ 
 (resp. $\Pt \setminus \{ [0:1:0]\} $). 
\end{proof}
The following theorem is the key for  
proving the density of ${\cal MS}$ in ${\frak M}_{1,c}(X^{+})$ with respect to 
${\cal O}$ in 
Theorem~\ref{t:rpmms1} and 
the density of ${\cal MS}\cap {\frak M}_{1,c}(X_{1}^{+})$ in 
${\frak M}_{1,c}(X_{1}^{+})$ with respect to ${\cal O}$ 
in Theorem~\ref{t:rpmms2}. 
\begin{thm}
\label{t:nyItaut}
Let ${\cal Y}$ be  
a nonempty intersection of 
an open subset of $X^{+}$ (resp. $X^{-}$) and a closed subset of 
$X^{+}$ (resp. $X^{-}$). 
Suppose that ${\cal Y}$ satisfies 
nice condition (I). 
Let $\{ \tau _{t}\} _{t\in [0.1]}$ be a family of 
elements of ${\frak M}_{1,c}({\cal Y})$ such that 
 the following {\em (i)}, {\em (ii)} and {\em (iii)} hold. 
\begin{itemize}
\item[{\em (i)}] 
$t\in [0,1]\mapsto \tau _{t}\in {\frak M}_{1,c}({\cal Y})$ is continuous 
with respect to the wH-topology ${\cal O}.$ 

\item[{\em (ii)}]
If $t_{1}, t_{2}\in [0,1]$ and $t_{1}<t_{2}$, then 
$\supp\,\tau _{t_{1}}\subset \mbox{int}(\supp\,\tau _{t_{2}}).$ 
Here, int denotes the set of interior points with respect to the topology 
in ${\cal Y}.$  
\item[{\em (iii)}] 
$\mbox{int} (\supp\,\tau _{0})\neq \emptyset .$ 
\end{itemize} 
Let $A:=\{ t\in [0,1]\mid \tau _{t} \mbox{ is not mean stable on }
\Pt \setminus \{ [1:0:0]\} (\mbox{resp. } 
\Pt \setminus \{ [0:1:0]\}).$ 
Then, $\sharp A\leq \sharp \emMin(\tau _{0})-1<\infty $ and 
$B:=\{ t\in [0,1]\mid s\mapsto \sharp \emMin(\tau _{s}) \mbox{ is 
constant in a neighborhood of }t\} $ satisfies 
$B= [0,1]\setminus A.$

\end{thm}
\begin{proof} 
We show the statement of our theorem for the case 
that  ${\cal Y}$ is a nonempty intersection of 
an open subset of $X^{+}$ and a closed subset of 
$X^{+}$ and satisfies nice condition (I) 
(we can show the statement of our theorem for the case 
that  ${\cal Y}$ is a nonempty intersection 
of an open subset of $X^{-}$ and a closed subset of  
$X^{-}$ and satisfies nice condition (I) 
by arguments similar to those below
). 
   
By Lemma~\ref{l:wrconv}, there exist a number 
$R>1$ and a number $\rho >1$ such that 
each $g\in \supp\,\tau_{0}$ satisfies 
condition (A) for $(R,\rho).$ 
By Lemma~\ref{l:wrconv}(2)(3), it follows that each 
$L\in \Min (\tau _{0})$ with $L\subset \Ct$ 
should be included in $D_{R}.$ 
Combining this with (iii) in the assumptions of our theorem 
and the assumption that ${\cal Y}$ satisfies nice condition 
(I),  
we have that $\sharp \{ L\in \Min (\tau _{0})\mid L\subset \Ct\}<\infty .$  
Moreover, by Zorn's lemma, we have that 
if $t_{1},t_{2}\in [0,1], t_{1}<t_{2}$ then 
$1\leq \sharp \Min (\tau _{t_{2}})\leq 
\sharp \Min (\tau _{t_{1}})$. 
 (For, if $L\in \Min(\tau _{t_{2}})$, then $G_{\tau _{t_{1}}}(L)
 \subset G_{\tau _{t_{2}}}(L)\subset L .$ By Zorn's lemma, 
 there exists an element $L'\in \Min(\tau _{t_{1}})$ 
 with $L'\subset L$.)   
It follows that setting $$B:=\{ t\in [0,1]\mid 
s\mapsto \sharp \Min (\tau _{s}) \mbox{ is constant on a neighborhood of } t\},$$ 
we have $\sharp ([0,1]\setminus B)\leq \sharp 
\Min(\tau _{0})-1<\infty .$   

Let $t_{0}\in B.$ Then, there exists $t_{1}\in [0,1]$ with  
$t_{0}<t_{1}$ (arbitrarily close to $t_{0}$) such that  
$\sharp \Min(\tau _{t_{0}})=\sharp \Min(\tau _{t_{1}}).$
 We now prove the following claim. 

Claim 1. For each $K_{0}\in \Min (\tau _{t_{0}})$, 
there exists a unique element $K_{1}\in 
\Min (\tau _{t_{1}})$ such that 
$K_{0}\subset K_{1}.$ Moreover, 
the map 
$K_{0}\in \Min (\tau _{t_{0}})\mapsto 
K_{1}\in \Min (\tau _{t_{1}})$  
is injective. 

To prove Claim 1, let $L\in \Min (\tau _{t_{1}}).$ 
Since  $G_{\tau _{t_{0}}}(L)\subset 
G_{\tau _{t_{1}}}(L)\subset L$, 
 Zorn's lemma implies that there exists 
 an element 
$L'\in \Min (\tau _{t_{0}})$ such that $L'\subset L.$ 
Combining the above argument with the fact that $\sharp \Min(\tau _{t_{0}})=\sharp \Min(\tau _{t_{1}})$, 
we easily see that 
Claim 1 holds. 

We now let $K_{0}\in \Min(\tau _{t_{0}})$ 
with $K_{0}\subset \Ct.$  
%
%
Then, by Claim 1, 
there exists a unique element 
$K_{1}\in \Min (\tau _{t_{1}})$ 
such that $K_{0}\subset K_{1}.$ 
Since $K_{0}\neq \{ [0:1:0]\}$, we have $K_{1}\neq \{ [0:1:0]\}.$  
Hence by Lemma~\ref{l:ghsm}, we have 
$K_{1}\subset \Ct.$ 
Since $K_{1}\subset \Ct $, 
${\cal Y}$ satisfies nice condition (I) and 
int$(\supp\,\tau _{t_{1}})\neq \emptyset $,  
we have that $\emptyset \neq \mbox{int}(K_{1})$ and 
int$(K_{1})\subset F(G_{\tau _{t_{1}}}).$ We now prove the following claim. 

Claim 2. For each $g\in \supp\,\tau _{t_{0}}$, 
$g(\partial K_{1})\subset \mbox{int}(K_{1}).$ 

To prove this claim, suppose that there exist an element $g_{0}\in 
\supp\,\tau _{t_{0}}$ and an element $z_{0}\in \partial K_{1}$ 
such that $g_{0}(z_{0})\in \partial K_{1}.$ Then, since ${\cal Y}$ satisfies
 nice condition (I) and 
 $g_{0}\in \supp\,\tau _{t_{0}}\subset 
 \mbox{int}(\supp\,\tau _{t_{1}})$, 
 it follows that there exists an element $g_{1}$ (near $g_{0}$) 
in $\mbox{int}(\supp\,\tau _{t_{1}})$ such that $g_{1}(z_{0})\in 
\Ct \setminus K_{1}.$ However, this contradicts that 
$K_{1}\in \Min(\tau _{t_{1}}).$ Hence, we have proved Claim 2. 

 By Claim 2, we have that for each $g\in \supp\,\tau _{t_{0}}$, 
 $g(K_{1})\subset \mbox{int}(K_{1}).$ Therefore,  
 $C:=\cup _{g\in \supp\,\tau _{t_{0}}}g(K_{1})$ is a compact subset 
 of $\mbox{int}(K_{1}).$ Since $G_{\tau _{0}}(C)\subset C$, Zorn's lemma 
 implies that there exists an element $K'\in \Min (\tau _{t_{0}}) $
 such that $K'\subset C\subset K_{1}.$ By Claim 1, 
 we have 
 \begin{equation}
\label{eq:k0ksube}
 K_{0}=K'\subset C\subset \mbox{int}(K_{1}).
 \end{equation}
  
 Let $\{ W_{1},\ldots, W_{r}\}$ be the set of all connected components of int$(K_{1})$ that intersect with $C.$ 
 Let  $V, U $ be two open subsets of $\cup _{j=1}^{r}W_{j}$ 
 such that 
 $C\subset V\subset \overline{V}\subset U\subset \overline{U}
 \subset \cup _{j=1}^{r}W_{j}$ and such that 
 $U\cap W_{j}$ is connected for each $j=1.\ldots r.$ 
 Then, we have 
\begin{equation}
\label{eq:ugsuppgou}
\cup _{g\in \supp\,\tau _{t_{0}}}g(\overline{U})\subset 
C\subset V\subset U. 
\end{equation}
Let $d_{j}$ be the Carath\'{e}odory distance on $U\cap W_{j}$ 
for each $j=1,\ldots, r.$  
Note that since $U\cap W_{j}$ is a bounded region in $\Ct$, 
the Carath\'{e}odory pseudodistance on $U\cap W_{j}$ is a 
distance on $U\cap W_{j}.$ By (\ref{eq:ugsuppgou}) and 
\cite[page 253]{Mi}, there exists a number $c\in (0,1)$ 
such that for each $g\in \supp\tau _{t_{0}}$ and for each 
$i,j$, 
if $g(U\cap W_{j})\subset W_{i}$, then 
$$\mbox{diam}_{d_{i}}(g(U\cap W_{j}))\leq c\mbox{diam}_{d_{j}}
(U\cap W_{j}), $$
where diam$_{d_{i}}(D)=\sup \{ d_{i}(x,y)\mid x,y\in D\}$ for each $i$ and  
for each subset $D\subset U\cap W_{i}.$ 
Therefore, for each 
$j=1,\ldots, r$, we have that 
diam$(\gamma _{n,1}(U\cap W_{j}))\rightarrow 0 $ as $n\rightarrow 
\infty $ uniformly on $(\supp\,\tau _{t_{0}})^{\NN}$,  
where diam$(\gamma _{n,1}(U\cap W_{j}))$ denotes the 
diameter of $\gamma _{n,1}(U\cap W_{j})$ 
with respect to the Euclidean distance on $\Ct.$ Since $K_{0}\subset C\subset 
\cup _{j=1}^{r}(U\cap W_{j})$, 
it follows that $K_{0}$ is an attracting minimal set for 
$\tau_{t_{0}}.$ 

From the above arguments, we see that each 
$K_{0}\in \Min(\tau _{t_{0}})$ with $K_{0}\subset \Ct$ 
is an attracting minimal set for $\tau _{t_{0}}.$ 
Therefore, Lemma~\ref{l:taumsema} 
implies that 
$\tau _{0}$ is mean stable  on $\Pt \setminus \{[ 1:0:0]\}.$ 
Also, by (i) in the assumptions of our theorem, it is easy to see that $[0,1]\setminus A\subset B$ 
(see the argument in the proof of \cite[Lemma 5.7]{Sadv}). 
Thus, we have proved our theorem.
\end{proof}
\begin{rem}
We remark that in the proof of  
Theorem~\ref{t:nyItaut},  
we introduce some new and strong methods 
that are valid for  random holomorphic dynamical systems 
of any dimension. 
Note that we have several results on complex-one-dimensional 
random holomorphic dynamical systems in \cite{Sadv} that 
are similar to Theorem~\ref{t:inyItaut} and the result on the density of 
${\cal MS}$ in $X^{+}$ in Theorem~\ref{t:rpmms1}, 
but 
the methods and arguments in the proofs of them in \cite{Sadv} 
  were valid only for complex-one-dimensional random holomorphic dynamical systems.

\end{rem}
\begin{ex}
\label{e:addnoise}
Let $f\in X^{+}$. 
Let ${\cal Y}=\{ f+(a,b)\in X^{+}\mid (a,b)\in \Ct \} .$ 
Then, ${\cal Y}$ is a closed subset of $X^{+}$ and 
satisfies nice condition (I). 
Let $u>v>0$ and for each $t\in [0,1]$, 
let $\mu _{t}$ be the normalized Lebesgue measure 
on $B((0,0),ut+(1-t)v)$. 
Let $\Psi : \Ct \mapsto {\cal Y}$ be the continuous map defined 
 by $\Psi (a,b)=f+(a,b)$. 
For each $t\in [0,1]$, 
let $\tau _{t} :=\Psi _{\ast }(\mu _{t})\in {\frak M}_{1,c}({\cal Y}).$ 
Then, $\{ \tau _{t}\}_{t\in [0,1]}$ satisfies the assumptions 
of Theorem~\ref{t:nyItaut}. 
We now suppose that $f$ has an attracting periodic cycle $E$ in $\Ct .$   
Let $v>0$ be small enough and let $u$ be large enough. 
Then there exists an attracting minimal set of $\tau _{0}$ which contains 
$E.$ Hence, $\sharp \Min (\tau _{0})>1$ (note that 
$\{ [0:1:0]\} \in \Min (\tau _{0})).$ 
If we 
adopt a large enough value for $u$, 
then we see that there exists no minimal set $L$ of $\tau _{1}$ 
with $L\subset \Ct.$ Therefore,  
$\sharp \Min(\tau _{1})=1.$ It follows that 
the set $A$ in Theorem~\ref{t:nyItaut} for 
this family $\{ \tau _{t}\} _{t\in [0,1]}$ is not empty. 
 \end{ex}
 We now show that 
 if ${\cal Y}$ is a subset of $X^{+}$ (resp. $X^{-}$) 
 satisfying some mild conditions, then  
 generic $\tau \in {\frak M}_{1,c}({\cal Y})$ are 
 mean stable on $\Pt \setminus \{ [1:0:0]\}$
  (resp.  $\Pt \setminus \{ [0:1:0]\}$). 
\begin{thm}
\label{t:yniceaod}
Let ${\cal Y}$ be an open subset of a 
(not necessarily connected) closed submanifold of $X^{+}$  
(resp. $X^{-}$).  
Suppose that ${\cal Y}$ satisfies 
nice condition (I). Let 
$${\cal A}:= \{ \tau \in {\frak M}_{1,c}({\cal Y}) \mid 
\tau \mbox{ is mean stable on }\Pt \setminus 
\{ [1:0:0]\}  (\mbox{resp. } \Pt \setminus 
\{ [0:1:0]\})\} .$$ Then,  
${\cal A}$ is open and dense in 
${\frak M}_{1,c}({\cal Y})$ with respect to the 
wH-topology ${\cal O}$ in  ${\frak M}_{1,c}({\cal Y})$. 
\end{thm}
\begin{proof}
It is easy to see that ${\cal A}$ is open in ${\frak M}_{1,c}({\cal Y})$. 
To show the density of ${\cal A}$ in 
${\frak M}_{1,c}({\cal Y})$, let $\zeta \in {\frak M}_{1,c}({\cal Y})$ 
and let $U$ be an open neighborhood of $\zeta $ in 
${\frak M}_{1,c}({\cal Y}).$ 
Then, there exists an element $\zeta _{0}\in U$ with 
$\sharp \supp\,\zeta _{0}<\infty .$ By enlarging 
the support of $\zeta _{0}$, we can construct a family 
$\{ \tau _{t}\} _{t\in [0.1]}$ of elements in $U$ 
that satisfies (i)(ii)(iii) in the assumptions of 
Theorem~\ref{t:nyItaut}.  
Then, by Theorem~\ref{t:nyItaut}, 
there exists a $t\in [0,1]$ such that 
$\tau _{t}\in U$ and 
$\tau _{t}$ is mean stable on 
$\Pt \setminus \{ [1:0:0]\}$ (resp. $\Pt \setminus 
\{ [0:1:0]\}$).  
Hence, ${\cal A}$ is dense in ${\frak M}_{1,c}({\cal Y}).$ 
Thus, we have proved our theorem.
\end{proof}
To present results on elements $\tau \in {\frak M}_{1,c}(X^{+})$ 
that are mean stable on $\Pt \setminus \{ [1:0:0]\}$
 (Theorem~\ref{t:mtauspec}), 
 we need 
the following definitions. 
\begin{df}
For a topological space $Y$, we denote by $C(Y)$ 
the set of all complex-valued bounded continuous functions 
on $Y.$ 
\end{df}
\begin{df}
Let ${\cal B}$ be a complex vector space, and let $M:{\cal B}\rightarrow {\cal B}$ be a linear operator.  
Let $\varphi \in {\cal B}$ and $a\in \CC $ be such that 
$\varphi \neq 0, |a|=1$, and $M(\varphi )=a\varphi .$ Then, we say that  
$\varphi $ is a unitary eigenvector of $M$ with respect to $a$
and 
that $a$ is a unitary eigenvalue.   
\end{df}
\begin{df}
For any topological space  $Y$, we denote by 
${\frak M}_{1}(Y)$ the set of all Borel probability measures on $Y.$ 
Also, we set ${\frak M}_{1,c}(Y)=\{ \tau \in {\frak M}_{1}(Y)\mid 
\supp\,\tau \mbox{ is a compact subset of }Y\}.$ 
\end{df}
\begin{df}
Let $Y$ be a $\sigma$-compact metric space. 
Let 
$\tau \in {\frak M}_{1}(\CMX).$ 
We set $G_{\tau }=\{ h_{n}\circ \cdots \circ h_{1}\mid 
n\in \NN , h_{j}\in \supp\,\tau (\forall j)\}.$ 
Let $K$ be a nonempty subset of $Y$ such that 
$G_{\tau }(K)\subset K$. 
Let $M_{\tau }:C(K)\rightarrow C(K)$ be the linear operator 
defined as $M_{\tau }(\varphi )(z)=
\int _{CM(Y)} \varphi (g(z))d\tau (g)$ for each $\varphi \in C(K)$ 
and for each $z\in K.$ 
We denote by ${\cal U}_{f,\tau }(K)$ the set of 
all unitary eigenvectors of $M_{\tau }:C(K)\rightarrow C(K)$. Moreover, 
we denote by ${\cal U}_{v,\tau }(K)$ the set of all 
unitary eigenvalues of $M_{\tau }:C(K)\rightarrow C(K).$  
Similarly, we denote by 
${\cal U}_{f,\tau ,\ast}(K)$ the set of all unitary eigenvectors of 
$M_{\tau }^{\ast }:C(K)^{\ast }\rightarrow C(K)^{\ast }$, 
 and we denote by ${\cal U}_{v,\tau ,\ast }(K)$ the set of all 
unitary eigenvalues of $M_{\tau }^{\ast }:C(K)^{\ast }\rightarrow C(K)^{\ast }.$ 
\end{df}
\begin{df}
Let $V$ be a complex vector space, and let $A$ be a subset of $V.$ 
We set $\mbox{LS}(A):= \{ \sum _{j=1}^{m}a_{j}v_{j}\mid a_{1},\ldots ,a_{m}\in \CC , 
v_{1},\ldots ,v_{m}\in A, m\in \NN \} .$ 
\end{df}
\begin{df}
Let $Y$ be a topological space, and let $V$ be a subset of $Y.$ 
We denote by $C_{V}(Y)$ the space of all $\varphi \in C(Y)$ such that 
for each connected component $U$ of $V$, there exists a constant $c_{U}\in \CC $ with 
$\varphi |_{U}\equiv c_{U}.$ 
\end{df}
\begin{df}
Let $Y$ be a metric space. 
We set $\OCMX :=\{ g\in \CMX \mid g \mbox{ is an open map on }Y\}$ endowed with the relative topology from 
$\CMX.$   
\end{df}
\begin{rem}
$C_{V}(Y)$ is a linear subspace of $C(Y).$ Moreover, if 
$Y$ is compact, metrizable, and locally connected and $V$ is an open subset of $Y$, then 
$C_{V}(Y)$ is a closed subspace of $C(Y)$ 
(we endow $C(Y)$ the supremum norm on $Y$). 
Furthermore, 
if $Y$ is compact, metrizable, and locally connected, $\tau \in {\frak M}_{1}(\CMX)$, and $G_{\tau }$ is a 
subsemigroup of $\OCMX$, then $M_{\tau }(C_{F(G_{\tau })}(Y))\subset C_{F(G_{\tau })}(Y).$  
\end{rem}
\begin{df}
Let $Y$ be a compact metric space. 
Let $\rho \in  C(Y)^{\ast }.$  
We denote by $a(\rho )$ the set of points $z\in Y$ 
which satisfies that there exists a neighborhood $U$ of $z$ in $Y$ 
such that for each $\varphi \in C(Y)$ with $\mbox{supp}\, \varphi \subset U$, 
$\rho (\varphi )=0.$ 
We set $\mbox{supp}\, \rho := Y\setminus a(\rho ).$ 
\end{df}

\begin{df}
Let $\{ \varphi _{n}:U\rightarrow \Pt \} _{n=1}^{\infty }$ be a sequence of 
holomorphic maps on an open set $U$ of $\Pt .$ Let 
$\varphi :U\rightarrow \Pt $ be a holomorphic map. 
We say that $\varphi $ is a limit map of 
$\{ \varphi _{n}\} _{n=1}^{\infty }$ if 
there exists a strictly increasing sequence $\{ n_{j}\} _{j=1}^{\infty }$ in 
$\NN $ such that $\varphi _{n_{j}}\rightarrow \varphi $ as $j\rightarrow \infty $ locally 
uniformly on $U.$ 
\end{df}
\begin{df}
For a topological space $Z$, we denote by Con$(Z)$ the set of all 
connected components of $Z.$  
\end{df}
\begin{df}
Let $K\in \Cpt (\Pt).$ Let $\{\varphi _{n}\}_{n=1}^{\infty }$ 
be a family in $C(K)$ with $\overline{\{\varphi _{n}\mid n\in \NN\}}
=C(K).$ For each $\tau, \rho \in {\frak M}_{1}(K)$, we set 
$d_{0,K}(\tau, \rho ):=
\Sigma _{n=1}^{\infty }\frac{|\int \varphi _{n}d\tau -
\int \varphi _{n} d\rho |}
{1+|\int \varphi _{n}d\tau -\int \varphi _{n}d\rho |}.$ 
Note that $d_{0,K}$ is a distance on 
${\frak M}_{1}(K)$ that induces the weak$^{\ast }$ topology 
on ${\frak M}_{1}(K)$.  

\end{df}
\begin{df}
\label{d:hcvk}
Let $K$ be a nonempty subset of 
$\Pt .$ Let $V$ be an open subset of $\Pt .$ 
We denote by $\hat{C}_{V}(K)$ the set of all 
elements $\varphi \in C(K)$ satisfying that 
for each $U\in \mbox{Con}(V)$ with $U\cap K\neq \emptyset,$  
$\varphi |_{U\cap K}$ is constant. 

\end{df}
\begin{df}
Let $Y$ be a topological space. 
For each subspace $A$ of $C(Y)$, we denote by 
$A_{nc}$ the set of all nonconstant elements in $C(Y).$ 
\end{df}
\begin{df}
Let $G$ be a semigroup generated by a subset $\Lambda $ of $X^{+}$, 
i.e., 
$G=\{ h_{n}\circ \cdots \circ h_{1}\mid n\in \NN, h_{j}\in \Lambda (\forall j)\} .$  
We regard $G$ as a subsemigroup of 
$\mbox{CM}(\Pt \setminus \{ [1:0:0]\}).$  
We denote by $J_{res}(G)$ the set of all elements 
$z\in J(G)$ satisfying that for each $U\in \mbox{Con}(F(G))$, 
$z\not\in \partial U.$ 
\end{df}
\begin{df}
Let $\tau \in {\frak M}_{1,c}(X^{+})$ (resp. ${\frak M}_{1,c}(X^{-})$).  
Let $A$ be a nonempty subset of $\Pt .$ 
For each $z\in \Pt \setminus \{ [1:0:0]\}$ (resp. 
$\Pt \setminus \{ [1:0:0]\}$), 
we set\\  
$T_{A,\tau }(z)=\tau ^{\ZZ} (\{ \gamma =(\gamma _{j})_{j\in \ZZ}\in (\supp\,\tau )^{\ZZ}\mid 
d(\gamma _{n,1}(z), A)\rightarrow 0\ \mbox{ as } n\rightarrow \infty \}).$ 
(Note that $T_{A,\tau }(z)=\tau ^{\NN} (\{ \gamma =(\gamma _{j})_{j\in \NN}\in (\supp\,\tau )^{\NN}\mid 
d(\gamma _{n,1}(z), A)\rightarrow 0\ \mbox{ as } n\rightarrow \infty \}).$)
\end{df}
We now present results on elements $\tau \in {\frak M}_{1,c}(X^{+})$ 
that are mean stable on $\Pt \setminus \{ [1:0:0]\}$. 
\begin{thm}[Cooperation Principle: Disappearance of Chaos]
\label{t:mtauspec}
Let $\tau \in {\frak M}_{1,c}(X^{+})$. 
Let $Y=\Pt \setminus \{ [1:0:0]\}.$  
Suppose that $\tau $ is mean stable on $Y.$ 
Let 
$S_{\tau }:=\bigcup _{L\in \emMin (\tau )}L.$  
Then, for each neighborhood $B$ of $[1:0:0]$,    
there exists a compact subset  \textcolor{black}{$K$} of 
$Y$ 
with 
$\Pt \setminus K\subset B$  
such that $G_{\tau }(K)\subset K$ and 
$S_{\tau }\subset \mbox{int}(K).$ 
Moreover, there exists a finite dimensional subspace 
$\textcolor{black}{W_{\tau }}\neq \{ 0\} $ of $C(Y)$ with $M_{\tau }(W_{\tau })=W_{\tau }$ 
such that 
for each compact subset $K$ of $Y$ satisfying that  
$G_{\tau }(K)\subset K$ and 
$S_{\tau }\subset \mbox{int}(K)$,  
all of the following statements 1,$\ldots ,$19 hold. 
\begin{enumerate}
\item \label{t:mtauspec2}
Let ${\cal B}_{0,\tau, K }:= \{ \varphi \in C(K )\mid M_{\tau }^{n}(\varphi )\rightarrow 0 \mbox{ as }n\rightarrow \infty \} $, and 
let $W_{\tau }|_{K}:=\{ \varphi |_{K} \in C(K)\mid \varphi \in W_{\tau }\}$.  We endow $C(K)$ with the supremum norm 
$\| \cdot \| _{\infty }.$ 
Then, ${\cal B}_{0,\tau, K }$ is a closed subspace of $C(K )$,
$W_{\tau }|_{K}=\mbox{{\em LS}}({\cal U}_{f,\tau }(K))$ 
 and there exists a direct sum decomposition 
$C(K)=\mbox{{\em LS}}({\cal U}_{f,\tau }(K))\oplus {\cal B}_{0,\tau, K}.$ 
Also, the projection map 
$\pi _{\tau, K}: C(K)\rightarrow \mbox{{\em LS}}({\cal U}_{f,\tau}(K))$ 
{\em (}$\subset C(K)${\em )}  
is continuous. Moreover, 
$W_{\tau }\subset C_{F(G_{\tau })}(Y).$ 
Furthermore, $\mbox{{\em LS}}({\cal U}_{f,\tau }(K))\subset 
\hat{C}_{F(G_{\tau })}(K )$,  
the map $g\in W_{\tau }\mapsto g|_{K}\in \mbox{{\em LS}}({\cal U}_{f,\tau }(K))$ is a linear isomorphism from $W_{\tau }$ to 
$\mbox{{\em LS}}({\cal U}_{f,\tau }(K))$, 
and 
$\dim _{\CC }(\mbox{{\em LS}}({\cal U}_{f,\tau }(K)))=\dim _{\CC }
W_{\tau }<\infty .$ 
\item \label{t:mtauspec2-1}
Let $q:=\dim_{\CC }(\emLSfk )=\dim _{\CC}W_{\tau }.$ 
Let $\{ \varphi _{j}\} _{j=1}^{q}$ be a basis of 
$\emLSfk$ such that for each $j=1,\ldots ,q$, there exists an $\alpha _{j}\in \Uvk $ with 
$M_{\tau }(\varphi _{j})=\alpha _{j}\varphi _{j}.$ Then, 
there exists a unique family $\{ \rho _{j}:C(K)\rightarrow \CC \} _{j=1}^{q}$ of 
complex linear functionals such that for each $\varphi \in C(K )$, 
$\| M_{\tau }^{n}(\varphi -\sum _{j=1}^{q}\rho _{j}(\varphi )\varphi _{j})\| _{\infty }
\rightarrow 0$ as $n\rightarrow \infty .$ 
Moreover, $\{ \rho _{j}\} _{j=1}^{q}$ satisfies all of the following.
\begin{itemize}
\item[{\em (a)}] For each $j=1,\ldots ,q$, $\rho _{j}:C(K)\rightarrow \CC $ is continuous. 
\item[{\em (b)}] For each $j=1,\ldots ,q$, $M_{\tau }^{\ast }(\rho _{j})=\alpha _{j}\rho _{j}.$ 
\item[{\em (c)}] For each $(i,j)$, $\rho _{i}(\varphi _{j})=\delta _{ij}.$ 
Moreover, $\{ \rho _{j}\} _{j=1}^{q}$ is a basis of $\emLSfak.$  
\item[{\em (d)}] For each $j=1,\ldots ,q$, {\em supp}$\, \rho _{j}\subset S_{\tau }.$ 
\end{itemize} 

\item \label{t:mtauspec1}
There exists a Borel measurable subset ${\cal A}$ of 
$(X^{+})^{\NN }$ with $\tau ^{\NN}({\cal A})=1$ such that 
\begin{itemize}
\item[{\em (a)}]
for each $\gamma \in {\cal A}$ and for each $U\in \mbox{{\em Con}}(F(G_{\tau }))$, 
each limit map of $\{ \gamma _{n,1}|_{U}\} _{n=1}^{\infty }$ is constant, and 
\item[{\em (b)}]
for each $\gamma \in {\cal A}$ and for each $Q\in \mbox{{\em Cpt}}(F(G_{\tau }))$, 
$\sup _{a\in Q}\| D(\gamma _{n,1})_{a}\| \rightarrow 0$ as $n\rightarrow \infty $, 
where $\| D(\gamma _{n,1})_{a} \| $ denotes the norm of the derivative 
of $\gamma _{n,1}$ at a point $a$ 
with respect to the Fubini-Study metric on $\Pt.$ 
\end{itemize} 
\item \label{t:mtauspecaz} For each $z\in Y $, there exists a Borel subset ${\cal A}_{z}$ 
of $(X^{+})^{\NN }$ with $\tau ^{\NN}({\cal A}_{z})=1$ with the following property.
\begin{itemize}
\item For each $\gamma =(\gamma _{1},\gamma _{2},\ldots )\in {\cal A}_{z}$, 
there exists a number $\delta =\delta (z,\gamma )>0$ such that 
$\mbox{diam}(\gamma _{n,1}(B(z,\delta )))\rightarrow 0$ as 
$n\rightarrow \infty $, where diam denotes the diameter with respect to the distance induced by the 
Fubini--Study metric on $\Pt $, and $B(z,\delta )$ denotes the ball with center $z$ and radius 
$\delta .$ 
\end{itemize}

\item \label{t:mtauspec3}
$\sharp \emMin(\tau )<\infty .$
\item \label{t:mtauspec4}
Let $\Omega= \bigcup _{A\in \text{{\em Con}}(F(G_{\tau })), A\cap S_{\tau }\neq \emptyset }A$. 
Then, $S_{\tau }$ is compact. Moreover, for each $z\in Y$, there exists a Borel measurable subset 
${\cal C}_{z}$ of $(X^{+})^{\NN }$ with $\tau ^{\NN}({\cal C}_{z})=1$ such that 
for each $\gamma \in {\cal C}_{z}$, there exists an $n\in \NN $ with 
$\gamma _{n,1}(z)\in \Omega$ and   
$d(\gamma _{m,1}(z),S_{\tau })\rightarrow 0$ as $m\rightarrow \infty .$ 
\item \label{t:mtauspec5}
Let $L\in \emMin(\tau)$ and 
$r_{L}:= \dim _{\CC }(\emLSfl ).$ Then, 
${\cal U}_{v,\tau }(L)$ is a finite subgroup of $S^{1}$ with 
$\sharp {\cal U}_{v,\tau }(L)=r_{L}.$  
Moreover, there exist an $a_{L}\in S^{1}$   
and a family $\{ \psi _{L,j}\} _{j=1}^{r_{L}}$ in ${\cal U}_{f,\tau }(L)$ 
such that 
\begin{itemize}
\item[{\em (a)}]
\label{t:mtauspec5a}
$a_{L}^{r_{L}}=1$, ${\cal U}_{v,\tau }(L)=\{ a_{L}^{j}\} _{j=1}^{r_{L}}$, 
\item[{\em (b)}]
\label{t:mtauspec5b}
$M_{\tau }(\psi _{L,j})=a_{L}^{j}\psi _{L,j} $ for each $j=1,\ldots ,r_{L}$, 
\item[{\em (c)}] 
\label{t:mtauspec5c}
$\psi _{L,j}=(\psi _{L,1})^{j}$ for each $j=1,\ldots ,r_{L}$, and 
\item[{\em (d)}] 
\label{t:mtauspec5d}
$\{ \psi _{L,j}\} _{j=1}^{r_{L}}$ is a basis of $\mbox{{\em LS}}({\cal U}_{f,\tau }(L)).$ 
\end{itemize}
\item \label{t:mtauspec6}
Let $\Psi _{S_{\tau }}: \mbox{{\em LS}}({\cal U}_{f,\tau }(K))\rightarrow 
C(S_{\tau })$ be the map defined by $\varphi \mapsto \varphi |_{S_{\tau }}.$ 
Then, $\Psi _{S_{\tau }}(\mbox{{\em LS}}({\cal U}_{f,\tau }(K)))=
\mbox{{\em LS}}({\cal U}_{f,\tau }(S_{\tau }))$
and 
$\Psi _{S_{\tau }}: \emLSfk \rightarrow 
\mbox{{\em LS}}({\cal U}_{f,\tau }(S_{\tau }))$ is a linear isomorphism. 
Furthermore, $\Psi _{S_{\tau }}\circ M_{\tau }=M_{\tau }\circ \Psi _{S_{\tau }}$ 
on $\emLSfk .$ 
\item \label{t:mtauspec7}
$\Uvk={\cal U}_{v,\tau }(S_{\tau })=\bigcup _{L\in \emMin(\tau)}\Uvl=\bigcup _{L\in \emMin(\tau)}\{ a_{L}^{j}\} _{j=1}^{r_{L}}$ 
and $\dim _{\CC }(W_{\tau })=\dim _{\CC }(\emLSf )=\sum _{L\in \emMin(\tau)}r_{L}.$ 
\item \label{t:mtauspec7-1}
$\Uvak=\Uvk ,{\cal U}_{v,\tau, \ast }(S_{\tau })={\cal U}_{v,\tau}(S_{\tau })$,  and $\Uval=\Uvl$ for each $L\in \emMin(\tau).$ 
\item \label{t:mtauspec7-2} 
Let $L\in \emMin(\tau).$ 
Let $\Lambda _{r_{L}}:= \{ g_{1}\circ \cdots \circ g_{r_{L}}\mid 
\forall j, g_{j}\in \supp\,\tau \} .$ Moreover, 
let $G_{\tau }^{r_{L}}$ be the semigroup generated by the elements of 
$\Lambda _{r_{L}}$, i.e., 
$G_{\tau }^{r_{L}}=\{ h_{1}\circ \cdots \circ h_{n}\mid n\in \NN, h_{j}\in \Lambda _{r_{L}} (\forall j)\}.$ 
Then, $r_{L}=\sharp \emMin(G_{\tau }^{r_{L}},L).$  
\item \label{t:mtauspec8}
There exist a basis 
$\{ \varphi _{L,i,K}\mid L\in \emMin(\tau), i=1,\ldots ,r_{L}\} $ 
of $\emLSfk $ and a basis $\{ \rho _{L,i,K}\mid L\in \emMin(\tau), i=1,\ldots ,r_{L}\} $ 
of $\emLSfak $ such that for each $L\in \emMin(\tau)$ and for each 
$i=1,\ldots ,r_{L}$, we have all of the following statements 
{\em (a)--(f)}.  
\begin{itemize}
\item[{\em (a)}]
$M_{\tau }(\varphi _{L,i,K})=a_{L}^{i}\varphi _{L,i,K}.$ 
\item[{\em (b)}] 
$|\varphi _{L,i,K}||_{L}\equiv 1.$ 
\item[{\em (c)}] $\varphi _{L,i,K}|_{L'}\equiv 0$ for any  $L'\in \emMin(\tau)$ with $L'\neq L.$ 
\item[{\em (d)}] $\varphi _{L,i,K}|_{L}=(\varphi _{L,1,K}|_{L})^{i}.$  
\item[{\em (e)}] $\mbox{{\em supp}}\, \rho _{L,i,K}=L.$ 
\item[{\em (f)}] $\rho _{L,i,K}(\varphi _{L,j,K})=\delta _{ij}$ for each $j=1,\ldots ,r_{L}.$   
\end{itemize}
\item \label{t:mtauspecdual}
For each $\nu \in {\frak M}_{1}(K)$, 
$d_{0,K}((M_{\tau }^{\ast })^{n}(\nu ), \emLSfak \cap {\frak M}_{1}(K))\rightarrow 0$ as 
$n\rightarrow \infty. $ Moreover, 
$\dim _{T}(\emLSfak \cap {\frak M}_{1}(K))\leq 2\dim _{\CC }\emLSfk <\infty $, 
where $\dim _{T}$ denotes the topological dimension.  
\item \label{t:mtauspec9}
For each $L\in \emMin(\tau)$, we have that 
$T_{L,\tau }: Y \rightarrow [0,1]$ is continuous,     
$M_{\tau }(T_{L,\tau })=T_{L,\tau }$, 
$T_{L,\tau }\in W_{\tau },$ and $T_{L,\tau }|_{K}\in \mbox{{\em LS}}({\cal U}_{f,\tau}(K)).$    
Moreover, $\sum _{L\in \emMin(\tau)}T_{L,\tau }(z)=1$ for each 
$z\in Y.$ 
\item \label{t:mtauspecjconti}
For each $\beta \in (\supp\,\tau)^{\ZZ}$, 
$J_{\beta }^{+}$ is a closed subset of $Y.$ Moreover, 
there exists a Borel subset ${\cal B}$ of $(\supp\,\tau )^{\ZZ}$ with 
$\tau ^{\ZZ}({\cal B})=1$ such that for each $\gamma \in {\cal B}$, 
we have $Leb_{4}(J_{\gamma }^{+})=0$, 
and such that for each $\gamma \in {\cal B}$, 
the set-valued map 
$\beta \in (\supp\,\tau )^{\ZZ} \mapsto J_{\beta }^{+}$ is continuous 
at $\beta =\gamma $ (where we regard 
$J_{\beta }^{+}$ as a closed subset of $Y$). 
\item \label{t:mtauspec9-2}
If $\sharp \emMin (\tau )\geq 2$, then 
{\em (a)} for each $L\in \emMin (\tau)$, $T_{L,\tau }(J(G_{\tau }) )=[0,1]$, 
and {\em (b)} $\dim _{\CC} W_{\tau }=\dim _{\CC }(\emLSfk )>1.$   

\item \label{t:mtauspeccfi}
If $\dim _{\CC}W_{\tau }>1$ and {\em int}$(J(G_{\tau }))=\emptyset $, 
then $\sharp \mbox{{\em Con}}(F(G_{\tau }))=\infty .$ 

\item \label{t:mtauspec10}
$S_{\tau }= \{ \overline{z\in S_{\tau }\mid \exists g\in G_{\tau } \mbox{ {\em s.t.} } z \mbox{ is an attracting fixed point of }g} \} $, 
where the closure is taken in $\Pt $.
\item \label{t:mtauspec12} 
If $\dim _{\CC} (W_{\tau })>1$, then 
for any $\varphi \in (W_{\tau })_{nc}$, there exists  
an uncountable subset $A$ of $\CC $ such that 
for each $t\in A$, 
$\emptyset \neq \varphi ^{-1}(\{ t\} )\cap J(G_{\tau })\subset J_{res}(G_{\tau }).$ 
\end{enumerate}
\end{thm}
To prove Theorem~\ref{t:mtauspec}, we need the 
following Lemmas~\ref{l:minfin}--\ref{l:pfmtauspec12}.   
\begin{lem}
\label{l:minfin}
Under the assumptions of Theorem~\ref{t:mtauspec}, 
let $U_{1},\ldots, U_{m}, Q $ be the sets for $\tau $ 
coming from the assumption that $\tau $ is mean stable 
on $Y$ (see Definition~\ref{d:lc2ms}(5)). 
Then $\sharp \emMin (\tau)<\infty $ and 
$\emMin(\tau )= \emMin (G_{\tau }, \cup _{j}U_{j}).$ 
\end{lem}
\begin{proof}
Under the assumptions of our lemma, for each $(\gamma _{j})_{j\in \NN}\in (\supp\,\tau)^{\NN}$ 
and for each $j=1,\ldots, m$, 
diam$(\gamma _{n,1}(U_{j}))\rightarrow 0$ as 
$n\rightarrow \infty $. Therefore 
$\sharp \Min (G_{\tau }, \cup _{j}U_{j})\leq m <\infty .$ 
Moreover, since for each $a\in Y$ there exists a map $g_{a}\in G_{\tau }$ with $g_{a}(a)\in \cup _{j}U_{j}$, we have 
$\Min(\tau )=\Min (G_{\tau }, Y)=\Min (G_{\tau }, \cup _{j}U_{j}).$  
Thus, 
we have proved Lemma~\ref{l:minfin}. 
\end{proof}

\begin{lem}
\label{l:Kexists}
Under the assumption of Theorem~\ref{t:mtauspec}, 
for each neighborhood $B$ of $[1:0:0]$,    
there exists a compact subset  $K$ of 
$\Pt \setminus \{ [1:0:0]\}$ 
with 
$\Pt \setminus K\subset B$ such that 
$f(K)\subset K$ for each $f\in \supp\,\tau $ and such that 
$S_{\tau }\subset \mbox{int}(K).$  
\end{lem}
\begin{proof}
By Lemmas~\ref{l:ghsm} and \ref{l:minfin}, it is easy to see that the statement of 
our lemma holds. 
\end{proof}
Let $B$ be any neighborhood of $[1:0:0]$ in $\Pt.$ 
For the rest, let $K$ be a compact set in $\Pt $ 
with $\Pt \setminus K\subset B$ such that 
$f(K)\subset K$ for each $f\in \supp\,\tau$ and such that 
$S_{\tau }\subset \mbox{int}(K)$. 
Then, we may assume that 
$S_{\tau }\subset Q\subset \cup _{j=1}^{m}U_{j}\subset K$, where 
$U_{1},\ldots, U_{m} , Q$ are the sets for $\tau $ coming from the 
assumption that $\tau $ is mean stable on $Y$ 
(see Definition~\ref{d:lc2ms}(5)). 
\begin{lem}
\label{l:hypmetcont}
Under the assumptions of Theorem~\ref{t:mtauspec}, 
let $U_{1},\ldots, U_{m}, Q$ be the sets for $\tau $ coming 
from the definition of mean stability. 
Then, for each $\gamma =(\gamma _{j})_{j\in \NN} \in 
(\supp\,\tau)^{\NN}$ and for each $j$,   
$\mbox{diam}(\gamma _{n}\cdots \gamma _{1}(U_{j}))\rightarrow 
0 $ as $ n\rightarrow \infty $ and 
for each  $\gamma =(\gamma _{j})_{j\in \NN} \in 
(\supp\,\tau)^{\NN}$, for each $z\in \cup _{j}U_{j}$ 
$\| D(\gamma _{n,1})_{z}\| \rightarrow 0$ as 
$n\rightarrow \infty .$     
\end{lem}
\begin{proof}
The statement follows from the assumption that 
$\tau $ is mean stable on $Y.$  
\end{proof}
\begin{lem}
\label{l:azproof}
Under the assumptions of Theorem~\ref{t:mtauspec}, statement \ref{t:mtauspecaz} of Theorem~\ref{t:mtauspec} holds. 
\end{lem}
\begin{proof}
Let $U_{1},\ldots, U_{m}, Q$ be the sets for $\tau $ 
coming from the definition of mean stability. 
We may assume that $\cup _{j}U_{j}\subset K.$ 
Note that $f(\cup _{j}U_{j})\subset Q\subset \cup _{j}U_{j}$ 
for each $f\in \supp\,\tau $ 
by the mean stability of $\tau .$ 
Moreover, 
for each $a\in Y$, there exists a $g_{a}\in G_{\tau }$ with 
$g_{a}(a)\in \cup _{j}U_{j}.$ 
By this property and \cite[Lemma 4.6]{Splms10},   
for each $z\in K$, there exists a Borel subset 
${\cal A}_{z,K}$ of $(\supp\,\tau)^{\NN} $ with $\tau ^{\NN}({\cal A}_{z,K})=1$ such that 
for each $\gamma \in {\cal A}_{z,K}$, there exists an $n\in \NN $ such that 
$\gamma _{n,1}(z)\in \cup _{j}U_{j}.$ Since $K$ can be taken 
so that $\Pt \setminus K$ is an arbitrarily small neighborhood of 
$[1:0:0]$, 
for each $z\in Y$ 
there exists a Borel subset  ${\cal A}_{z}$ of $X_{\tau }$ with $\tau ^{\NN}({\cal A}_{z})=1$ such that 
for each $\gamma \in {\cal A}_{z}$, there exists an $n\in \NN $ such that 
$\gamma _{n,1}(z)\in \cup _{j}U_{j}.$
Then ${\cal A}_{z}$ satisfies the desired property. 
\end{proof}

\begin{lem}
\label{l:kcpt} Under the assumptions of Theorem~\ref{t:mtauspec}, 
statement~\ref{t:mtauspec4} of Theorem~\ref{t:mtauspec} holds.
\end{lem}
\begin{proof}Let $U_{1},\ldots, U_{m}, Q $ be the sets for $\tau $ 
coming from the definition of mean stability. 
By Lemma~\ref{l:minfin}, 
$S_{\tau }=\bigcup _{L\in \Min(\tau)} L$ is a compact 
subset of $\cup _{j}U_{j}(\subset F(G_{\tau })).$   
Moreover, $G_{\tau }(S_{\tau })\subset S_{\tau }$ and 
$G_{\tau }(F(G_{\tau }))\subset F(G_{\tau }).$  
Let $\Omega:= \bigcup _{A\in \mbox{Con}(F(G_{\tau })),A\cap S_{\tau }\neq \emptyset } A.$ 
Then, $G_{\tau }(\Omega )\subset \Omega.$ 
Since $\tau $ is mean stable on $Y$, 
for each $z_{0}\in Y$, there exists an element $g\in G_{\tau }$ such that 
$g(z_{0})\in \Omega.$ 
Combining this with $G_{\tau }(\Omega \cap K)\subset \Omega \cap K$ and 
\cite[Lemma 4.6]{Splms10},   
it follows that for each $z_{0}\in K$, there exists a Borel measurable subset 
${\cal V}_{z_{0},K}$ of $(\supp\,\tau)^{\NN}$ with $\tau ^{\NN}({\cal V}_{z_{0},K})=1$ such that 
for each $\gamma \in {\cal V}_{z_{0},K}$, 
there exists an $n\in \NN $ such that 
$\gamma _{n,1}(z_{0})\in \Omega \cap K\subset \Omega .$ 
Since $K$ can be taken so that $\Pt \setminus K$ is an arbitrarily 
small neighborhood of $[1:0:0]$, it follows that 
for each $z_{0}\in Y$, there exists a Borel measurable subset 
${\cal C}_{z_{0}}$ of $(\supp\,\tau )^{\NN}$ with $\tau ^{\NN}({\cal C}_{z_{0}})=1$ such that 
for each $\gamma \in {\cal C}_{z_{0}}$, 
there exists an $n\in \NN $ such that $\gamma _{n,1}(z_{0})\in \Omega.$ 
Also, for each $\gamma =(\gamma _{j})_{j\in \NN} \in (\supp\,\tau )^{\NN}$ and 
 for each $z\in \Omega$, 
$d(\gamma _{n,1}(z),S_{\tau })\rightarrow 0$ as $n\rightarrow \infty .$ 
It follows that  for each $\gamma \in {\cal C}_{z_{0}}$, 
$d(\gamma _{n,1}(z_{0}),S_{\tau })\rightarrow 0$ as $n\rightarrow \infty .$ 
Thus, we have proved our lemma. 
\end{proof}
\begin{lem}
\label{l:lscf}
Under the assumptions of Theorem~\ref{t:mtauspec}, 
$\emLSfk \subset \hat{C}_{F(G_{\tau })}(K).$ 
\end{lem}
\begin{proof}
Let $\varphi \in C(K)$ be such that $\varphi \neq 0$ and $M_{\tau }(\varphi )=a\varphi $ 
for some $a\in S^{1}.$ 
Since $\tau $ is mean stable on $Y$, we have 
$J_{\ker }(G_{\tau }|_{K})=\emptyset$, where 
$G_{\tau }|_{K}:=\{ g|_{K}\mid g\in G_{\tau }\} .$  
Combining this with 
\cite[Proposition 4.7, Lemma 4.2(2)(6)]{Splms10}, 
there exists a strictly increasing sequence $\{ n_{j}\} _{j=1}^{\infty }$ in $\NN $  and 
an element $\psi \in C(K)$ such that 
$M_{\tau }^{n_{j}}(\varphi )\rightarrow \psi $ 
and $a^{n_{j}}\rightarrow 1$ 
as $j\rightarrow \infty $. Thus,  
$\varphi =\frac{1}{a^{n_{j}}}M_{\tau }^{n_{j}}(\varphi )\rightarrow \psi $ 
in $C(K)$ as $j\rightarrow \infty .$ 
Therefore, $\varphi =\psi .$ 
Let $U\in \mbox{Con}(F(G_{\tau }))$ and let $x,y\in U\cap K.$ 
By Lemma~\ref{l:azproof}, 
we have 
$\psi (x)-\psi (y)=\lim _{j\rightarrow \infty }(M_{\tau }^{n_{j}}(\varphi )(x)-M_{\tau }^{n_{j}}(\varphi )(y))
=0.$ Therefore, $\varphi =\psi \in \hat{C}_{F(G_{\tau })}(K).$ 
Thus, we have proved our lemma.  
\end{proof}
\begin{lem}
\label{l:lspec}
Under the assumptions of Theorem~\ref{t:mtauspec}, 
statement~\ref{t:mtauspec5} holds.
\end{lem}
\begin{proof}
Let $L\in \Min(\tau ).$ 
Let $\varphi \in \Ufl $ be such that $M_{\tau }(\varphi )=a\varphi $ for some $a\in S^{1}$ 
and $\sup _{z\in L}|\varphi (z)|=1.$ 
Let $\Theta := \{ z\in L\mid |\varphi (z)|=1\} .$ 
For each $z\in L$, we have 
$|\varphi (z)|= |M_{\tau }(\varphi )(z)|\leq M_{\tau }(|\varphi |)(z)\leq 1.$ 
Thus, $G_{\tau }(\Theta )\subset \Theta .$ 
Since $L\in \Min(\tau )$, $\overline{G_{\tau }(z)}=L$ for each 
$z\in \Theta .$ Hence, we obtain $\Theta =L.$ 
 Using the argument of the proof of Lemma~\ref{l:lscf}, 
it is easy to see the following claim. \\ 
Claim 1: 
For each $A\in \mbox{Con}(F(G_{\tau }))$ with $A\cap L\neq \emptyset $, 
$\varphi |_{A\cap L}$ is constant and 
$|\varphi ||_{A\cap L}\equiv 1.$ 

 Let $A_{0}\in \mbox{Con}(F(G_{\tau }))$ be an element with $A_{0}\cap L\neq \emptyset $ and 
 let $z_{0}\in A_{0}\cap L$ be a point. We now show the following claim. \\
Claim 2: The map $h\mapsto \varphi (h(z_{0})), h\in \supp\,\tau $, is constant. 

 To show this claim, by claim 1 and that 
 $\bigcup _{h\in \supp\,\tau }\{ h(z_{0})\} $ is a  
compact subset of $F(G_{\tau })$, we obtain that 
$\varphi (z_{0})=\frac{1}{a}M_{\tau }(\varphi )(z_{0})$ is equal to a finite convex combination 
of elements of $S^{1}.$  Since $|\varphi (z_{0})|=1$, it follows that 
$h\mapsto \varphi (h(z_{0})), h\in \supp\,\tau $ is constant. Thus, Claim 2 holds. 

 By Claim 2 and $M_{\tau }(\varphi )=a\varphi $, we immediately obtain the following claim.\\ 
Claim 3: For each $h\in \supp\,\tau $, $\varphi (h(z_{0}))=a\varphi (z_{0}).$ 

Since $L\in \Min(\tau )$, $\overline{G_{\tau }(z_{0})}=L.$ 
Hence there exist an $l\in \NN $ and an element $\beta =(\beta _{1},\ldots ,\beta _{l})\in 
(\supp\,\tau )^{l}$ such that $\beta _{l}\cdots \beta _{1}(z_{0})\in A_{0}.$ 
From Claim 3, it follows that $a^{l}=1.$ 
Thus, we have shown that $\Uvl\subset \{ a\in S^{1}\mid a^{l}=1\} .$ 
Moreover, by Claims 1, 3, and the previous argument, we obtain that if $\varphi _{1},\varphi _{2}\in C(L)$ 
with $\sup _{z\in L}|\varphi _{i}(z)|=1$,  
$a_{1},a_{2}\in S^{1}$, and $M_{\tau }(\varphi _{i})=a_{i}\varphi _{i}$, then 
$|\varphi _{i}|\equiv 1$, 
$M_{\tau }(\varphi _{1}\varphi _{2})=a_{1}a_{2}\varphi _{1}\varphi _{2}$, and 
$M_{\tau }(\varphi _{1}^{-1})=a_{1}^{-1}\varphi _{1}^{-1}.$ 
From these arguments, it follows that $\Ufl$ is a finite subgroup of $S^{1}.$ 
Let $r_{L}:= \sharp \Ufl .$ 
Let $a_{L}\in \Ufl$ be an element such that 
$\{ a_{L}^{j}\mid j=1,\ldots, r_{L}\}=\Ufl .$ 
By Claim 3 and $\overline{G_{\tau }(z_{0})}=L$, we obtain that 
any element $\varphi \in C(L)$ satisfying $M_{\tau }(\varphi )=a_{L}^{j}\varphi $ 
is uniquely determined by the constant $\varphi |_{A_{0}\cap L}.$ 
Thus, for each $j=1,\ldots r_{L}$, there exists a unique 
$\psi _{L,j}\in \Ufl $ such that 
$M_{\tau }\psi _{L,j}=a_{L}^{j}\psi _{L,j}$ and $\psi _{L,j}|_{A_{0}\cap L}\equiv 1.$ 
It is easy to see that $\{ \psi _{L,j}\} _{j=1}^{r_{L}}$ is a basis of 
$\LSfl $. Moreover, by the previous argument, we obtain that 
$\psi _{L,j}=(\psi _{L,1})^{j}$ for each $j=1,\ldots ,r_{L}.$ 
Thus, we have proved our lemma.   
\end{proof}
\begin{lem}
\label{l:lsfklsfl}
Under the assumptions and notation of Theorem~\ref{t:mtauspec}, 
the map 
$\alpha : \mbox{{\em LS}}({\cal U}_{f,\tau}(S_{\tau })) \rightarrow \oplus _{L\in \emMin(\tau)}\emLSfl$ 
defined by $\alpha (\varphi )=(\varphi |_{L})_{L\in 
\emMin (\tau )}$ 
is a linear isomorphism.  
\end{lem}
\begin{proof}
By Lemma~\ref{l:minfin}, $\sharp \Min(\tau)<\infty .$ 
Moreover, elements of $\Min(\tau)$ are mutually disjoint. 
Furthermore, for each $L\in \Min (\tau)$ and for each $\varphi \in C(S_{\tau })$, 
$(M_{\tau }(\varphi ))|_{L}=M_{\tau }(\varphi |_{L}).$ 
Thus, we easily see that the statement of our lemma holds. 
\end{proof}
\begin{lem}
\label{l:phiinj}
Under the assumptions and notation of Theorem~\ref{t:mtauspec}, \\ 
$\Psi _{S_{\tau }}(\emLSfk )\subset  \mbox{{\em LS}}({\cal U}_{f,\tau}(S_{\tau }))$ and 
$\Psi _{S_{\tau }}: \emLSfk \rightarrow \mbox{{\em LS}}({\cal U}_{f,\tau}(S_{\tau }))$ is injective.   
\end{lem}
\begin{proof}
We first prove the following claim. 

\noindent Claim 1: $\Psi _{S_{\tau }}: \LSfk \rightarrow C(S_{\tau })$ is injective. 

To prove this claim, let $\varphi \in \Ufk $ and let $a\in S^{1}$ with 
$M_{\tau }(\varphi )=a\varphi $ and suppose $\varphi |_{S_{\tau }}\equiv 0.$ 
Let $\{ n_{j}\} _{j=1}^{\infty }$ be a sequence in $\NN $ such that 
$a^{n_{j}}\rightarrow 1$ as $j\rightarrow \infty .$ 
By Lemma~\ref{l:kcpt}, it follows that 
$\varphi =\frac{1}{a^{n_{j}}}M_{\tau }^{n_{j}}(\varphi )\rightarrow 0$ as $j\rightarrow \infty .$ 
Thus, $\varphi =0$. However, this is a contradiction. Therefore, Claim 1 holds. 

 The statement of our lemma easily follows from Claim 1. 
 Thus, we have proved our lemma.
%
\end{proof}
\begin{lem}
\label{l:cclsfcb0}
Under the assumptions and notation of Theorem~\ref{t:mtauspec}, 
${\cal B}_{0,\tau , K}$ is a closed subspace of $C(K)$ and  
there exists a direct sum decomposition 
$C(K)=\emLSfk \oplus {\cal B}_{0,\tau , K}.$ 
Moreover, $\dim _{\CC }(\emLSfk )<\infty $ and 
the projection $\pi _{\tau, K}:C(K)\rightarrow \emLSfk $ is continuous.   
Furthermore, setting $r:=\prod _{L\in \emMin (G_{\tau }, Y)}r_{L} $, we have that  
for each $\varphi \in \emLSfk $, $M_{\tau }^{r}(\varphi )=\varphi .$ 
\end{lem}
\begin{proof}
Since $\tau $ is mean stable on $Y$, we have 
$J_{\ker }(G_{\tau }|_{K})=\emptyset.$ 
Combining this with \cite[Proposition 4.7, Lemma 4.2(2)(6)]{Splms10}, 
we obtain that 
for each $\varphi \in C(K)$, 
$\overline{\bigcup _{n=1}^{\infty }\{ M_{\tau }^{n}(\varphi )\} }$ is 
compact in $C(K).$ 
By \cite[p.352]{Ly}, 
it follows that there exists a direct sum decomposition 
$C(K) =\overline{\LSfk} \oplus  {\cal B}_{0,\tau ,K}.$ 
Moreover, combining Lemma~\ref{l:phiinj},  
Lemma~\ref{l:lspec}, and Lemma~\ref{l:lsfklsfl}, 
we obtain that  
$$\dim _{\CC }(\LSfk )<\infty $$ and for each $\varphi \in \LSfk ,$ $M_{\tau }^{r}(\varphi )=\varphi .$ 
Hence, there exists a direct sum decomposition 
$C(K)=\LSfk \oplus {\cal B}_{0,\tau ,K}.$ 
Since ${\cal B}_{0,\tau ,K}$ is closed in $C(K)$ and 
$\dim _{\CC }(\LSfk )<\infty $, 
it follows that the projection $\pi _{\tau, K}: C(K)\rightarrow \LSfk $ is continuous. 
Thus, we have proved our lemma.  
\end{proof}
\begin{lem}
\label{l:lsiso}
Under the assumptions and notation of Theorem~\ref{t:mtauspec}, 
statement~\ref{t:mtauspec6} holds. 
\end{lem}
\begin{proof}
It is easy to see that $\Psi _{S_{\tau }}\circ M_{\tau }=M_{\tau }\circ \Psi _{S_{\tau }}$ 
on $\LSfk .$ To prove our lemma, 
by Lemma~\ref{l:phiinj}, it is enough to show that 
$\Psi _{S_{\tau }}: \LSfk \rightarrow 
\mbox{{\em LS}}({\cal U}_{f,\tau }(S_{\tau })) \cong \oplus _{L\in \Min(\tau)}\LSfl $ 
is surjective. To show this, 
let $L\in \Min(\tau)$ and let $a_{L},r_{L}$, and $\{ \psi _{L,j}\} _{j=1}^{r_{L}}$ 
be as in Lemma~\ref{l:lspec} (statement~\ref{t:mtauspec5} of Theorem~\ref{t:mtauspec}).  
Let $\tilde{\psi }_{L,j}\in C(K)$ be an element such that 
$\tilde{\psi }_{L,j}|_{L}=\psi _{L,j}$  
and 
$\tilde{\psi }_{L,j}|_{L'}\equiv 0$ for each $L'\in \Min(\tau )$ with 
$L'\neq L.$ 
Let $r$ be the number in Lemma~\ref{l:cclsfcb0} and 
let $\pi _{\tau, K}: C(K)\rightarrow \LSfk $ be the projection. 
Then,  
$M_{\tau }^{rn}(\tilde{\psi }_{L,j})\rightarrow \pi _{\tau, K}(\tilde{\psi }_{L,j})$ 
in $C(K)$ 
as 
$n\rightarrow \infty .$ Therefore, 
$\pi _{\tau, K}(\tilde{\psi }_{L,j})|_{L}=\lim _{n\rightarrow \infty }
M_{\tau }^{rn}(\tilde{\psi }_{L,j}|_{L})=\psi _{L,j}.$ 
Similarly, $\pi _{\tau, K}(\tilde{\psi }_{L,j})|_{L'}\equiv 0$ for each 
$L'\in \Min(\tau )$ with $L'\neq L.$ 
Therefore, $\Psi _{S_{\tau }}: \LSfk \rightarrow 
\mbox{LS}({\cal U}_{f,\tau }(S_{\tau })) $ is surjective. 
Thus, we have completed the proof of our lemma. 
\end{proof}
\begin{rem}
\label{r:lkncpty}
Let $\{ K_{n}\}$ be an increasing sequence of 
compact subsets of $Y$ with 
$Y=\cup _{n=1}^{\infty }K_{n}$ such that 
for each $n\in \NN $ and for each 
$g\in \supp\,\tau$, we have 
$f(K_{n})\subset K_{n}$, and such that 
$S_{\tau }\subset Q\subset \cup _{j=1}^{m}U_{j}\subset K_{1}
\subset K_{2}\subset \cdots .$ Then, for each $n\in \NN$, 
the statements of Lemmas~\ref{l:azproof}--\ref{l:lsiso} 
with $K=K_{n}$  
hold. For each $n\in \NN$, let 
$W_{\tau, n}:=\mbox{LS}({\cal U}_{f,\tau}(K_{n}))\ (\subset 
C(K_{n})).$ 
Then, for each $n\in \NN$, 
by Lemma~\ref{l:lsiso} with $K=K_{n}$, the map 
$\varphi \in W_{\tau, n+1}\mapsto \varphi |_{K_{n}}\in W_{\tau, n}$ 
is a linear isomorphism. Hence,  
there exists the unique finite subspace $W_{\tau }$ of $C(Y)$ such that 
for each $n\in \NN$, we have $W_{\tau }|_{K_{n}}=
W_{\tau ,n}=\mbox{LS}({\cal U}_{f,\tau}(K_{n}))$ and 
the map $\varphi \in W_{\tau }\mapsto 
\varphi |_{K_{n}}\in W_{\tau ,n}=\mbox{LS}({\cal U}_{f,\tau}(K_{n}))$ 
is a linear isomorphism. Here, we set 
$W_{\tau}|_{K_{n}}:=\{ \varphi |_{K_{n}}\mid \varphi 
\in W_{\tau }\}.$ 
Since $M_{\tau }(W_{\tau,n})=W_{\tau, n}$ for each $n\in \NN$, 
we have  
$M_{\tau }(W_{\tau })=W_{\tau }.$ 
\end{rem}
\begin{lem}
\label{l:pf7}
Under the assumptions of Theorem~\ref{t:mtauspec}, 
statement \ref{t:mtauspec7} holds.
\end{lem}
\begin{proof}
Statement~\ref{t:mtauspec7} of Theorem~\ref{t:mtauspec} follows 
from Lemmas~\ref{l:lspec}, \ref{l:lsfklsfl}, \ref{l:lsiso},   and 
Remark~\ref{r:lkncpty}. 
\end{proof}   
\begin{lem}
\label{l:pf2-1}
Under the assumptions of Theorem~\ref{t:mtauspec}, 
statement~\ref{t:mtauspec2-1} holds. 
\end{lem}
\begin{proof}
Let $\{ \varphi _{j}\} $ and $\{ \alpha _{j}\} $ be as in 
statement~\ref{t:mtauspec2-1} of Theorem~\ref{t:mtauspec}. 
Let $\varphi \in C(K)$. 
Then, there exists a unique family $\{ \rho _{j}(\varphi )\} _{j=1}^{q}$ in $\CC $  
such that $\pi _{\tau, K}(\varphi )=\sum _{j=1}^{q}\rho _{j}(\varphi )\varphi _{j}$, where $\pi _{\tau ,K}:C(K)\rightarrow 
\mbox{LS}({\cal U}_{f,\tau }(K))$ is the projection.  
It is easy to see that $\rho _{j}:C(K)\rightarrow \CC $ is a linear 
functional. 
Moreover, since $\pi _{\tau, K}:C(K)\rightarrow \LSfk $ 
is continuous (Lemma~\ref{l:cclsfcb0}), 
each $\rho _{j}:C(K)\rightarrow \CC $ is continuous. 
By Lemma~\ref{l:cclsfcb0} again, 
it is easy to see that $\rho _{i}(\varphi _{j})=\delta _{ij}.$ 
To show $M_{\tau }^{\ast }(\rho _{j})=\alpha _{j}\rho _{j}$, 
let $\varphi \in C(K)$ and let $\zeta := \varphi -\pi_{\tau ,K} (\varphi ).$ 
Then,  
$M_{\tau }(\varphi )=\sum _{j=1}^{q}\rho _{j}(\varphi )\alpha _{j}\varphi _{j}+M_{\tau }(\zeta ).$ 
Hence, $\rho _{j}(M_{\tau }(\varphi ))=\alpha _{j}\rho _{j}(\varphi ).$ 
Therefore, $M_{\tau }^{\ast }(\rho _{j})=\alpha _{j}\rho _{j}.$ 
To prove that $\{ \rho _{j}\} $ is a basis of 
$\LSfak$, let $\rho \in \Ufak $ and $a\in S^{1}$ be such that 
$M_{\tau }^{\ast }(\rho )=a\rho .$ 
Let $r$ be the number in Lemma~\ref{l:cclsfcb0}. 
Let $\{ n_{i}\} _{i=1}^{\infty }$ be a strictly increasing sequence in $\NN $ such that 
$a^{rn_{i}}\rightarrow 1$ as $i\rightarrow \infty .$ 
Let $\varphi \in C(K) $ and let 
$\zeta =\varphi -\pi (\varphi ).$ 
Then, $\rho (\varphi )=\frac{1}{a^{rn_{i}}}(M_{\tau }^{\ast })^{rn_{i}}(\rho )(\varphi )
=\frac{1}{a^{rn_{i}}}\rho (\sum _{j=1}^{q}\rho _{j}(\varphi )\varphi _{j}+
M_{\tau }^{rn_{i}}(\zeta ))\rightarrow \sum _{j=1}^{q}\rho _{j}(\varphi )
\rho (\varphi _{j})$ as $i\rightarrow \infty .$ 
Therefore, $\rho \in \mbox{LS}(\{ \rho _{j} \mid j=1\ldots, q\}). $ 
Thus, $\{ \rho _{j}\} _{j=1}^{q}$ is a basis of 
$\LSfak. $ 
To prove supp$\, \rho _{j}\subset S_{\tau }$, 
let $\varphi \in C(K)$ be such that supp$\, \varphi \subset K \setminus S_{\tau }.$ 
Let $\zeta =\varphi -\pi (\varphi ).$ Then,  
$\varphi =\sum _{j=1}^{q}\rho _{j}(\varphi )\varphi _{j}+\zeta .$ 
Let $r$ be the number in Lemma~\ref{l:cclsfcb0}. 
Then, $M_{\tau }^{rn}(\varphi )\rightarrow \sum _{j=1}^{q}\rho _{j}(\varphi )\varphi _{j}$ 
as $n\rightarrow \infty .$ 
Hence, $\sum _{j=1}^{q}\rho _{j}(\varphi )\varphi _{j}|_{S_{\tau }}=
\lim _{n\rightarrow \infty }M_{\tau }^{rn}(\varphi |_{S_{\tau }})=0.$ 
By Lemma~\ref{l:phiinj}, we obtain $\rho _{j}(\varphi )=0$ for each $j.$ 
Therefore, supp$\, \rho _{j}\subset S_{\tau }$ for each $j.$ 
Thus, we have completed the proof of our lemma. 
\end{proof}
\begin{lem}
\label{l:pf7-1}
Under the assumptions of Theorem~\ref{t:mtauspec}, 
statement~\ref{t:mtauspec7-1} holds. 
\end{lem} 
\begin{proof}
By items (b) and (c) of statement~\ref{t:mtauspec2-1} of Theorem~\ref{t:mtauspec} 
(see Lemma~\ref{l:pf2-1}), 
we obtain $\Uvak =\Uvk. $ Using the same method as that in the proof of Lemma~\ref{l:pf2-1},  
we obtain ${\cal U}_{v,\tau, \ast}(S_{\tau })=
{\cal U}_{v,\tau}(S_{\tau })$ and $\Uval=\Uvl $ 
for each $L\in \Min(\tau).$ 
Thus, we have completed the proof of our lemma. 
\end{proof}
\begin{lem}
\label{l:suppllj}
Under the assumptions of Theorem~\ref{t:mtauspec}, 
statements~\ref{t:mtauspec7-2} and \ref{t:mtauspec8} of Theorem~\ref{t:mtauspec} hold. 
\end{lem}
\begin{proof}
Let $L\in \Min(\tau)$ and 
let $z_{0}\in L\subset F(G_{\tau })$ be a point. 
Let $\{ \psi _{L,j}\} _{j=1}^{r_{L}}$ be as in statement \ref{t:mtauspec5} of 
Theorem~\ref{t:mtauspec}. 
We may assume $\psi _{L,1}(z_{0})=1.$ 
For each $j=1,\ldots ,r_{L}$,  
let $L_{j}:= \{ z\in L\mid \psi _{L,1}(z)=a_{L}^{j}\} .$ 
By Claim 3 in the proof of Lemma~\ref{l:lspec}, 
$L$ is equal to the disjoint union of 
compact subsets $L_{j}$, and 
for each $h\in \supp\,\tau $ and 
for each $j=1,\ldots ,r_{L}$,  
$h(L_{j})\subset L_{j+1}$, where $L_{r_{L}+1}:=L_{1}.$ 
In particular, 
$G_{\tau }^{r_{L}}(L_{j})\subset L_{j}$ for each 
$j=1,\ldots ,r_{L}.$ 

Since $\overline{G_{\tau }(z)}=L$ for each 
$z\in L$, it follows that $\overline{G_{\tau }^{r_{L}}(z)}=L_{j}$ for each 
$j=1,\ldots ,r_{L}$ and for each $z\in L_{j}.$ 
Therefore, $\{ L_{j}\mid j=1,\ldots, r_{L}\}=
\Min (G_{\tau }^{r_{L}}, L).$ 
Thus, statement~\ref{t:mtauspec7-2} of Theorem~\ref{t:mtauspec} holds. 

To prove statement \ref{t:mtauspec8} of Theorem~\ref{t:mtauspec}, let $j\in \{ 1,\ldots ,r_{L}\} .$ 
Let us consider the argument in the proof of Lemma~\ref{l:lspec}, replacing $L$ by $L_{j}$ and 
$G_{\tau }$ by $G_{\tau }^{r_{L}}$. Then, the number $r_{L}$ in the proof of Lemma~\ref{l:lspec}
is equal to $1$ in this case. 
For, if there exists a nonzero element $\psi \in C(L_{j})$ and a $b=e^{\frac{2\pi i}{s}}\neq 1$ with 
$s\in \NN $ such that $M_{\tau }^{r_{L}}(\psi ) =b\psi $, 
then extending $\psi $ to the element $\tilde{\psi }\in C(L)$ by setting 
$\tilde{\psi }|_{L_{i}}=0$ for each $i$ with $i\neq j$, 
and setting $\hat{\psi }:= \sum _{j=1}^{sr_{L}}(e^{\frac{2\pi i}{sr_{L}}})^{-j}M_{\tau }^{j}(\tilde{\psi })\in C(L)$, 
we obtain $\hat{\psi }\neq 0$ and $M_{\tau }(\hat{\psi })=e^{\frac{2\pi i}{sr_{L}}}\hat{\psi }$, 
which is a contradiction.    
Therefore, using the argument in the proof of Lemmas~\ref{l:lspec} and \ref{l:cclsfcb0}, 
we see that for each $\varphi \in C(L_{j})$, 
there exists a number $\omega _{L,j}(\varphi )\in \CC $ 
such that 
$M_{\tau }^{nr_{L}}(\varphi )\rightarrow \omega _{L,j}(\varphi )\cdot 1_{L_{j}}$ 
as $n\rightarrow \infty .$ 
It is easy to see that $\omega _{L,j}$ is a positive linear functional. 
Therefore, $\omega _{L,j}\in {\frak M}_{1}(L_{j}).$ 
Thus, $\omega _{L,j}$ is the unique $(M_{\tau }^{\ast })^{r_{L}}$-invariant element 
of ${\frak M}_{1}(L_{j}).$ 
Since $L_{j}\in \Min(G_{\tau }^{r_{L}},L)$, 
it is easy to see that supp$\, \omega _{L,j}=L_{j}.$ 
Since $M_{\tau }^{\ast }(\omega _{L,j})\in {\frak M}_{1}(L_{j+1})$ and 
$(M_{\tau }^{\ast })^{r_{L}}(M_{\tau }^{\ast }(\omega _{L,j}))=
M_{\tau }^{\ast }(\omega _{L,j})$, 
it follows that $M_{\tau }^{\ast }(\omega _{L,j})=\omega _{L,j+1}$ for each 
$j=1,\ldots ,r_{L}$, where $\omega _{L,r_{L}+1}:=\omega _{L,1}.$ 
For each 
$i=1,\ldots, r_{L}$, 
let $\rho _{L,i}:= \frac{1}{r_{L}}\sum _{j=1}^{r_{L}}a_{L}^{-ij}\omega _{L,j}\in 
C(L)^{\ast }$  
and $\tilde{\psi } _{L,i}:=\sum _{j=1}^{r_{L}}a_{L}^{ij}1_{L_{j}}\in C(L).$ 
Then, it is easy to see that 
$M_{\tau }^{\ast }(\rho _{L,i})=a_{L}^{i}\rho _{L,i}$,  
$M_{\tau }(\tilde{\psi }_{L,i})=a_{L}^{i}\tilde{\psi }_{L,i}$, and 
$\rho _{L,i}(\tilde{\psi }_{L,j})=\delta _{ij}.$ 
By Lemma~\ref{l:lsiso}, 
there exists a unique element $\varphi _{L,i,K}\in \LSfk$ such that 
$\varphi _{L,i,K}|_{L}=\tilde{\psi }_{L,i}$ and $\varphi _{L,i,K}|_{L'}\equiv 0$ for 
each $L'\in \Min (\tau)$ with $L'\neq L.$ 
Let $\rho _{L,i,K}\in C(K)^{\ast }$ be the unique element 
such that $\rho _{L,i,K}(\varphi )=\rho _{L,i}(\varphi |_{L})$ for each 
$\varphi \in C(K).$ 
Then, it is easy to see that $\{ \varphi _{L,i,K}\} _{L,i}$ and 
$\{ \rho _{L,i,K}\} _{L,i}$ 
are the desired families. Thus, we have completed the proof of our lemma. 
\end{proof} 
\begin{lem}
\label{l:mtdualpf}
Under the assumptions of Theorem~\ref{t:mtauspec}, 
statement~\ref{t:mtauspecdual} holds. 
\begin{proof}
Statement~\ref{t:mtauspecdual} follows from Lemma~\ref{l:pf2-1}. 
\end{proof}

\end{lem}
\begin{lem}
\label{l:mt9pf}
Under the assumptions and notation of Theorem~\ref{t:mtauspec}, 
statement~\ref{t:mtauspec9} holds.  
\end{lem}
\begin{proof}
By Lemma~\ref{l:kcpt}, 
we have $\sum _{L\in \Min(\tau)}T_{L,\tau }(z)=1$ for each 
$z\in Y.$ 
For each $L\in \Min(\tau)$, 
let $V_{L}$ be an open neighborhood of $L$ in $Y$ 
such that $\overline{V_{L}}\cap \overline{V_{L'}}=\emptyset $ whenever 
$L, L\in \Min(\tau ), L\neq L'.$  Here, 
$\overline{V_{L}}$ (resp. $\overline{V_{L'}}$) denotes the closure 
of $V_{L}$ (resp. $V_{L'}$) in $\Pt.$ 
For each $L\in \Min(\tau)$, 
let $\varphi _{L}\in C(K)$ be an element 
such that $\varphi _{L}|_{V_{L}\cap K}\equiv 1$ and 
$\varphi _{L}|_{\bigcup _{L'\neq L}V_{L'}\cap K}\equiv 0.$ 
From Lemma~\ref{l:kcpt} 
it follows that 
\begin{equation}
\label{eq:tltinv}
T_{L,\tau }(z)=\int _{(\supp\,\tau )^{\NN}}\lim _{n\rightarrow \infty }\varphi _{L}(\gamma _{n,1}(z))\ 
d\tau ^{\NN}(\gamma )=\lim _{n\rightarrow \infty }\int _{(\supp\,\tau)^{\NN}}\varphi _{L}(\gamma _{n,1}(z))\ 
d\tau ^{\NN}(\gamma )
 =\lim _{n\rightarrow \infty }M_{\tau}^{n}(\varphi _{L})(z)
\end{equation}
 for each $z\in K.$ 
 Since $\tau $ is mean stable on $Y$, we have 
$J_{\ker }(G_{\tau }|_{K})=\emptyset.$ 
Combining this with \cite[Proposition 4.7, Lemma 4.2(2)(6)]{Splms10}, 
we obtain that  
for each $\varphi \in C(K)$, 
$\overline{\bigcup _{n=1}^{\infty }\{ M_{\tau }^{n}(\varphi )\} }$ is 
compact in $C(K).$ 
Combining this with (\ref{eq:tltinv}),  
we obtain that 
$T_{L,\tau }|_{K}\in C(K)$ for each $L\in \Min(\tau )$. 
Since $K$ can be taken so that $\Pt \setminus K$ is an arbitrarily small neighborhood of $[1:0:0]$, it follows that 
$T_{L,\tau }$ is continuous on $Y$ for each 
$L\in \Min(\tau ).$  
Moreover, from (\ref{eq:tltinv}) again, we obtain 
$M_{\tau }(T_{L,\tau })=T_{L,\tau }.$

Thus, we have proved our lemma. 
\end{proof} 
\begin{lem}
\label{l:pfmtauspecjconti}
Under the assumptions of Theorem~\ref{t:mtauspec}, 
statement~\ref{t:mtauspecjconti} holds. 
\end{lem}
\begin{proof}
Let $\beta \in (\supp\,\tau )^{\ZZ}.$ 
Then, by Lemma~\ref{l:igkg} and Lemma~\ref{l:ghsm} (1), 
the closure of $J_{\beta }^{+}$ in $Y$ does not meet 
$\Poi \setminus \{ [1:0:0]\}.$ 
Hence, $J_{\beta }^{+}$ is a closed subset of $Y.$ 

By Proposition~\ref{p:Ggcontiph}(viii), 
the map $\beta \in (\supp\,\tau)^{\ZZ}
\rightarrow J_{\beta }^{+}$ has the lower semicontinuity 
in Definition~\ref{d:contiseqsets}, where 
we regard the sets $J_{\beta }^{+}$ as closed subsets of $Y.$ 

As in Remark~\ref{r:lkncpty}, 
let $\{ K_{n}\}$ be an increasing sequence of 
compact subsets of $Y$ with 
$Y=\cup _{n=1}^{\infty }K_{n}$ such that 
for each $n\in \NN $ and for each 
$g\in \supp\,\tau$, we have 
$g(K_{n})\subset K_{n}$, and such that 
$S_{\tau }\subset Q\subset \cup _{j=1}^{m}U_{j}\subset K_{1}
\subset \mbox{int }(K_{2})\subset K_{2}\subset 
\mbox{int}(K_{3})\subset K_{3}\subset \cdots .$ Then for each $n\in \NN$, 
the statements of Lemmas~\ref{l:azproof}--\ref{l:lsiso} 
with $K=K_{n}$  
hold. Also, we have that 
for each $\beta \in (\supp\,\tau)^{\ZZ}$, 
\begin{equation}
\label{eq:jrhopikn}
J_{\beta }^{+}=\cup _{n=1}^{\infty }(J_{\beta }^{+}\cap 
\mbox{int}(K_{n})).
\end{equation} 
By Lemma~\ref{l:igkg}, we have the following claim. 

Claim 1. 
For each 
$n\in \NN $ and for each $\beta =(\beta _{j})_{j\in \ZZ}\in (\supp\,\tau)^{\ZZ}$,  
$J_{\beta }^{+}\cap \mbox{int}(K_{n})$ 
is equal to the set of elements 
$z\in \mbox{int}(K_{n})$ for which 
there exists no open neighborhood $U$ of $z$ 
in int$(K_{n})$ such that $\{ \beta _{m,1}: U\rightarrow K_{n}\} _{m=1}^{\infty }$ 
is equicontinuous on $U.$

Let $n\in \NN.$ Since 
$\tau $ is mean stable on $Y$, 
for each $z\in K_{n}$, there exists a map 
$h_{z}\in G_{\tau }$ with $h_{z}(z)\in K_{n}\cap F(G_{\tau }).$ 
Therefore, $J_{\ker }(G_{\tau }|_{K_{n}})=\emptyset $,  
where $G_{\tau }|_{K_{n}}=\{ g|_{K_{n}} \mid g\in G_{\tau }\}.$ 
Combining the above argument  with Claim 1 and \cite[Proposition 4.8]{Splms10}, 
we obtain that for $\tau ^{\ZZ}$-a.e. $\gamma \in (\supp\,\tau)^{\ZZ}$, 
$\mbox{Leb}_{4}(J_{\gamma }^{+}\cap \mbox{int}(K_{n}))=0.$ 
Combining this with (\ref{eq:jrhopikn}), we see that 
there exists a Borel subset ${\cal B}_{1}$ of 
$(\supp\,\tau)^{\ZZ}$ with $\tau ^{\ZZ}({\cal B}_{1})=1$ 
such that for each $\gamma \in {\cal B}_{1}$, 
we have $\mbox{Leb}_{4}(J_{\gamma }^{+})=0$.

Let $\{ x_{n}\}_{n=1}^{\infty }$ be a sequence in $Y$ 
such that $\overline{\{x_{n}\mid n\in \NN\}}=Y.$ 
By Lemma~\ref{l:kcpt}, 
there exists a Borel subset ${\cal B}_{2}$ of $(\supp\,\tau )^{\ZZ}$ 
with $\tau ^{\ZZ}({\cal B}_{2})=1$ such that 
for each $\gamma =(\gamma _{j})_{j\in \ZZ }\in 
{\cal B}_{2}$ and for each $p\in \NN$, we have 
\begin{equation}
\label{eq:dgamman1xp}
d(\gamma _{n,1}(x_{p}), S_{\tau })\rightarrow 0 \mbox{ as }  
n\rightarrow \infty .
\end{equation}  

 Let $\gamma =(\g _{j})_{j\in \ZZ} \in {\cal B}_{2}.$ 
 We want to show  the upper semicontinuity 
 of the map 
 $\beta \in (\supp\,\tau )^{\ZZ} \mapsto J_{\beta }^{+}$ 
 at $\beta =\gamma $.  
 To do this, 
 let $\{ \beta ^{n}\} _{n=1}^{\infty }$ 
 be a sequence in $(\supp\,\tau )^{\ZZ}$ 
 which converges to $\gamma $ in $(\supp\,\tau )^{\ZZ}.$ 
 Also, let $\{ z_{n}\} _{n=1}^{\infty }$ be a sequence 
 in $\Ct $ such that $z_{n}\in J_{\beta ^{n}}^{+}$ for each 
 $n\in \NN.$ Suppose that 
 there exists a point $z_{0}\in Y$ such that 
 $z_{n}\rightarrow z_{0}$ as $n\rightarrow \infty .$ 
 To prove  the upper semicontinuity 
 of the map 
 $\beta \in (\supp\,\tau )^{\ZZ} \mapsto J_{\gamma }^{+}$ 
 at $\beta =\gamma $, we need to show that $z_{0}\in J_{\gamma }^{+}.$  
 Suppose that 
 $z_{0}\in Y\setminus J_{\gamma }^{+}.$ 
 Let $U$ be the connected component of 
 $Y\setminus J_{\gamma }^{+}$ with 
 $z_{0}\in U.$  Then, by (\ref{eq:dgamman1xp}),  
 Lemma~\ref{l:hypmetcont} and the fact 
 $S_{\tau }\subset F(G_{\tau })$, we obtain that 
 there exist an element $p\in \NN $, a 
 number $\ve >0$, and a compact subset 
 $A_{0}$ of $F(G_{\tau })$  such that 
 $\gamma _{p,1}(B(z_{0},\ve))\subset A_{0}\subset F(G_{\tau }).$ 
 Since $\beta ^{n}\rightarrow \gamma $ as $n\rightarrow \infty $, 
 it follows that there exists a number $N_{1}\in \NN$ 
 such that for each $n\in \NN $ with $n\geq N_{1}$, 
 we have $\beta ^{n}_{p,1}(B(z_{0},\ve))\subset F(G_{\tau }).$ 
 Let $N_{2}\in \NN $ be an element such that 
 for each $q\in \NN $ with $q\geq N_{2}$, $z_{q}\in B(z_{0}, \ve ).$ 
 Let $s=\max \{ N_{1}, N_{2}\}.$ Then,  
 we have $\beta ^{s}_{p,1}(z_{s})\in F(G_{\tau }).$ 
 However, since $z_{s}\in J_{\beta ^{s}}^{+}$, we have 
 $\beta ^{s}_{p,1}(z_{s})\in J(G_{\tau })$, which is a contradiction. 
 Hence, we must have that $z_{0}\in J_{\gamma }^{+}.$ 
 From the above argument, 
 we obtain that the upper semicontinuity 
 of the map 
 $\beta \in (\supp\,\tau )^{\ZZ} \mapsto J_{\beta }^{+}$ 
 at $\beta =\gamma $. 
 
Thus, it follows that for each $\gamma \in {\cal B}_{2}$, 
the map 
$\beta \in (\supp\,\tau)^{\ZZ}\mapsto 
J_{\beta }^{+}$ is continuous at $\beta =\gamma .$ 

Let ${\cal B}={\cal B}_{1}\cap {\cal B}_{2}.$ Then 
${\cal B}$ has the desired property. 
Therefore, we have proved our lemma.  
\end{proof}

\begin{lem}
\label{l:mt9-2pf} 
Under the assumptions of Theorem~\ref{t:mtauspec}, statement~\ref{t:mtauspec9-2} holds.  
\end{lem} 
\begin{proof} 
We now suppose  $\sharp \Min (\tau )\geq 2$. 
Let $L\in \Min(\tau ).$ 
Since 
$T_{L,\tau }: Y\rightarrow [0,1]$ is continuous, and since 
$T_{L,\tau }|_{L}\equiv 1$ and $T_{L,\tau }|_{L'}\equiv 0$ for each 
$L'\in \Min (\tau )$ with $L'\neq L$, 
it follows that $T_{L,\tau }(Y)=[0,1].$ Since 
$T_{L,\tau }$ is continuous on $Y$ and since $T_{L,\tau }\in C_{F(G_{\tau })}(Y)$, 
we obtain that $T_{L,\tau }(J(G_{\tau }))=[0,1].$ In particular, 
$\dim _{\CC }(\LSfk )>1.$  
Thus, we have proved our lemma. 
\end{proof}
\begin{lem}
\label{l:pfmtauspeccfi}
Under the assumptions of Theorem~\ref{t:mtauspec}, 
statement~\ref{t:mtauspeccfi} holds.
\end{lem}
\begin{proof}
Suppose $\dim _{\CC }(W_{\tau })>1$ and int$(J(G_{\tau }))=\emptyset .$ Then, $(W_{\tau })_{nc}\neq \emptyset.$ 
Let $\varphi \in (W_{\tau }) _{nc}.$ 
Then, $\sharp \varphi (Y)>\aleph _{0}.$ 
Since  int$(J(G_{\tau }))=\emptyset $ and 
$\varphi $ is continuous on $Y$, 
we have $\varphi (Y)=\overline{\varphi (F(G_{\tau }))}$ 
in $\CC .$  
Therefore, $\sharp \mbox{Con}(F(G_{\tau }))=\infty .$ 
Thus, we have proved our lemma. 
\end{proof}

\begin{lem}
\label{l:kattfpt1}
Under the assumptions of Theorem~\ref{t:mtauspec}, 
statement~\ref{t:mtauspec10} holds. 
\end{lem}
\begin{proof}
Let $L\in \Min(\tau).$ 
To prove our lemma, it suffices to prove the following claim. \\  
Claim: If $L\subset \Ct$ then 
$L\subset 
\overline{\{ z\in L\mid \exists g\in G_{\tau } \mbox{ s.t. } z \mbox{ is an attracting fixed point of } g\} }.$ 

To prove the above claim, 
let $b\in L.$ 
Let $\epsilon >0.$ By Lemma~\ref{l:hypmetcont}, there exists an element $g_{1}\in G_{\tau }$ 
such that $g_{1}(B(b,\epsilon ))\subset B(g_{1}(b),\frac{\epsilon }{2}).$ 
Since $\overline{G_{\tau }(g_{1}(b))}=L$,  
Lemma~\ref{l:hypmetcont} implies that there exists an element $g_{2}\in G_{\tau }$ 
such that $\overline{g_{2}(B(g_{1}(b),\frac{\epsilon }{2}))}\subset B(b,\epsilon ).$ 
Thus, $\overline{g_{2}g_{1}(B(b,\epsilon ))}\subset B(b,\epsilon ).$ 
Let $g=g_{2}g_{1}.$ 
Then, $z_{0}:= \lim _{n\rightarrow \infty }g^{n}(b)\in B_{h}(b,\epsilon )\cap L$ is an attracting fixed point of 
$g.$ Therefore, we have proved the above claim. 
Thus, we have proved our lemma.
\end{proof}
\begin{lem}
\label{l:pfmtauspec12}
Under the assumptions of Theorem~\ref{t:mtauspec}, 
statement~\ref{t:mtauspec12} holds.
\end{lem}
\begin{proof}
Suppose $\dim _{\CC }(W_{\tau })>1$ and let 
$\varphi \in (W_{\tau }) _{nc}.$ 
Let $A:= \varphi (Y )\setminus \varphi (F(G_{\tau })) .$ 
Since $\varphi \in C_{F(G_{\tau })}(Y)$ and 
since $\sharp \mbox{Con}(F(G_{\tau }))\leq \aleph _{0}$, 
we have $\sharp A>\aleph _{0}.$ 
Moreover, since $\varphi $ is continuous on $Y$, 
it is easy to see that for each $t\in A$, 
$\emptyset \neq \varphi ^{-1}(\{ t\} )\subset J_{res}(G_{\tau }).$ 
Thus, we have proved our lemma.  
\end{proof}

We now prove Theorem~\ref{t:mtauspec}. \\ 

\noindent {\bf Proof of Theorem~\ref{t:mtauspec}:}
Combining Lemmas~\ref{l:minfin}--
\ref{l:pfmtauspecjconti} and Remark~\ref{r:lkncpty}, 
we easily see that all of the statements 1--19 of Theorem~\ref{t:mtauspec} hold. 
\qed 
\begin{rem}
\label{r:sametauthm}
If $\tau \in {\frak M}_{1,c}(X^{-})$ is mean stable 
on $Y'=\Pt \setminus \{ [0:1:0]\}$, then we can prove  
the statements that are similar to those in  
Theorem~\ref{t:mtauspec} hold for $\tau $ 
 using the same arguments above.  
\end{rem}

We now show that if 
$\tau \in {\frak M}_{1,c}(X^{+})$ is mean stable on 
$\Pt \setminus \{ [1:0:0]\}$, then 
the space $\mbox{LS}({\cal U}_{f,\tau }(K))$ in Theorem~\ref{t:mtauspec} is included in the Banach space of 
$\alpha$-\Hol der continuous functions on $K$ for some $\alpha 
\in (0,1)$. 
\begin{df}
Let $K\in \Cpt (\Pt ).$ 
For each $\alpha \in (0,1)$,  
let \\ $C^{\alpha }(K ):= \{ \varphi \in C(K)\mid \sup _{x,y\in K, x\neq y} |\varphi (x)-\varphi (y)|/d(x,y)^{\alpha} <\infty \} $ 
 be the Banach space of all complex-valued $\alpha $-H\"{o}lder continuous functions on $K $
endowed with the $\alpha $-H\"{o}lder norm $\| \cdot \| _{\alpha },$ 
where $\| \varphi \| _{\alpha }:= \sup _{z\in K}| \varphi (z)| +\sup _{x,y\in K,x\neq y}|\varphi (x)-\varphi (y)|/d(x,y)^{\alpha }$ 
for each $\varphi \in C^{\alpha }(K).$ 
\end{df}

\begin{thm}
\label{t:utauca}
Let $\tau \in {\frak M}_{1,c}(X^{+})$. Suppose that 
$\tau $ is mean stable on $Y=\Pt \setminus \{ [1:0:0]\}.$ 
Let $K$ be a compact subset of $Y$ such that 
$G_{\tau }(K)\subset K$ and $\cup _{L\in \Min(\tau)}L\subset 
\mbox{int}(K).$ 
Then, 
there exists an $\alpha _{0}\in (0,1)$ such that for each $\alpha \in (0,\alpha _{0})$, 
$\mbox{{\em LS}}({\cal U}_{f,\tau }(K))\subset C^{\alpha }(K).$ 
Moreover, for each $\alpha \in (0,\alpha _{0})$, 
there exists a constant $E_{\alpha }>0$ such that 
for each $\varphi \in C(K)$, 
$\| \pi _{\tau, K }(\varphi )\| _{\alpha }\leq E_{\alpha }\| \varphi \| _{\infty }.$  
Furthermore, for each $\alpha \in (0,\alpha _{0})$ and for each $L\in \emMin (\tau)$, 
$T_{L,\tau }|_{K}\in C^{\alpha }(K).$ 
\end{thm}
\begin{rem}
\label{r:lrm1mbk}
Let $\tau \in {\frak M}_{1,c}(X^{+})$. Suppose that 
$\tau $ is mean stable on $Y=\Pt \setminus \{ [1:0:0]\}.$ Then,  
by Theorem~\ref{t:mtauspec}, for each 
neighborhood $B$ of $[1:0:0]$, there exists a compact subset $K$ 
of $Y$ such that $\Pt \setminus K\subset B$, 
$G_{\tau }(K)\subset K$ and $\cup _{L\in \Min(\tau)}L\subset 
\mbox{int}(K).$ 
\end{rem}
\noindent {\bf Proof of Theorem~\ref{t:utauca}.}
By 
Theorem~\ref{t:mtauspec}, 
there exists an $r\in \NN $ 
such that 
for each $\varphi \in \mbox{LS}({\cal U}_{f,\tau }(K ))$, 
$M_{\tau }^{r}(\varphi )=\varphi .$ 
Since 
$\tau $ is mean stable on $Y$, 
for each $z\in K$, there exist a map $g_{z}\in G_{\tau }$ and a 
compact  connected neighborhood $\Lambda _{z}$ of $z$ in $Y$ such that 
$g_{z}(\Lambda_{z})\subset F(G_{\tau }).$ 
Since $K$ is compact, there exists a finite family 
$\{ z_{j}\} _{j=1}^{s}$ in $K$ such that 
$\bigcup _{j=1}^{s}\mbox{int}(\Lambda_{z_{j}})\supset K.$ 
Since $G_{\tau }(F(G_{\tau }))\subset F(G_{\tau })$, 
replacing $r$ by a larger number if necessary,  
we may assume that for each $j=1,\ldots ,s$, 
there exists an element $\beta ^{j}=(\beta _{1}^{j},\ldots ,\beta _{r}^{j})\in (\supp\,\tau )^{r}$ 
such that $g_{z_{j}}=\beta _{r}^{j}\circ \cdots \circ \beta _{1}^{j}.$  
For each $j=1,\ldots ,s$, let 
$V_{j}$ be a compact neighborhood of $\beta ^{j}$ in 
$(\supp\,\tau)^{r}$ 
such that for each $\zeta =(\zeta _{1},\ldots ,\zeta _{r})\in V_{j}$, 
$\zeta _{r}\cdots \zeta _{1}(\Lambda_{z_{j}})\subset F(G_{\tau }).$ 
Let $a:=\max \{ \tau ^{r}((\supp\,\tau )^{r}\setminus V_{j})\mid j=1,\ldots ,s\} \in [0,1).$ Here, we set 
$\tau ^{r}=\otimes _{j=1}^{n}\tau .$ 
Let $$C_{1}:= 2\max \{ \sup\{ \mbox{Lip}(\zeta _{r}\circ \cdots \circ \zeta _{1}, K)\mid 
(\zeta _{1},\ldots ,\zeta _{r})\in (\supp\,\tau )^{r} \} ,1\} \geq 2.$$  
Here, for each holomorphic map $h$ on $Y$ we set 
$\mbox{Lip}(h,K):=\sup \{ d(h(z), h(w))/d(z,w)\mid z,w\in K, z\neq w\}.$  
Let $\alpha _{0} \in [0,1)$ be a number such that 
$aC_{1}^{\alpha _{0}}<1.$ Let $0<\alpha <\alpha _{0}.$ Then 
$aC_{1}^{\alpha }<1.$ Let $C_{2}>0$ be a number such that 
for each $z\in K$, there exists a $j\in \{ 1,\ldots ,s\} $ with 
$B(z,C_{2})\cap K\subset \mbox{int}(\Lambda_{z_{j}}).$  
Let $\varphi \in \mbox{LS}({\cal U}_{f,\tau }(K)).$ 
Let $z_{0},z\in K$ be two points. 
If $d(z,z_{0})>C_{1}^{-1}C_{2}$, then 
$$| \varphi (z)-\varphi (z_{0})|/d(z,z_{0})^{\alpha }\leq 
2\| \varphi \| _{\infty }\cdot (C_{1}C_{2}^{-1})^{\alpha }.$$ 
We now suppose that there exists an $n\in \NN $ such that 
$C_{1}^{-n-1}C_{2}\leq d(z,z_{0})\leq C_{1}^{-n}C_{2}.$ 
Then, for each $j\in \NN $ with $1\leq j\leq n$ and 
for each $(\g _{1},\ldots ,\g_{rj})\in (\supp\, \tau )^{rj}$, we have 
$d(\g _{rj}\circ \cdots \circ \g _{1}(z),\g _{rj}\circ \cdots \circ \g _{1}(z_{0}))<C_{2}.$ 
Let $i_{0}\in \{ 1,\ldots ,s\} $ be a number such that 
$B(z_{0},C_{2})\cap K\subset \Lambda_{z_{i_{0}}}.$ 
Let $A(0):= \{ \g =(\g_{j})_{j\in \NN} \in (\supp\,\tau )^{\NN }\mid 
(\g _{1},\ldots ,\g _{r})\in V_{i_{0}}\} $ and 
$B(0):=\{ \g =(\g _{j})_{j\in \NN}\in (\supp\,\tau )^{\NN }\mid 
(\g _{1},\ldots ,\g _{r})\not\in V_{i_{0}}\} .$ 
 Inductively, for each $j=1,\ldots ,n-1$, 
 let $A(j):= \{ \g =(\g _{j})_{j\in \NN}\in B(j-1)\mid \exists i \mbox{ s.t. } 
 B(\g _{rj,1}(z_{0}),C_{2})\subset \Lambda_{z_{i}}, (\g _{rj+1},\ldots ,\g _{rj+r})\in V_{i}\} $ 
 and $B(j):= B(j-1)\setminus A(j).$ 
Then for each $j=1,\ldots ,n-1$, 
$\tau ^{\NN}(B(j))\leq a \tau ^{\NN}(B(j-1)).$ Therefore, 
$\tau ^{\NN}(B(n-1))\leq a^{n}.$ 
Moreover, we have 
$(\supp\,\tau )^{\NN }=\amalg _{j=0}^{n-1}A(j)\amalg B(n-1).$ 
Furthermore, by Theorem~\ref{t:mtauspec}-\ref{t:mtauspec2}, 
$\varphi \in \hat{C}_{F(G_{\tau })}(K).$ 
Thus, we obtain that 
\begin{align*}
\    & |\varphi (z)-\varphi (z_{0})| = |M_{\tau }^{rn}(\varphi )(z)-M_{\tau }^{rn}(\varphi )(z_{0})| \\ 
\leq & |\sum _{j=0}^{n-1}\int _{A(j)}\varphi (\g _{rn,1}(z))-\varphi (\g _{rn,1}(z_{0})) d\tau ^{\NN}(\g )| 
       + |\int _{B(n-1)}\varphi (\g _{rn,1}(z))-\varphi (\g _{rn,1}(z_{0}))
       d\tau ^{\NN}(\g )|\\ 
\leq & \int _{B(n-1)}|\varphi (\g _{rn,1}(z))
-\varphi (\g _{rn,1}(z_{0}))|d\tau ^{\NN}(\g )\\ 
\leq & 2a^{n}\| \varphi \| _{\infty }\leq a^{n}(C_{1}^{n+1}C_{2}^{-1})^{\alpha }d(z,z_{0})^{\alpha }2\| \varphi \| _{\infty }
\leq C_{1}^{\alpha }C_{2}^{-\alpha }2\| \varphi \| _{\infty } d(z,z_{0})^{\alpha }.        
\end{align*}  
From these arguments, it follows that 
$\varphi $ belongs to $C^{\alpha }(K).$ 

Let $\{ \rho _{j}\} _{j=1}^{q}$ be a basis of 
$\mbox{LS}({\cal U}_{f,\tau ,\ast }(K))$ and let 
$\{ \varphi _{j}\} _{j=1}^{q}$ be a basis of 
$\mbox{LS}({\cal U}_{f,\tau }(K))$ such that 
for each $\psi \in C(K)$, 
$\pi _{\tau ,K}(\psi )=\sum _{j=1}^{q}\rho _{j}(\psi )\varphi _{j}$
(see Theorem~\ref{t:mtauspec}). 
Then, for each $\psi \in C(K)$, 
$\| \pi _{\tau ,K}(\psi )\| _{\alpha }\leq \sum _{j=1}^{q}|\rho _{j}(\psi )|\| \varphi _{j}\| _{\alpha }
\leq (\sum _{j=1}^{q}\| \rho _{j}\| _{\infty }\| \varphi _{j}\| _{\alpha })\| \psi \| _{\infty }$, 
where $\| \rho _{j}\| _{\infty }$ denotes the operator norm of 
$\rho _{j}:(C(K ),\| \cdot \| _{\infty })\rightarrow \CC .$ 
 
 We now let $L\in \Min(\tau)$ and let $\alpha \in (0,\alpha _{0})$. 
 By Theorem~\ref{t:mtauspec}-\ref{t:mtauspec9},
 $T_{L,\tau }|_{K}\in \mbox{LS}({\cal U}_{f,\tau }(K)).$ 
 Thus $T_{L,\tau }|_{K}\in C^{\alpha }(K).$ 
 
Thus, we have proved Theorem~\ref{t:utauca}.  
\qed 

We now prove that if $\tau \in {\frak M}_{1,c}(X^{+})$ 
is mean stable on $Y=\Pt \setminus \{ [1:0:0]\}$ and 
$K$ is a compact subset $Y$ such that 
$G_{\tau }(K)\subset K$ and 
$\cup _{L\in \Min(\tau )}L\subset \mbox{int}(K)$, then 
there exists a number $\alpha =\alpha (K)\in (0,1)$ 
such that for each $\varphi\in C^{\alpha }(K)$, 
$ M_{\tau }^{n}(\varphi)$ tends to $\mbox{LS}({\cal U}_{f,\tau }(K))$ 
as $n\rightarrow \infty $ exponentially fast. 
\begin{thm}[Cooperation Principle: Exponential Rate of Convergence] 
\label{t:kjemfhf}
Let $\tau \in {\frak M}_{1,c}(X^{+})$. Suppose that 
$\tau $ is mean stable on $\Pt \setminus \{ [1:0:0]\}.$ 
Let $Y=\Pt \setminus \{ [1:0:0]\}.$ 
Let $r:= \prod _{L\in \emMin(\tau)}\dim _{\CC }(\mbox{{\em LS}}({\cal U}_{f,\tau }(L)))$. (Note that by Theorem~\ref{t:mtauspec}, 
$r$ is finite.)   
Let $K$ be a compact subset $Y$ such that 
$G_{\tau }(K)\subset K$ and 
$\cup _{L\in \Min(\tau )}L\subset \mbox{int}(K).$ 
Then, 
there exist a constant $\alpha =\alpha (K)\in (0,1)$, 
a constant $\lambda =\lambda (K)\in (0,1)$, and a constant $C=C(K)>0$ satisfying that 
for each $\varphi \in C^{\alpha }(K)$, 
we have all of the following.
\begin{itemize}
\item[{\em (1)}] 
$\| M_{\tau }^{nr}(\varphi )-\pi _{\tau ,K}(\varphi )\| _{\alpha }\leq 
C\lambda ^{n}\| \varphi -\pi _{\tau ,K}(\varphi )\| _{\alpha }$ for each $n\in \NN .$ 
\item[{\em (2)}]
$\| M_{\tau }^{n}(\varphi -\pi _{\tau ,K}(\varphi ))\| _{\alpha }\leq C\lambda ^{n}\| \varphi -\pi _{\tau ,K}(\varphi )\| _{\alpha }$
 for each $n\in \NN .$ 
\item[{\em (3)}] 
$\| M_{\tau }^{n}(\varphi -\pi _{\tau ,K}(\varphi ))\| _{\alpha }\leq C\lambda ^{n}\| \varphi \| _{\alpha }$
 for each $n\in \NN .$ 
\item[{\em (4)}] 
$\| \pi _{\tau ,K}(\varphi )\| _{\alpha }\leq C\| \varphi \| _{\alpha }.$ 
\end{itemize}
 
\end{thm}
\begin{rem}
\label{r:lrm1mbk2}
Let $\tau \in {\frak M}_{1,c}(X^{+})$. Suppose that 
$\tau $ is mean stable on $Y=\Pt \setminus \{ [1:0:0]\}.$ Then,  
by Theorem~\ref{t:mtauspec}, for each 
neighborhood $B$ of $[1:0:0]$, there exists a compact subset $K$ 
of $Y$ such that $\Pt \setminus K\subset B$, 
$G_{\tau }(K)\subset K$ and $\cup _{L\in \Min(\tau)}L\subset 
\mbox{int}(K).$ 
\end{rem}

To prove Theorem~\ref{t:kjemfhf}, we need the following  lemma.  
Let $\tau \in {\frak M}_{1,c}(X^{+}).$ 
Suppose 
$\tau $ is mean stable on $Y=\Pt \setminus \{ [1:0:0]\}.$ 
Then,   
all statements in Theorem~\ref{t:mtauspec}
hold for $\tau. $  
Let $L\in \Min(\tau)$ and 
let $r_{L}:=\dim _{\CC }(\mbox{LS}({\cal U}_{f,\tau }(L))).$ 
Using the notation in the proof of Theorem~\ref{t:mtauspec}, 
we have $r_{L}=\sharp (\Min (G_{\tau }^{r_{L}},L)).$   

\begin{lem}
\label{l:suptheta}
Let $\tau \in {\frak M}_{1,c}(X^{+}).$ Suppose $\tau $ is mean 
stable on $Y=\Pt \setminus \{ [1:0:0]\}.$ 
Let $L\in \emMin (\tau)$ and let 
$r_{L}=\dim _{\CC}(\mbox{\em LS}({\cal U}_{f,\tau }(L))).$ 
Let $j\in \{ 1,\ldots, r_{L}\} .$  
Let $A_{1},\ldots, A_{t}$ be the sets 
such that 
$\{ A_{1},\ldots, A_{t}\}
 =\{ A\in \mbox{{\em Con}}(F(G_{\tau }))\mid A\cap L_{j}\neq \emptyset\} $. 
 Let $i=1,\ldots, t$ and  
let $K_{0}$ be a nonempty compact subset of $A_{i}.$ Let $\alpha \in (0,1).$  
Then, there exist a constant $\tilde{C}_{K_{0},\alpha }\geq 1$ 
and a constant $\theta _{\alpha }\in (0,1)$ 
such that 
for each $\varphi \in C^{\alpha }(K_{0})$, for each $z,w\in K_{0}$, and 
for each $n\in \NN $, 
$|M_{\tau }^{n}(\varphi )(z)-M_{\tau }^{n}(\varphi )(w)|\leq \| \varphi \| _{\alpha }
\theta _{\alpha }^{n}\tilde{C}_{K_{0},\alpha}d(z,w)^{\alpha }.$  
\end{lem}
\begin{proof}
Let $K_{0}$ be a compact subset of $A_{i}$. 
Since $\tau $ is mean stable on $Y$, 
there exist a constant $C_{K_{0}}\geq 1$ and a constant 
$\theta \in (0,1)$ such that 
for each $\gamma =(\gamma _{j})_{j\in \NN}\in 
(\supp\,\tau)^{\NN}$, 
for each $z,w\in K_{0}$ and for each $n\in \NN$, 
we have 
$d(\gamma _{n,1}(z), \gamma _{n,1}(w))
\leq C_{K_{0}}\theta ^{n}d(z,w).$ 
Let $\alpha \in (0,1).$  
Let $\varphi \in C^{\alpha }(K_{0})$ and let $z,w\in K_{0}.$ 
Let $n\in \NN .$ 
Then, we obtain 
\begin{align*}
|M_{\tau }^{n}(\varphi )(z)-M_{\tau }^{n}(\varphi )(w)|
& \leq \int _{(\supp\,\tau )^{\NN }}|\varphi (\g _{n,1}(z))-\varphi (\g _{n,1}(w))| d\tau ^{\NN}(\g )\\ 
& \leq \int _{(\supp\,\tau )^{\NN }}\| \varphi \| _{\alpha }d(\g _{n,1}(z),\g _{n,1}(w))^{\alpha }d\tau ^{\NN}(\g )\\ 
& \leq \| \varphi \| _{\alpha }\int _{(\supp\,\tau )^{\NN }}C_{K_{0}}^{\alpha }
\theta ^{\alpha n}d(z,w)^{\alpha } d\tau ^{\NN}(\g )\leq \| \varphi \| _{\alpha }\theta ^{\alpha n}C_{K_{0}}^{\alpha }d(z,w)^{\alpha }. 
\end{align*}
Therefore, the statement of our lemma holds. 
\end{proof}
We now prove Theorem~\ref{t:kjemfhf}.\\ 
\noindent {\bf Proof of Theorem~\ref{t:kjemfhf}:} 
Let $L\in \Min(\tau).$ Let 
$r_{L}:=\dim _{\CC }(\mbox{LS}({\cal U}_{f,\tau }(L))).$ 
Using the notation in the proof of Theorem~\ref{t:mtauspec}, 
let $L_{1},\ldots, L_{r_{L}}$ be the sets such that 
$\{ L_{j} \mid j=1,\ldots, r_{L}\}=\Min (G_{\tau }^{r_{L}},L).$ 
Since $\tau $ is mean stable on $Y$, for each 
$j=1,\ldots, r_{L}$, there exists an open subset  
$H_{j}$ of $Y$ with $L_{j}\subset H_{j}$ 
such that $\overline{H_{j}}$ is a compact subset of 
$F(G_{\tau })$ and such that $G_{\tau }(H_{j})\subset H_{j}.$   
Let $j\in \{ 1,\ldots, r_{L}\}$ and let $A_{1},\ldots, A_{t}$ be the sets such that 
$\{ A_{i} \mid i=1,\ldots, t\}=\{ A\in \mbox{Con}(F(G_{\tau }))\mid A\cap L_{j}\neq \emptyset \} .$ 
We may assume that $H_{j}\subset \cup _{i=1}^{t}A_{i}.$  
Let $H_{j,i}:= H_{j}\cap A_{i}$ and $L_{j,i}:=L_{j}\cap A_{i}.$ 
By Lemma~\ref{l:suptheta}, 
there exist a family $\{ D_{0,\alpha }\} _{\alpha \in (0,1)} $ of 
positive constants, a family $\{ D_{1,\alpha }\} _{\alpha \in (0,1)}$ of positive constants,
 and  
a family $\{ \lambda _{1,\alpha }\} _{\alpha \in (0,1)}$ 
of elements in $ (0,1)$ 
such that for each $\alpha \in (0,1)$, 
for each $L\in \Min(\tau)$, 
for each $i$, for each $j$, 
for each $\g =(\g _{j})_{j\in \NN} \in (\supp\,\tau )^{\NN }$, 
for each $z,w\in H_{j,i}$, for each $n\in \NN $ and for each $\varphi \in C^{\alpha }(K)$, 
\begin{equation}
\label{eq:d0d1}
|M_{\tau }^{n}(\varphi )(z)-M_{\tau }^{n}(\varphi )(w)|\leq D_{0,\alpha }\lambda _{1,\alpha }^{n}
\| \varphi \| _{\alpha } d(z,w)^{\alpha }\leq D_{1,\alpha }\lambda _{1,\alpha }^{n}\| \varphi \| _{\alpha }.
\end{equation}
For each subset $\Gamma $ of $\Pt $ and for each bounded function $\psi :\Gamma\rightarrow \CC $, 
we set $\| \psi \| _{\Gamma}:=\sup _{z\in \Gamma }|\psi (z)|.$ 
For each $i=1,\ldots ,t$, let $x_{i}\in L_{j,i}$ be a point. 
Let $\varphi \in C^{\alpha }(K)$. By (\ref{eq:d0d1}), we obtain   
$$\sup _{z\in H_{j,i}}|M_{\tau }^{nr_{L}}(\varphi )(z)-M_{\tau }^{nr_{L}}(\varphi )(x_{i})|\leq 
D_{1,\alpha }\lambda _{1,\alpha }^{n}\| \varphi \| _{\alpha }$$ for each $i, j, n.$  
Therefore, for each $j$ and for each $l,n \in \NN $, 
\begin{equation}
\label{eq:d1l1a}
\| M_{\tau }^{lr_{L}}(M_{\tau }^{nr_{L}}(\varphi )-\sum _{i=1}^{t}M_{\tau }^{nr_{L}}(\varphi )(x_{i})\cdot 1_{H_{j,i}})\| _{H_{j}}
\leq D_{1,\alpha }\lambda _{1,\alpha }^{n}\| \varphi \| _{\alpha }. 
\end{equation}
We now consider $M_{\tau }^{r_{L}}:\hat{C}_{F(G_{\tau })}(H_{j})\rightarrow \hat{C}_{F(G_{\tau })}(H_{j}).$ 
We have $\dim _{\CC }(\hat{C}_{F(G_{\tau })}(H_{j}))<\infty .$ 
Moreover, by the argument in the proof of 
Lemma~\ref{l:suppllj}, 
$M_{\tau }^{r_{L}}:\hat{C}_{F(G_{\tau })}(H_{j})\rightarrow 
\hat{C}_{F(G_{\tau })}(H_{j})$ has 
exactly one unitary eigenvalue $1$, and has exactly one unitary eigenvector 
$1_{H_{j}}.$ Therefore, there exist a constant $\lambda _{2}\in (0,1)$ 
and a constant $D_{2}>0$, each of which depends only on $\tau $ 
and does not depend on $\alpha $ and $\varphi $, such that 
for each $l\in \NN $, 
\begin{align}
\label{eq:d2la2}
\ & \| M_{\tau }^{lr_{L}}(\sum _{i=1}^{t}M_{\tau }^{nr_{L}}(\varphi )(x_{i})1_{H_{j,i}})
-\lim _{m\rightarrow \infty }M_{\tau }^{mr_{L}}(\sum _{i=1}^{t}
M_{\tau }^{nr_{L}}(\varphi )(x_{i})1_{H_{j,i}})\| _{H_{j}} \notag \\  
\leq & D_{2}\lambda _{2}^{l}\| \sum _{i=1}^{t}M_{\tau }^{nr_{L}}(\varphi )(x_{i})1_{H_{j,i}}\| _{H_{j}}
\leq D_{2}\lambda _{2}^{l}t\| \varphi \| _{K }.
\end{align} 
Since $\lambda _{2}$ does not depend on $\alpha $, we may assume that 
for each $\alpha \in (0,1)$, 
$\lambda _{1,\alpha }\geq \lambda _{2}.$ 
From (\ref{eq:d1l1a}) and (\ref{eq:d2la2}), 
it follows that for each $n\in \NN $ and for each $l_{1},l_{2}\in \NN $ with $l_{1},l_{2}\geq n$, 
\begin{align*}
\    & \| M_{\tau }^{(l_{1}+n)r_{L}}(\varphi )-M_{\tau }^{(l_{2}+n)r_{L}}(\varphi )\| _{H_{j}}\\ 
\leq & \| M_{\tau }^{(l_{1}+n)r_{L}}(\varphi )-M_{\tau }^{l_{1}r_{L}}(\sum _{i=1}^{t}
       M_{\tau }^{nr_{L}}(\varphi )(x_{i})1_{H_{j,i}})\| _{H_{j}}\\ 
\    & \ + \| M_{\tau }^{l_{1}r_{L}}(\sum _{i=1}^{t}M_{\tau }^{nr_{L}}(\varphi )(x_{i})1_{H_{j,i}})
         -\lim _{m\rightarrow \infty }M_{\tau }^{mr_{L}}(\sum _{i=1}^{t}M_{\tau }^{nr_{L}}(\varphi )(x_{i})1_{H_{j,i}})\| _{H_{j}}\\ 
\    & \ + \| \lim _{m\rightarrow \infty }M_{\tau }^{mr_{L}}(\sum _{i=1}^{t}M_{\tau }^{nr_{L}}(\varphi )(x_{i})1_{H_{j,i}})
      - M_{\tau }^{l_{2}r_{L}}(\sum _{i=1}^{t}M_{\tau }^{nr_{L}}(\varphi )(x_{i})1_{H_{j,i}})\| _{H_{j}}\\ 
\    & \ +\| M_{\tau }^{l_{2}r_{L}}(\sum _{i=1}^{t}M_{\tau }^{nr_{L}}(\varphi )(x_{i})1_{H_{j,i}})
      -M_{\tau }^{(l_{2}+n)r_{L}}(\varphi )\| _{H_{j}}\\ 
\leq & 2D_{1,\alpha }\lambda _{1,\alpha }^{n}\| \varphi \| _{\alpha }+D_{2}\lambda _{2}^{l_{1}}t\| \varphi \| _{\alpha }
      +D_{2}\lambda _{2}^{l_{2}}t\| \varphi \| _{\alpha } 
      \leq (2D_{1,\alpha }+2D_{2}t)\lambda _{1,\alpha }^{n}\| \varphi \| _{\alpha }.      
\end{align*} 
Letting $l_{1}\rightarrow \infty $, we obtain that 
for each $l_{2}\in \NN $ with $l_{2}\geq n$, 
$\| \pi _{\tau ,K}(\varphi )-M_{\tau }^{(l_{2}+n)r_{L}}(\varphi )\| _{H_{j}}\leq 
(2D_{1,\alpha }+2D_{2}t)\lambda _{1,\alpha }^{n}\| \varphi \| _{\alpha }.$ 
In particular, for each $n\in \NN $, 
$\| \pi _{\tau ,K}(\varphi )-M_{\tau }^{2nr_{L}}(\varphi )\| _{H_{j}}\leq 
(2D_{1,\alpha }+2D_{2}t)\lambda _{1,\alpha }^{n}\| \varphi \| _{\alpha }.$ 
Therefore,  for each $n\in \NN $, 
\begin{equation}
\label{eq:pitauhj}
\| \pi _{\tau ,K}(\varphi )-M_{\tau }^{nr_{L}}(\varphi )\| _{H_{j}}\leq 
(2D_{1,\alpha }+2D_{2}t)\lambda _{1,\alpha }^{-1/2}(\lambda _{1,\alpha }^{1/2})^{n}\| M_{\tau }^{r_{L}}\| _{\alpha }\| \varphi \| _{\alpha }, 
\end{equation}  
where $\| M_{\tau }^{r_{L}}\| _{\alpha }$ denotes the operator norm of 
$M_{\tau }^{r_{L}}:C^{\alpha }(K)\rightarrow C^{\alpha }(K).$ 
Let $U:=\bigcup _{L,j}H_{j}$ and let $r:=\prod _{L}r_{L}.$ 
From the above arguments, it follows that there exist a family $\{ D_{3,\alpha }\} _{\alpha \in (0,1)} $ 
of positive constants and a family 
$\{ \lambda _{3,\alpha }\} _{\alpha \in (0,1)} $ of elements in 
$ (0,1)$ 
such that for each $\alpha \in (0,1)$, for each 
$\varphi \in C^{\alpha }(K)$ and for each 
$n\in \NN $, 
\begin{equation}
\label{eq:UleqD3}
\| \pi _{\tau ,K}(\varphi )-M_{\tau }^{rn}(\varphi )\| _{\overline{U}}\leq 
D_{3,\alpha }\lambda _{3,\alpha }^{n}\| \varphi \| _{\alpha }. 
\end{equation} 
Since $\tau $ is mean stable on $Y$, 
for each $z\in K$, there exist a map $g_{z}\in G_{\tau }$ and a compact connected neighborhood 
$\Lambda_{z}$ of $z$ in $Y$ such that 
$g_{z}(\Lambda_{z})\subset U.$ 
Since $K$ is compact, there  exists a finite family 
$\{ z_{j}\} _{j=1}^{s}$ in $K$ such that 
$\bigcup _{j=1}^{s}\mbox{int}(\Lambda_{z_{j}})\supset K.$ 
Since $G_{\tau }(U)\subset U$, we may assume that 
there exists a $k$ such that 
for each $j=1,\ldots ,s$, there exists an element $\beta ^{j}=(\beta _{1}^{j},\ldots ,\beta _{k}^{j})\in (\supp\,\tau)^{k}$ 
with $g_{z_{j}}=\beta _{k}^{j}\circ \cdots \circ \beta _{1}^{j}.$ 
We may also assume that $r|k.$ 
For each $j=1,\ldots ,s$, let $V_{j}$ be a compact neighborhood of $\beta ^{j}$ in $(\supp\,\tau )^{k}$ such that 
for each $\zeta =(\zeta _{1},\ldots ,\zeta _{k})\in V_{j}$, 
$\zeta _{k}\circ \cdots \circ \zeta _{1}(\Lambda_{z_{j}})\subset U.$ Let 
$a:=\max \{ \tau ^{k}((\supp\,\tau)^{k}\setminus V_{j})\mid j=1,\ldots ,s\} \in [0,1)$, where $\tau^{k}=\otimes _{j=1}^{k}\tau .$  
Let $C_{1}:=2\max \{ \max \{ \mbox{Lip}(\zeta _{k}\circ \cdots \circ \zeta _{1}, K) \mid 
(\zeta _{1},\ldots ,\zeta _{k})\in (\supp\,\tau )^{k},z\in K \} ,1\} .$ 
Here, for each holomorphic map $h$ on $Y$, 
we set $\mbox{Lip}(h, K)=\sup \{ d(h(z), h(w))/d(z,w)\mid z, w\in K, z\neq w\} .$ 
Let $\alpha _{1}\in (0,1)$ be such that 
$aC_{1}^{\alpha _{1}}<1$ and $\mbox{LS}({\cal U}_{f,\tau }(K ))\subset C^{\alpha _{1}}(K )$ (see Theorem~\ref{t:utauca}). 
 Let $C_{2}>0$ be a constant such that 
 for each $z\in K $ there exists a $j\in \{ 1,\ldots ,s\} $ with  
$B(z,C_{2})\subset \Lambda_{z_{j}}.$ 
Let $n\in \NN .$ 
Let $z_{0}\in K$ be any point. 
Let $i_{0}\in \{ 1,\ldots ,s\} $ be such that 
$B(z_{0},C_{2})\in \Lambda_{z_{i_{0}}}.$ 
Let $A(0):=\{ \g =(\g _{j})\in (\supp\,\tau )^{\NN }\mid (\g _{1},\ldots ,\g _{k})\in V_{i_{0}}\} $ 
and let $B(0):= \{ \g=(\g_{j})_{j\in \NN}\in (\supp\,\tau )^{\NN }\mid (\g _{1},\ldots ,\g _{k})\not\in V_{i_{0}}\} .$ 
Inductively, for each $j=1,\ldots ,n-1$, let 
$A(j):=\{ \g =(\g_{j})_{j\in \NN}\in B(j-1)\mid \exists i \mbox{ s.t. }B(\g _{kj,1}(z_{0}),C_{2})\subset \Lambda_{z_{i}}, 
(\g _{kj+1},\ldots ,\g _{kj+k})\in V_{i}\} $ and let 
$B(j):= B(j-1)\setminus A(j).$ Then, for each $j=1,\ldots ,n-1$, 
$\tau ^{\NN}(B(j))\leq a \tau ^{\NN}(B(j-1))\leq \cdots \leq a^{j+1}$ and 
$\tau ^{\NN}(A(j))\leq \tau ^{\NN}(B(j-1))\leq a^{j}.$  
Moreover, we have 
$(\supp\,\tau )^{\NN }=\amalg _{j=0}^{n-1}A(j)\amalg B(n-1).$ 
Therefore, we obtain that 
\begin{align}
\label{eq:mtauknz0}
\    & |M_{\tau }^{kn}(\varphi )(z_{0})-\pi _{\tau ,K}(\varphi )(z_{0})|=
       |M_{\tau }^{kn}(\varphi )(z_{0})-M_{\tau }^{kn}(\pi _{\tau ,K}(\varphi ))(z_{0})| \notag \\ 
\leq & \left|\sum _{j=0}^{n-1}\int _{A(j)}(\varphi (\g _{kn,1}(z_{0}))
        -\pi _{\tau ,K}(\varphi )(\g _{kn,1}(z_{0})))d\tau ^{\NN}(\g )\right| \notag \\ 
\    & \ +\left|\int _{B(n-1)}(\varphi (\g _{kn,1}(z_{0}))
        -\pi _{\tau ,K}(\varphi )(\g _{kn,1}(z_{0})))d\tau ^{\NN}(\g )\right|. 
\end{align}
For each $j=0,\ldots ,n-1$, there exists a Borel subset $A'(j)$ of 
$(\supp\,\tau )^{k(j+1)}$ such that 
$A(j)=A'(j)\times \supp\,\tau \times \supp\,\tau \times \cdots .$ 
Hence, by (\ref{eq:UleqD3}), we obtain that for each $\alpha \in (0,1)$ 
and for each $\varphi \in C^{\alpha }(K)$, 
\begin{align}
\label{eq:apjph}
\    & \left| \int _{A(j)}(\varphi (\g _{kn,1}(z_{0}))-\pi _{\tau ,K}(\varphi )(\g _{kn,1}(z_{0}))) 
       d\tau ^{\NN}(\g )\right| \notag \\ 
= & \left|  \int _{A'(j)}(M_{\tau }^{k(n-j-1)}(\varphi )(\g _{k(j+1)}\circ \cdots \circ \g _{1}(z_{0}))
              -\pi _{\tau ,K}(\varphi )(\g _{k(j+1)}\circ \cdots \circ \g _{1}(z_{0}))) d\tau (\g _{k(j+1)})\cdots 
              d\tau (\g _{1})\right| \notag \\ 
\leq & D_{3,\alpha }\lambda _{3,\alpha }^{n-j-1}\| \varphi \| _{\alpha }\tau ^{\NN}(A(j))
\leq D_{3,\alpha }\lambda _{3,\alpha }^{n-j-1}a^{j}\| \varphi \|_{\alpha }.    
\end{align}
By (\ref{eq:mtauknz0}) and (\ref{eq:apjph}), 
it follows that 
\begin{align*}
\    & |M_{\tau }^{kn}(\varphi )(z_{0})-\pi _{\tau ,K}(\varphi )(z_{0})|
\leq  \sum _{j=0}^{n-1}D_{3,\alpha }\lambda _{3,\alpha }^{n-j-1}a^{j}\| \varphi \| _{\alpha }
        +a^{n}(\| \varphi \| _{\infty }+\| \pi _{\tau ,K}(\varphi )\| _{\infty })\\ 
\leq & \left( D_{3,\alpha }n(\max \{ \lambda _{3,\alpha }, a\} )^{n-1} 
       +a^{n}(1+\| \pi _{\tau ,K}\| _{\infty } )\right) \| \varphi \| _{\alpha },  
\end{align*} 
where $\| \pi _{\tau ,K}\| _{\infty }$ denotes the operator norm of 
$\pi _{\tau ,K}:(C(K),\| \cdot \| _{\infty })\rightarrow (C(K),\| \cdot \| _{\infty }).$ 
For each $\alpha \in (0,1)$, let 
$\zeta _{\alpha }:=\frac{1}{2}(1+\max \{ \lambda _{3,\alpha },a\} )<1.$ 
From these arguments, it follows that there exists a family $\{ C_{3,\alpha }\} _{\alpha \in (0,1)}$ 
of positive constants such that for each $\alpha \in (0,1)$,  
for each $\varphi \in C^{\alpha }(K)$ and for each $n\in \NN $,  
\begin{equation}
\label{eq:mtauC3}
\| M_{\tau }^{kn}(\varphi )-\pi _{\tau }(\varphi )\| _{\infty }
\leq C_{3,\alpha }\zeta _{\alpha }^{n}\| \varphi \| _{\alpha }. 
\end{equation} 
For the rest of the proof, let 
$\alpha \in (0,\alpha _{1}).$ Let $\eta _{\alpha }:=
\max \{ \lambda _{1,\alpha },aC_{1}^{\alpha }\} \in (0,1).$ 
Let $z,z_{0}\in K$ and let $\varphi \in C^{\alpha }(K).$  
If $d(z,z_{0})\geq C_{1}^{-1}C_{2}$, then 
\begin{equation}
\frac{| M_{\tau }^{kn}(\varphi )(z)-M_{\tau }^{kn}(\varphi )(z_{0})
-(\pi _{\tau }(\varphi )(z)-\pi _{\tau }(\varphi )(z_{0}))|}{d(z,z_{0})^{\alpha }}
\leq 2C_{3,\alpha }\zeta _{\alpha }^{n}\| \varphi \| _{\alpha } (C_{1}C_{2}^{-1})^{\alpha }.
\end{equation}
We now suppose that there exists an $m\in \NN $ such that 
$C_{1}^{-m-1}C_{2}\leq d(z,z_{0})<C_{1}^{-m}C_{2}.$ 
Then, for each $\g =(\g_{j})_{j\in \NN}\in (\supp\,\tau )^{\NN }$ and for 
each $j=1,\ldots ,m$, 
\begin{equation}
\label{eq:dgkjC2}
d(\g _{kj,1}(z),\g _{kj,1}(z_{0}))<C_{2}. 
\end{equation}  
Let $n\in \NN $. Let $\tilde{m}:=\min \{ n,m\} .$ 
Let $i_{0}\in  \{ 1,\ldots ,s\} $ be such that $B(z_{0},C_{2})\subset \Lambda_{z_{i_{0}}}$ and 
let $A(0),B(0),\ldots, A(\tilde{m}-1),B(\tilde{m}-1)$ be as before. 
Then,  
we have 
\begin{align}
\label{eq:a0am}
\    & |M_{\tau }^{kn}(\varphi )(z)-M_{\tau }^{kn}(\varphi )(z_{0})
      -(\pi _{\tau ,K}(\varphi )(z)-\pi _{\tau ,K}(\varphi )(z_{0}))|\notag \\ 
\leq & \left| \sum _{j=0}^{\tilde{m}-1}\int _{A(j)}[\varphi (\g _{kn,1}(z))-\varphi (\g _{kn,1}(z_{0}))
        -(\pi _{\tau ,K}(\varphi )(\g _{kn,1}(z))-\pi _{\tau ,K}(\varphi )(\g _{kn,1}(z_{0})))]
        d\tau ^{\NN}(\g )\right| \notag \\ 
\    & \ + \left| \int _{B(\tilde{m}-1)}[\varphi (\g _{kn,1}(z))-\varphi (\g _{kn,1}(z_{0}))
        -(\pi _{\tau ,K}(\varphi )(\g _{kn,1}(z))-\pi _{\tau ,K}(\varphi )(\g _{kn,1}(z_{0})))]
        d\tau ^{\NN}(\g )\right| . 
\end{align} 
Let $A'(j)$ be as before. 
By (\ref{eq:d0d1}), (\ref{eq:dgkjC2}), and 
the fact 
``$\mbox{LS}({\cal U}_{f,\tau }(K))\subset \hat{C}_{F(G_{\tau })}(K)$'' 
(Theorem~\ref{t:mtauspec}-\ref{t:mtauspec2} 
),  
we obtain that for each $j=0,\ldots ,\tilde{m}-1$, 
\begin{align}
\label{eq:ajineq}
\ & \left| \int _{A(j)}[\varphi (\g _{kn,1}(z))-\varphi (\g _{kn,1}(z_{0}))
        -(\pi _{\tau ,K}(\varphi )(\g _{kn,1}(z))-\pi _{\tau ,K}(\varphi )(\g _{kn,1}(z_{0})))]
        d\tau ^{\NN}(\g )\right|\notag \\
= & \left| \int _{A(j)}(\varphi (\g _{kn,1}(z))-\varphi (\g _{kn,1}(z_{0}))
 d\tau ^{\NN}(\g )\right| 
    \notag \\ 
= & \left| \int _{A'(j)}[M_{\tau }^{k(n-j-1)}(\varphi )(\g _{k(j+1),1}(z))-
           M_{\tau }^{k(n-j-1)}(\varphi )(\g _{k(j+1),1}(z_{0}))] 
           d\tau (\g _{k(j+1)})\cdots d\tau (\g _{1})\right| \notag \\ 
\leq & \int _{A'(j)}D_{0,\alpha }d(\g _{k(j+1),1}(z),\g _{k(j+1),1}(z_{0}))^{\alpha }
       \lambda _{1,\alpha }^{n-j-1} \| \varphi \| _{\alpha } 
       d\tau (\g _{k(j+1)})\cdots d\tau (\g _{1}) \notag \\ 
\leq & D_{0,\alpha }C_{1}^{\alpha (j+1)}d(z,z_{0})^{\alpha }
       \lambda _{1,\alpha }^{n-j-1}a^{j}\| \varphi \| _{\alpha } 
\leq  D_{0,\alpha }C_{1}^{\alpha }\eta _{\alpha }^{n-1}\| \varphi \| _{\alpha }
      d(z,z_{0})^{\alpha }.   
\end{align}
Let $B'(\tilde{m}-1)$ be a Borel subset of 
$(\supp\,\tau )^{k\tilde{m}}$ such that 
$B(\tilde{m}-1)=B'(\tilde{m}-1)\times 
\supp\,\tau \times \supp\,\tau \times \cdots .$ 
We now consider the following two cases. Case (I): $\tilde{m}=m$. Case (II): 
$\tilde{m}=n.$ 

Suppose we have Case (I). 
Then, by (\ref{eq:mtauC3}), we obtain that  
\begin{align}
\label{eq:Bm-1c1}
\    & \left| \int _{B(\tilde{m}-1)}[\varphi (\g _{kn,1}(z))-\varphi (\g _{kn,1}(z_{0}))
        -(\pi _{\tau ,K}(\varphi )(\g _{kn,1}(z))-\pi _{\tau ,K}(\varphi )(\g _{kn,1}(z_{0})))]
        d\tau ^{\NN}(\g )\right| \notag \\ 
\leq & \int _{B'(m-1)}| M_{\tau }^{k(n-m)}(\varphi )(\g _{km}\circ \cdots \circ \g_{1}(z))
        -\pi _{\tau ,K}(\varphi )(\g _{km}\circ \cdots \circ \g_{1}(z))| 
       d\tau (\g _{km})\cdots d\tau (\g _{1})\notag \\
\    & \ +\int _{B'(m-1)}  | M_{\tau }^{k(n-m)}(\varphi )(\g _{km}\circ \cdots \circ \g_{1}(z_{0}))
        -\pi _{\tau ,K}(\varphi )(\g _{km}\circ \cdots \circ \g_{1}(z_{0}))| 
       d\tau (\g _{km})\cdots d\tau (\g _{1})\notag \\
\leq & 2C_{3,\alpha }\zeta _{\alpha }^{n-m}\| \varphi \| _{\alpha }a^{m} 
       \leq  2C_{3,\alpha }\zeta _{\alpha }^{n-m}\| \varphi \| _{\alpha }a^{m}
       \cdot (C_{1}^{m+1}C_{2}^{-1}d(z,z_{0}))^{\alpha }\notag \\ 
=    & 2C_{3,\alpha }\zeta _{\alpha }^{n-m}(aC_{1}^{\alpha })^{m}(C_{1}C_{2}^{-1})^{\alpha }
       \| \varphi \| _{\alpha } d(z,z_{0})^{\alpha }  
       \leq  2C_{3,\alpha }(C_{1}C_{2}^{-1})^{\alpha }\zeta _{\alpha }^{n-m}\eta _{\alpha }^{m}
       \| \varphi \| _{\alpha }d(z,z_{0})^{\alpha }.          
\end{align}
We now suppose we have Case (II). 
Since $\mbox{LS}({\cal U}_{f,\tau }(K ))\subset C^{\alpha }(K )$, 
we obtain  
\begin{align}
\label{eq:Bm-1c2}
\    &   \left| \int _{B(\tilde{m}-1)}[\varphi (\g _{kn,1}(z))-\varphi (\g _{kn,1}(z_{0}))
        -(\pi _{\tau ,K}(\varphi )(\g _{kn,1}(z))-\pi _{\tau ,K}(\varphi )(\g _{kn,1}(z_{0})))]
        d\tau ^{\NN}(\g )\right| \notag \\ 
\leq & \int _{B(n-1)}| \varphi (\g _{kn,1}(z))-\varphi (\g _{kn,1}(z_{0}))| 
d\tau ^{\NN}(\g ) 
        + \int _{B(n-1)}| \pi _{\tau ,K}(\varphi )(\g _{kn,1}(z))
                      -\pi _{\tau ,K}(\varphi )(\g _{kn,1}(z_{0}))| d\tau ^{\NN}(\g )\notag \\ 
\leq & C_{1}^{\alpha n}d(z,z_{0})^{\alpha }a^{n}\| \varphi \| _{\alpha }
       +C_{1}^{\alpha n}d(z,z_{0})^{\alpha }a^{n}\| \pi _{\tau ,K}(\varphi )\| _{\alpha }\notag \\ 
\leq & C_{1}^{\alpha n}a^{n}(1+E_{\alpha })\| \varphi \| _{\alpha }d(z,z_{0})^{\alpha },          
\end{align}
where $E_{\alpha }$ denotes the number in Theorem~\ref{t:utauca}. 
Let $\xi _{\alpha }:= \frac{1}{2}(\max \{ \zeta _{\alpha },\eta _{\alpha }\} +1) \in (0,1).$ 
Combining (\ref{eq:a0am}), (\ref{eq:ajineq}), (\ref{eq:Bm-1c1}), and (\ref{eq:Bm-1c2}), 
it follows that there exists a constant $C_{4,\alpha }>0$ such that 
for each $\varphi \in C^{\alpha }(K)$, 
for each $n\in \NN$, and for each $z,z_{0}\in K$,  
\begin{equation}
\label{eq:mtauval}
 |M_{\tau }^{kn}(\varphi )(z)-M_{\tau }^{kn}(\varphi )(z_{0})
-(\pi _{\tau , K}(\varphi )(z)-\pi _{\tau , K}(\varphi )(z_{0}))| 
\leq C_{4,\alpha }\xi _{\alpha }^{n}\| \varphi \| _{\alpha }d(z,z_{0})^{\alpha }.
\end{equation}  
Let $C_{5,\alpha }=C_{3,\alpha }+C_{4,\alpha }$. 
By (\ref{eq:mtauC3}) and (\ref{eq:mtauval}), we obtain that 
for each $\varphi \in C^{\alpha }(K )$ and for each $n\in \NN $, 
\begin{equation}
\label{eq:mtauC5}
\| M_{\tau }^{kn}(\varphi )-\pi _{\tau ,K}(\varphi )\| _{\alpha }
\leq C_{5,\alpha }\xi _{\alpha }^{n}\| \varphi \| _{\alpha }. 
\end{equation}  
Let $\lambda =\xi _{\alpha }$ and 
$C=C_{5,\alpha }\cdot \max\{ \| M_{\tau }^{i}\| _{\alpha }\mid 
i=0,\ldots, k-1\}$, where $M_{\tau }^{0}$ denotes the identity map 
on $C^{\alpha }(K).$ 
From (\ref{eq:mtauC5}), statement (3) of our theorem holds. 

Let $\psi \in C^{\alpha }(K ).$ Setting 
$\varphi =\psi -\pi _{\tau ,K}(\psi )$, 
by (\ref{eq:mtauC5}), we obtain that statement (2) of our theorem holds. 
Statement (4) of our theorem follows from Theorem~\ref{t:utauca} 
replacing $C$ by another constant if necessary. 
Statement (1) follows from statement (2).  

Thus, we have proved Theorem~\ref{t:kjemfhf}. 
\qed 

\ 

We now consider the spectrum Spec$_{\alpha , K}(M_{\tau })$ of $M_{\tau }:C^{\alpha }(K) \rightarrow C^{\alpha }(K).$ 
By Theorem~\ref{t:utauca}, 
${\cal U}_{v,\tau }(K)\subset \mbox{Spec}_{\alpha ,K}(M_{\tau })$ for some $\alpha \in (0,1).$ 
From Theorem~\ref{t:kjemfhf}, we can show that 
the distance between ${\cal U}_{v,\tau }(K)$ and 
$\mbox{Spec}_{\alpha ,K}(M_{\tau })\setminus {\cal U}_{v,\tau }(K)$ is positive. 
\begin{thm}
\label{t:kjemfsp}
Under the assumptions of Theorem~\ref{t:kjemfhf}, 
we have all of the following.
\begin{itemize}
\item[{\em (1)}]
$\mbox{{\em Spec}}_{\alpha ,K}(M_{\tau })\subset \{ z\in \CC \mid |z|\leq \lambda \} 
\cup {\cal U}_{v,\tau }(K)$, 
where $\alpha =\alpha (K)\in (0,1)$ and $\lambda =\lambda (K)\in (0,1)$ are the constants in Theorem~\ref{t:kjemfhf}. 
\item[{\em (2)}] 
Let $\zeta \in \CC \setminus (\{ z\in \CC \mid |z|\leq \lambda \} 
\cup {\cal U}_{v,\tau }(K))$. 
Then, $(\zeta I-M_{\tau })^{-1}:C^{\alpha }(K)\rightarrow C^{\alpha }(K)$ 
is equal to 
$(\zeta I-M_{\tau })|_{\emLSfk}^{-1}\circ \pi _{\tau ,K}+\sum _{n=0}^{\infty }
\frac{M_{\tau }^{n}}{\zeta ^{n+1}}(I-\pi _{\tau ,K}),$  
where $I$ denotes the identity on $C^{\alpha }(K).$ 

\end{itemize}

\end{thm}
\begin{rem}
\label{r:lrm1mbk3}
Let $\tau \in {\frak M}_{1,c}(X^{+})$. Suppose that 
$\tau $ is mean stable on $Y=\Pt \setminus \{ [1:0:0]\}.$ Then,  
by Theorem~\ref{t:mtauspec}, for each 
neighborhood $B$ of $[1:0:0]$, there exists a compact subset $K$ 
of $Y$ such that $\Pt \setminus K\subset B$, 
$G_{\tau }(K)\subset K$ and $\cup _{L\in \Min(\tau)}L\subset 
\mbox{int}(K).$ 
\end{rem}

\noindent {\bf Proof of Theorem~\ref{t:kjemfsp}:} 
Let $A:= \{ z\in \CC \mid | z |\leq \lambda \} \cup 
{\cal U}_{v,\tau }(K ).$ 
Let $\zeta \in \CC \setminus A.$ 
Then, by Theorem~\ref{t:kjemfhf}, 
 $\sum _{n=0}^{\infty }
\frac{M_{\tau }^{n}}{\zeta ^{n+1}}(I-\pi _{\tau ,K})$ 
converges in the space of bounded linear operators on $C^{\alpha }(K)$ endowed 
with the operator norm. 
Let $\Omega := (\zeta I-M_{\tau })|_{\LSfk}^{-1}\circ \pi _{\tau ,K}+\sum _{n=0}^{\infty }
\frac{M_{\tau }^{n}}{\zeta ^{n+1}}(I-\pi _{\tau ,K}).$ 
Let $W_{\tau ,K}:= \mbox{LS}({\cal U}_{f,\tau }(K)).$ 
Then, we have 
\begin{align*}
(\zeta I-M_{\tau })\circ \Omega  
= &  \left( (\zeta I-M_{\tau })|_{W_{\tau ,K}}\circ \pi _{\tau ,K}
    +(\zeta I-M_{\tau })|_{{\cal B}_{0,\tau ,K}}
    \circ (I-\pi _{\tau ,K})\right) \\ 
\ & \ \circ \left( (\zeta I-M_{\tau })|_{W_{\tau ,K}}^{-1}\circ 
  \pi _{\tau ,K}
     + \sum _{n=0}^{\infty }\frac{M_{\tau }^{n}}{\zeta ^{n+1}}|_{{\cal B}_{0,\tau ,K}}\circ 
    (I-\pi _{\tau ,K}) \right) \\ 
= & I|_{W_{\tau ,K}}\circ \pi _{\tau ,K}+(\zeta I-M_{\tau })\circ 
    (\sum _{n=0}^{\infty }\frac{M_{\tau }^{n}}{\zeta ^{n+1}})\circ (I-\pi _{\tau ,K})\\ 
= & \pi _{\tau ,K}+(\sum _{n=0}^{\infty }\frac{M_{\tau }^{n}}{\zeta ^{n}}-
    \sum _{n=0}^{\infty }\frac{M_{\tau }^{n+1}}{\zeta ^{n+1}})\circ (I-\pi _{\tau ,K})=I.      
\end{align*}
Similarly, we have $\Omega \circ (\zeta I-M_{\tau })=I.$ 
Therefore, statements (1) and (2) of our theorem hold. 

Thus, we have proved Theorem~\ref{t:kjemfsp}.  
\qed 

\ 

Combining Theorem~\ref{t:kjemfsp} and perturbation theory for linear operators (\cite{K}), 
we obtain the following Theorem~\ref{t:kjemfsppt}. In particular,
we obtain complex-two-dimensional analogues of the Takagi  function.  
For the explanation on the Takagi function 
(a famous example of continuous function that is nowhere differentiable on $[0,1]$)
and 
complex-one-dimensional analogues of it, see 
\cite{Splms10, Sadv}. 
Note that by Theorem~\ref{t:yniceaod}, it is easy to see that 
if we denote by ${\cal A}$ the set of elements $\tau \in {\frak M}_{1,c}(X^{+})$ 
satisfying that $\sharp \supp\,\tau <\infty $ and 
$\tau\mbox{ is mean stable on } \Pt \setminus 
\{ [1:0:0]\}$, then ${\cal A}$  
is dense in ${\frak M}_{1,c}(X^{+})$ with respect to the wH-topology ${\cal O}$
(see also Remark~\ref{r:msfinsupp}).  
\begin{thm}
\label{t:kjemfsppt}
Let $m\in \NN $ with $m\geq 2.$
Let $h_{1},\ldots ,h_{m}\in X^{+}$. 
Let $G$ be the semigroup generated 
by $h_{1},\ldots, h_{m}$, 
i.e., $G=\{ h_{i_{n}}\circ \cdots \circ h_{i_{1}}\mid 
n\in \NN, i_{j}\in \{ 1,\ldots, m\} (\forall j)\} .$  
Let ${\cal W}_{m}:= \{ (a_{1},\ldots ,a_{m})\in (0,1)^{m}\mid \sum _{j=1}^{m}a_{j}=1 \} 
\cong \{ (a_{1},\ldots ,a_{m-1})\in (0,1)^{m-1}\mid \sum _{j=1}^{m-1}a_{j}<1 \}.$ 
For each $a=(a_{1},\ldots ,a_{m})\in {\cal W}_{m}$, 
let $\tau _{a}:= \sum _{j=1}^{m}a_{j}\delta _{h_{j}}\in {\frak M}_{1,c}(X^{+}).$ 
Suppose that there exists an element $c\in {\cal W}_{m}$
such that $\tau _{c}$ is mean stable on $\Pt \setminus \{ [1:0:0]\}$ 
{\em (}Remark: then for each $a\in {\cal W}_{m}$, $\tau _{a}$ is mean stable 
on $\Pt \setminus \{ [1:0:0]\}${\em )}.  
Let $Y=\Pt \setminus \{ [1:0:0]\}.$ 
Then,  
for each compact subset $K$ of $Y$ 
such that $G(K)\subset K$ and 
$\cup _{L\in \Min(G,Y)}L\subset \mbox{int}(K)$, 
we have the following statements {\em (1), (2)} and {\em (3)}. 
\begin{itemize}
\item[{\em (1)}] 
For each $b\in {\cal W}_{m}$   
there exist an $\alpha =\alpha (b, K)\in (0,1)$ and an open neighborhood $V_{b}$ of $b$ in ${\cal W}_{m}$ 
such that for each $a\in V_{b}$, we have 
$\mbox{{\em LS}}({\cal U}_{f,\tau _{a}}(K))\subset C^{\alpha }(K)$, 
$\pi _{\tau _{a},K}(C^{\alpha }(K))\subset C^{\alpha }(K)$ and 
$(\pi _{\tau _{a},K}:C^{\alpha }(K)\rightarrow C^{\alpha }(K))\in L(C^{\alpha }(K))$, 
where $L(C^{\alpha }(K))$ denotes the Banach space of bounded linear operators on $C^{\alpha }(K)$ 
endowed with the operator norm,  and such that the map 
$a\mapsto (\pi _{\tau _{a},K}:C^{\alpha }(K)\rightarrow C^{\alpha }(K))\in 
L(C^{\alpha }(K))$ 
 is real analytic in $V_{b}.$ 
\item[{\em (2)}] 
Let $L\in \emMin( G, Y).$ 
Then, for each $b\in {\cal W}_{m}$, there exists an $\alpha =\alpha (b,K)\in (0,1)$ such that 
the map $a\mapsto T_{L,\tau _{a}}|_{K}\in (C^{\alpha }(K),\| \cdot \| _{\alpha })$ is real analytic 
in an open neighborhood of $b$ in ${\cal W}_{m}.$ Moreover, 
the map $a\mapsto T_{L,\tau _{a}}|_{K}\in (C(K),\| \cdot \| _{\infty })$ is real analytic in 
${\cal W}_{m}.$ In particular, for each 
$z\in Y $, the map $a\mapsto T_{L,\tau _{a}}(z)$ is real analytic in ${\cal W}_{m}.$ 
Furthermore, for any multi-index $n=(n_{1},\ldots ,n_{m-1})\in (\NN \cup \{ 0\})^{m-1}$ and for any 
$b\in {\cal W}_{m},$  
 the function 
 $z\mapsto [(\frac{\partial }{\partial a_{1}})^{n_{1}}\cdots (\frac{\partial }{\partial a_{m-1}})^{n_{m-1}}
(T_{L,\tau _{a}}(z))]|_{a=b}$ belongs to $C_{F(G)}(Y).$ 
\item[{\em (3)}] 
Let $L\in \emMin(G, Y)$ and let $b\in {\cal W}_{m}.$ 
For each $i=1,\ldots ,m-1$ and for each $z\in Y$, let 
$\psi _{i,b}(z):=[\frac{\partial }{\partial a_{i}}(T_{L,\tau _{a}}(z))]|_{a=b}$  
and let $\zeta _{i,b}(z):= T_{L,\tau _{b}}(h_{i}(z))-T_{L,\tau _{b}}(h_{m}(z)).$  
Then,  
$\psi _{i,b}|_{K}$ is the unique solution of 
the functional equation $(I-M_{\tau _{b}})(\psi )=\zeta _{i,b}|_{K}, \psi |_{S_{\tau _{b}}}=0, \psi \in C(K)$,  
where $I$ denotes the identity map. Moreover, 
there exists a number $\alpha =\alpha (b,K)\in (0,1)$ such that  
$\psi _{i,b}|_{K}=\sum _{n=0}^{\infty }M_{\tau _{b}}^{n}(\zeta _{i,b}|_{K})$ in $(C^{\alpha }(K),\| \cdot \| _{\alpha }).$  
\end{itemize}  
\end{thm}
\begin{rem}
\label{r:lrm1mbk3_2}
Under the assumptions of Theorem~\ref{t:kjemfsppt}, 
by Theorem~\ref{t:mtauspec}, for each 
neighborhood $B$ of $[1:0:0]$, there exists a compact subset $K$ 
of $Y$ such that $\Pt \setminus K\subset B$, 
$G(K)\subset K$ and $\cup _{L\in \Min(G, Y)}L\subset 
\mbox{int}(K).$ 
\end{rem}

\noindent {\bf Proof of Theorem~\ref{t:kjemfsppt}:} 
Under the assumptions of Theorem~\ref{t:kjemfsppt}, 
by using the method in the proofs of \cite[Lemmas 5.1, 5.2]{SU1}, 
we obtain that for each $\alpha \in (0,1)$,  
the map  
$a\in {\cal W}_{m} \mapsto M_{\tau _{a}}\in 
L(C^{\alpha }(K ))$ is real-analytic, where $L(C^{\alpha }(K)) $ denotes the 
Banach space of bounded linear operators on $C^{\alpha }(K)$ endowed with 
the operator norm. 
Moreover, by using the method in the proof of Theorem~\ref{t:utauca}, 
we can show that for each $b\in {\cal W}_{m}$, 
there exist  an $\alpha \in (0,1)$ and an open neighborhood $V_{b}$ of $b$ in ${\cal W}_{m}$ 
such that for each $a\in V_{b}$, we have $\mbox{LS}({\cal U}_{f,\tau _{a}}(K ))\subset C^{\alpha }(K).$ 
In particular, $\pi _{\tau _{a},K}(C^{\alpha }(K))\subset C^{\alpha }(K )$ for each $a\in V_{b}.$   
Statement (1) follows from the above arguments, 
Theorem~\ref{t:mtauspec}, 
Theorem~\ref{t:kjemfsp}, and \cite[p368-369, p212]{K}. 

We now prove statement (2). 
For each $L\in \Min(G, Y)=\Min(\tau _{c})$, 
let $\varphi _{L}:Y \rightarrow [0,1]$ be a $C^{\infty }$ function on $Y$ such that  
$\varphi _{L}|_{L}\equiv 1$ and such that for each 
$L'\in \Min (G, Y)$ with $L'\neq L$, 
$\varphi _{L}|_{L'}\equiv 0.$   
Then, 
by Theorem~\ref{t:mtauspec}-\ref{t:mtauspec4}, 
we have that for each $z\in K$, 
$T_{L,\tau _{a}}(z)=\lim _{n\rightarrow \infty }M_{\tau _{a}}^{n}(\varphi _{L})(z).$ 
Moreover, since $\tau $ is mean stable on $Y$, 
we have $J_{\ker }(G_{\tau }|_{K})=\emptyset .$ 
Combining these with \cite[Lemma 4.2, Proposition 4.7]{Splms10}, 
we obtain $T_{L,\tau _{a}}=\lim _{n\rightarrow \infty }M_{\tau _{a}}^{n}(\varphi )$ 
in $C(K).$ By 
Theorem~\ref{t:mtauspec}--\ref{t:mtauspec3}, 
\ref{t:mtauspec5}, \ref{t:mtauspec6}, 
for each $a\in {\cal W}_{m}$, 
there exists a number $r\in \NN $ such that 
for each $\psi \in \mbox{LS}({\cal U}_{f,\tau }(K ))$, 
$M_{\tau _{a}}^{r}(\psi )=\psi .$ Therefore, 
by Theorem~\ref{t:mtauspec}-\ref{t:mtauspec2}, 
$T_{L,\tau _{a}}|_{K}=\lim _{n\rightarrow \infty }M_{\tau _{a}}^{nr}(\varphi _{L}|_{K})
=\lim _{n\rightarrow \infty }M_{\tau _{a}}^{nr}
(\varphi _{L}|_{K}-\pi _{\tau _{a},K}(\varphi _{L}|_{K})
+\pi _{\tau _{a},K}(\varphi _{L}|_{K}))=\pi _{\tau _{a},K}(\varphi _{L}|_{K}).$ 
Combining this with statement (1) of our theorem and 
Theorem~\ref{t:mtauspec}-\ref{t:mtauspec2}, 
it is easy to see that statement (2) of our theorem holds. 

 We now prove statement (3). 
By taking the partial derivative of $M_{\tau _{a}}(T_{L,\tau _{a}})(z)=T_{L,\tau _{a}}(z)$ 
with respect to $a_{i}$, it is easy to see that 
$\psi _{i,b}$ satisfies the functional equation 
$(I-M_{\tau _{b}})(\psi _{i,b})=\zeta _{i,b}, \psi _{i,b}|_{S_{\tau _{b}}}=0.$ 
Let $\psi \in C(K)$ be a solution of 
$(I-M_{\tau _{b}})(\psi )=\zeta _{i,b}|_{K}, \psi |_{S_{\tau _{b}}}=0.$ 
Then, for each $n\in \NN $, 
\begin{equation}
\label{eq:I-M}
(I-M_{\tau _{b}}^{n})(\psi )=\sum _{j=0}^{n-1}M_{\tau _{b}}^{j}
(\zeta _{i,b}|_{K}).
\end{equation}  
By the definition of $\zeta _{i,b}$,  
$\zeta _{i,b}|_{S_{\tau _{b}}}=0.$ Therefore, by 
Theorem~\ref{t:mtauspec}-\ref{t:mtauspec2-1}(d), 
$\pi _{\tau _{b},K}(\zeta _{i,b}|_{K})=0.$ 
Thus, denoting by $C=C(K)$ and $\lambda =\lambda (K)$ the constants in Theorem~\ref{t:kjemfhf}, 
we obtain $\| M_{\tau _{b}}^{n}(\zeta _{i,b}|_{K})\| _{\alpha }\leq C\lambda ^{n}\| \zeta _{i,b}|_{K}\| _{\alpha }$ in $C^{\alpha }(K).$  
Moreover, since $\psi |_{S_{\tau _{b}}}=0$, 
Theorem~\ref{t:mtauspec}-\ref{t:mtauspec2-1}(d) 
implies that 
$\pi _{\tau _{b},K}(\psi )=0.$ Therefore, $M_{\tau _{b}}^{n}(\psi )\rightarrow 0$  
in $C(K)$ as $n\rightarrow \infty .$ Letting $n\rightarrow \infty $ in (\ref{eq:I-M}), 
we obtain that 
$\psi =\sum _{j=0}^{\infty }M_{\tau _{b}}^{j}(\zeta _{i,b}|_{K}).$ 
Therefore, we have proved statement (3). 

  Thus, we have proved Theorem~\ref{t:kjemfsppt}. 
\qed  
\begin{rem}
\label{r:xnegsimth}
If $\tau \in {\frak M}_{1,c}(X^{-})$ is mean stable 
on $Y'=\Pt \setminus \{ 0:1:0]\}$, then we can 
show the results similar to Theorems~\ref{t:utauca}, 
\ref{t:kjemfhf}, \ref{t:kjemfsp}, and \ref{t:kjemfsppt}
 using the methods in the proofs of these theorems.  
\end{rem}

We now prove Theorem~\ref{t:rpmms1}. 

\noindent {\bf Proof of Theorem~\ref{t:rpmms1}. }
Let 
${\cal A}^{+}:= \{ \tau \in {\frak M}_{1,c}(X^{+}) \mid 
\tau \mbox{ is mean stable on }\Pt \setminus 
\{ [1:0:0]\}\} $ and 
${\cal A}^{-}:= \{ \tau \in {\frak M}_{1,c}(X^{+}) \mid 
\tau^{-1} \mbox{ is mean stable on }\Pt \setminus 
\{ [0:1:0]\}\} .$ 
 By Theorem~\ref{t:yniceaod}, 
${\cal A}^{+}$ and ${\cal A}^{-}$
are open and dense in ${\frak M}_{1,c}(X^{+})$ with respect 
to the wH-topology ${\cal O}$ in  ${\frak M}_{1,c}(X^{+})$. 
Thus, ${\cal MS}={\cal A}^{+}\cap {\cal A}^{-}$ is open and 
dense in ${\frak M}_{1,c}(X^{+})$ with respect to ${\cal O}$.  

We now let $\tau \in {\cal MS}.$ 
By Theorems~\ref{t:mtauspec}, \ref{t:kjemfhf}, 
\ref{t:kjemfsp} and Remark~\ref{r:sametauthm}, statements (1)--(6) hold for $\tau $. 
Hence, we have proved Theorem~\ref{t:rpmms1}. 
\qed 

\begin{ex}
\label{e:periodn}
For each $n\in \NN$, there exist an element $\tau \in {\cal MS}$ and 
an element $L\in \Min (\tau )$ with $L\subset \CC $ such that 
the number $r_{L}$ in Theorem~\ref{t:mtauspec} is equal to $n.$ 
For, let $p(y)$ be a polynomial of degree two or more having an attracting 
periodic cycle  in $\CC $ of period $n$. Let $\delta \in \CC \setminus \{ 0\} $ be a number  
such that $|\delta |$ is sufficiently small. Let 
$f(x,y)=(y, p(y)-\delta x).$ Then $f\in X^{+}$ and 
$f$ has an attracting periodic cycle $E$ in $\Ct $ of period $n.$ 
By  Theorem~\ref{t:rpmms1}, there exists an element $\tau \in {\cal MS}$ 
arbitrarily close to $\delta _{f}$ with respect to the wH-topology. 
Then  there  exists an element $L\in \Min(\tau )$ with $E\subset L\subset \Ct$ 
such that $r_{L}=n$ provided that $\tau $ is close enough to $\delta _{f}.$   
\end{ex}

We now prove Theorem~\ref{t:rpmms2}. 

\noindent {\bf Proof of Theorem~\ref{t:rpmms2}.} 
Let 
${\cal A}^{+}_{1}:= \{ \tau \in {\frak M}_{1,c}(X_{1}^{+}) \mid 
\tau \mbox{ is mean stable on }\Pt \setminus 
\{ [1:0:0]\}\} $ and 
${\cal A}^{-}_{1}:= \{ \tau \in {\frak M}_{1,c}(X_{1}^{+}) \mid 
\tau^{-1} \mbox{ is mean stable on }\Pt \setminus 
\{ [0:1:0]\}\} .$ 
 By Theorem~\ref{t:yniceaod}, 
${\cal A}^{+}_{1}$ and ${\cal A}^{-}_{1}$
are open and dense in ${\frak M}_{1,c}(X_{1}^{+})$ with respect 
to the wH-topology ${\cal O}$ in  ${\frak M}_{1,c}(X_{1}^{+})$. 
Thus, ${\cal MS}\cap {\frak M}_{1,c}(X_{1}^{+}) 
={\cal A}_{1}^{+}\cap {\cal A}_{1}^{-}$ is open and 
dense in ${\frak M}_{1,c}(X_{1}^{+})$ with respect to the 
wH-topology ${\cal O}$ in ${\frak M}_{1,c}(X_{1}^{+}).$

Let $\tau \in {\cal MS}\cap {\frak M}_{1,c}(X_{1}^{+}).$ 
By Lemma~\ref{l:gjcaek}(6), 
there exists no attracting minimal set $L$ of $\tau $ 
with $L\subset \Ct $ and 
 there exists no attracting minimal set $L$ of $\tau ^{-1}$ 
with $L\subset \Ct $. 
Combining this with Theorem~\ref{t:mtauspec} and 
Lemma~\ref{l:taumsema}, 
we can easily see that statements (1) and (2) in Theorem~\ref{t:rpmms2} 
hold. Thus, we have proved Theorem~\ref{t:rpmms2}. 
 \qed  
 
 \ 
 
\noindent {\bf Acknowledgements.} 
This
research was partially supported by JSPS Kakenhi 
grant numbers 
19H01790 and 
24K00526. English of this paper has been checked by  professional English proofreading companies.  \\ 
{\bf Statements and Declarations.} The author declares no conflicts of interest associated with this manuscript. 
Data sharing is not applicable to this article as no datasets were generated or analyzed during the current study.

\end{document}